\documentclass{article}
\usepackage{amsmath,amssymb}
\usepackage[dvipdfmx]{graphicx}
\usepackage{amsthm}
\usepackage{enumerate}
\usepackage{cite}
\usepackage{bm}
\usepackage{cases}
\usepackage{subfig}
\theoremstyle{plain}% Theorem-like structures provided by amsthm.sty
\newtheorem{theorem}{Theorem}
\newtheorem{lemma}{Lemma}
\newtheorem{corollary}{Corollary}
\newtheorem{proposition}{Proposition}
\newtheorem{claim}{Claim}

\theoremstyle{definition}
\newtheorem{definition}{Definition}
\newtheorem{example}{Example}
\newtheorem{fact}{Fact}
\newtheorem{problem}{Problem}
\newtheorem{remark}{Remark}

\theoremstyle{remark}

\newcommand{\bbR}{\mathbb{R}}

\newcommand{\bbN}{\mathbb{N}}

\newcommand{\Hlb}{\mathcal{H}}
\newcommand{\calX}{\mathcal{X}}
\newcommand{\calY}{\mathcal{Y}}
\newcommand{\calZ}{\mathcal{Z}}

\newcommand{\calR}{\mathcal{H}}

\newcommand{\calS}{\mathcal{S}}

\newcommand{\bmv}{{{v}}}
\newcommand{\bmw}{{{w}}}
\newcommand{\bmx}{{{x}}}
\newcommand{\bmy}{{{y}}}
\newcommand{\bmz}{{{z}}}
\newcommand{\bmzero}{{{0}}}
\newcommand{\bmxi}{{\xi}}
\newcommand{\bmzeta}{{\zeta}}
\newcommand{\bmeta}{{\eta}}

\newcommand{\frakL}{{\mathfrak L}}

\newcommand{\numI}{{I}}
\newcommand{\numII}{{I\hspace{-.1em}I}}
\newcommand{\numIII}{{I\hspace{-.1em}I\hspace{-.1em}I}}
\newcommand{\numIV}{{I\hspace{-.1em}V}}

\newcommand{\rmO}{{\rm O}}
\newcommand{\rmI}{{\rm I}}
\newcommand{\rmId}{{\rm Id}}
\newcommand{\Tlcp}{{T_{\rm LiGME}}}

\newcommand{\rank}{\mathop{\rm rank}\nolimits}
\newcommand{\ran}{\mathop{\rm ran}\nolimits}
\newcommand{\Null}{\mathop{\rm null}\nolimits}

\newcommand{\dom}{\mathop{\rm dom}\nolimits}

\newcommand{\gra}{\mathop{\rm gra}\nolimits}
\newcommand{\Prox}{\mathop{\rm Prox}\nolimits}

\newcommand{\Fix}{\mathop{\rm Fix}\nolimits}

\newcommand{\argmin}{\mathop{\rm arg\,min}\limits}

\newcommand{\minimize}{\mathop{\rm minimize}\limits}

\newcommand{\sfT}{{*}}

\newcommand{\FI}{{\Rightarrow}}
\newcommand{\IFF}{\Leftrightarrow}

\makeatletter
\newcommand\fleqnoff{\@fleqnfalse\@mathmargin\@centering}
\newcommand\fleqnon[1][\leftmargini]{\@fleqntrue\@mathmargin=#1\relax
	\@ifundefined{mathindent}{\let\mathindent\@mathmargin}{}}
\makeatother

\makeatletter
\newcommand{\opnorm}{\@ifstar\@opnorms\@opnorm}
\newcommand{\@opnorms}[1]{%
  \left|\mkern-1.5mu\left|\mkern-1.5mu\left|
   #1
  \right|\mkern-1.5mu\right|\mkern-1.5mu\right|
}
\newcommand{\@opnorm}[2][]{%
  \mathopen{#1|\mkern-1.5mu#1|\mkern-1.5mu#1|}
  #2
  \mathclose{#1|\mkern-1.5mu#1|\mkern-1.5mu#1|}
}
\makeatother

\title{Linearly Involved Generalized Moreau Enhanced Models and Their Proximal Splitting Algorithm under Overall Convexity Condition}

\author{Jiro Abe, Masao Yamagishi, and Isao Yamada \\
  Tokyo Institute of Technology\\[1mm]
  abe@sp.ce.titech.ac.jp, \{myamagi, isao\}@ict.e.titech.ac.jp
}

\begin{document}
\maketitle

\begin{abstract}
The convex envelopes of the direct discrete measures, for the sparsity of vectors or for  the low-rankness of matrices, have been utilized extensively as practical penalties 
in order to compute a globally optimal solution of the corresponding regularized least-squares models.  
Motivated mainly by the ideas in [Zhang'10, Selesnick'17, Yin, Parekh, Selesnick'19] to exploit nonconvex penalties in the regularized least-squares models without losing their overall convexities,  this paper presents {\em the Linearly involved Generalized Moreau Enhanced (LiGME) model} as a unified extension of such utilizations of nonconvex penalties.  The proposed model can admit multiple nonconvex  penalties without losing its overall convexity and thus is applicable to much broader scenarios in the sparsity-rank-aware signal processing. 
Under the general overall-convexity condition of  the LiGME model,  we also present a novel proximal splitting type algorithm of guaranteed convergence to a globally optimal solution.  Numerical experiments in  typical examples of  the sparsity-rank-aware signal processing demonstrate the effectiveness of the LiGME models and the proposed proximal splitting algorithm.
\end{abstract}

%\pacs{02.30.Zz, 20.00, 42.10}
%\noindent{\it Keywords}: convex optimization, nonconvex penalty, generalized minimax concave penalty function, linearly involved convexity-preserving model,  proximal splitting, nonexpansive operator, Krasnosel'ski{\u{\i}}-Mann iteration, signal recovery

%\submitto{\IP}


\section{Introduction}
Many tasks in inverse problems for data sciences and engineerings (see, e.g., \cite{nashed76:_gener_inver_applic,groetsch1993inverse,Bertero-Boccacci98,nashed02:_inver_probl_image_analy_medic_imagin,Ben-Israel-Greville03,byrne2007applied,Elad2010-sparse,Starck-Murtagh-Fadili15,Theodoridis15} and references therein), including signal processing and machine learning, 
have been studied as estimations of an unknown vector $x^{\star}\in {\cal X}$ from the observed data $y \in {\cal Y}$ that follows the 
linear regression model:  
\begin{equation}
y=Ax^{\star}+{\varepsilon}
\label{intro:lin-regression}, 
\end{equation}
where $({\cal X}, \langle \cdot,\cdot\rangle_{\cal X}, \|\cdot\|_{\cal X})$ and  
$({\cal Y}, \langle \cdot,\cdot\rangle_{\cal Y}, \|\cdot\|_{\cal Y})$ are 
finite dimensional real Hilbert spaces, $A:{\cal X}\rightarrow {\cal Y}$ is a known bounded linear operator and ${\varepsilon}\in {\cal Y}$ is an unknown noise vector.  
A common approach for such  estimation problems is to solve  
the {\sl regularized least-squares} minimization problem:  
\begin{equation}
\minimize_{x \in {\cal X}} \  J_{\Psi \circ {\mathfrak L}}(x) := 
\frac{1}{2} \|y - Ax\|_{\cal Y}^2 + \mu \Psi \circ {\mathfrak L}(x), \ \ \mu > 0,
\label{intro:involveL}
\end{equation}
where $\frac{1}{2}\|y - Ax \|_{\cal Y}^2$ is the least-squares term that measures the distance between $y$ and $A x$,  
$\Psi \circ {\mathfrak L}$ is a regularizer (or a penalty) 
designed  strategically, e.g., based on a prior knowledge on $x^{\star}\in {\cal X}$,  to obtain its better estimate as a minimizer of $J_{\Psi \circ L}$ with a certain 
real Hilbert space $({\cal Z}, \langle \cdot,\cdot\rangle_{\cal Z}, \|\cdot\|_{\cal Z})$,  
a certain bounded linear operator  ${\mathfrak L}:{\cal X}\rightarrow {\cal Z}$,  
a certain function $\Psi: {\cal Z}\rightarrow (-\infty,\infty]$ and  a regularization parameter $\mu>0$ providing the trade-off between the lest-squares term and the regularizer. To study optimization algorithms for  \eqref{intro:involveL}  
with general $\Psi$ which is not necessarily differentiable at every $x\in {\cal X}$,   
the decoupled expression of  $\Psi$ and ${\mathfrak L}$ in \eqref{intro:involveL} is   
very crucial even if $\Psi$ is convex because we usually need many nontrivial ideas to deal with  $\Psi$ and ${\mathfrak L}$ separately. 
Design of $(\Psi, {\mathfrak L}, \mu)$ depends on applications as well as mathematical tractability for the optimization task. Typical examples are found as follows.  

\begin{example}\label{ex-regularizations}
	\begin{enumerate}
		\item[(a)] ({\sl Ridge regression or Tikhonov type 
                    regularization}) %By letting $\Psi(\cdot)=\|\cdot\|_{\cal X}^{2}$,
                   By letting $\Psi(\cdot)=\|\cdot\|_{\cal Z}^2$ and ${\mathfrak L}={\rm Id}$, 
                  the problem (\ref{intro:involveL}) reproduces a classical regularization known as the {\sl ridge regression estimator} \cite{Horel62,Horel-Kennard70}, essentially based on common idea of the so-called 
		Tikhonov type regularization \cite{Tikhonov63,Tikhonov77} which has been extensively studied and extended \cite{Hansen93,Hanke-Hansen93,Bertero-Boccacci98,Golub-Hansen-OLeary99,Ben-Israel-Greville03}. 
		\item[(b)] ({\sl $\ell_{1}$ regularization}) By letting ${\cal X}={\cal Z}:={\mathbb R}^{n}$, $\Psi(\cdot)=\|\cdot\|_{1}$ ($\ell_{1}$-norm)  and ${\mathfrak L}={\rm Id}$,  the problem (\ref{intro:involveL}) reproduces the  $\ell_{1}$ regularization problem which has been a standard model in applications demanding sparse estimates 
                  $x=(x_1,\ldots, x_n)\in {\cal X}$ of $x^{\star}$. For example, in a classification task based on $n$ features corresponding to the components of $x^{\star}$,  not all features are informative, hence we want to keep the most informative components and make the less informative ones equal to zero.   Since the naive approach by choosing $\Psi(x)=\|x\|_{0}$,
                  where $\|x\|_{0}$  stands for the number of nonzero components of $x$,  makes  the problem (\ref{intro:involveL}) 
		in general NP-hard, its convex envelope $\Psi(x)=\|x\|_{1}:=\sum_{i=1}^{n}|x_{i}|$ has been utilized in many applications.  Although this type of regularizations appeared in 70s at the latest in seismology, e.g.,    \cite{Claerbout-Muir73,Taylor-Bank-McCoy79,Santosa-Symes86}, it has attracted an intensive revived interest in statistics  \cite{Tibshirani1996-lasso}, which addressed the LASSO (Least Absolute Shrinkage and Selection Operator) task,  as well as in signal processing and machine learning, in particular in compressed sensing \cite{Candes-Romberg-Tao06,Donoho2006-compressed} and related sparsity aware applications \cite{Elad2010-sparse,Theodoridis15}.  
		\item[(c)] ({\sl Linearly involved $\ell_{p}$ regularization / Wavelet-based regularization  / Total-Variation based  regularization}) By letting 
		${\cal X}={\mathbb R}^{n}$, ${\cal Z}={\mathbb R}^{l}$, 
		$\Psi(z)=(\|z\|_{p})^{p}
		:=\left(\sqrt[p]{\sum_{i=1}^{l}|z_{i}|^{p}}\right)^{p}$ ($p\geq 1$) for 
		$z:=(z_1,\ldots, z_l)\in \bbR^{l}$,  the problem (\ref{intro:involveL}) reproduces the linearly involved $\ell_{p}$ regularizations. For example, setting 
${\mathfrak L}=W$, where $W$ is a wavelet transform matrix,  the problem (\ref{intro:involveL}) reproduces the  so-called wavelet-based regularization, e.g., in  \cite{Daubechies-Defrise-Mol06,Starck-Murtagh-Fadili15}. If we set $\Psi(\cdot)=\|\cdot\|_{1}$ and ${\mathfrak L}=D$, where $D$ is the first order differential  operator (see (\ref{def:D})),  the problem (\ref{intro:involveL}) reproduces the  so-called convex Total Variation (TV) regularization \cite{rudin1992nonlinear}.  The choices of $\Psi(\cdot)=(\|\cdot\|_{p})^{p}$ ($1\leq p<2$), in such applications,  have been preferred to $p=2$ because smaller $p$ is more effective than $p=2$ in order to promote the sparsity of ${\mathfrak L}(x)$ in (\ref{intro:involveL}) 
and also because the choice  $0\leq p<1$ looses the convexity of the function $\Psi$, which usually makes it very hard to find a global minimizer of $J_{\Psi\circ {\mathfrak L}}$.  
The great success of  the model $J_{\Psi\circ {\mathfrak L}}$ with $\Psi(\cdot)=(\|\cdot\|_{p})^{p}$ ($1\leq p<2$)  especially for large scale applications has been achieved 
by the modern  computational techniques,   e.g., proximal splitting \cite{Beck-Teboulle10,Combettes-Pesquet2011,Yamada-Yukawa-Yamagishi2011}  in convex analysis \cite{Rockafellar70,Ekeland-Temam99,Rockafellar2009-variational,Boyd-Vandenberghe04,Bauschke2011-convex,Combettes2015-Compositions}.
\item[(d)]  ({\em Regularized least-squares with multiple penalties})  
Thanks to the remarkable expressive ability of the abstract Hilbert space, the simple form of the regularized least-squares minimization problem in (\ref{intro:involveL}) is very flexible. 
For example,  by letting ${\cal X}={\mathbb R}^{m\times n}\times
{\mathbb R}^{m\times n}=
\left\{
{\mathbf z}=(z_{1}, z_{2})\mid z_{i}\in {\mathbb R}^{m\times n}\ (i=1,2)\right\}$ 
equipped with the addition ${\cal X}\times {\cal X}\rightarrow 
{\cal X}: ({\mathbf x},{\mathbf y})\mapsto 
(x_{1}+y_{1}, x_{2}+y_{2})$, the scalar multiplication 
${\mathbb R}\times {\cal X}\rightarrow {\cal X}:(\alpha, {\mathbf z})\mapsto (\alpha z_{1}, \alpha z_{2})$, and the inner product 
$\langle \cdot, \cdot\rangle_{\cal X}:({\mathbf x},{\mathbf y})\mapsto 
{\rm tr}(x_{1}^{\top}y_{1})+{\rm tr}(x_{2}^{\top}y_{2})
$, we can use  \eqref{intro:involveL} for estimation of a pair of matrices. Moreover,  the form (\ref{intro:involveL})  covers 
seemingly much more general case: 
\begin{equation}
\minimize_{x \in {\cal X}} \  J_{\Psi \circ L}(x) := 
\frac{1}{2} \|y - Ax\|_{\cal Y}^2 + \sum_{i=1}^{\cal M}\mu_{i} \Psi^{\langle i\rangle} \circ {\mathfrak L}_{i}(x), 
\label{intro:multiple-involveL}
\end{equation}
 where multiple penalties are employed in terms of 
 real Hilbert spaces 
 $({\cal Z}_{i}, \langle\cdot,\cdot\rangle_{{\cal Z}_{i}}, \|\cdot\|_{{\cal Z}_{i}})$,   functions $\Psi^{\langle i\rangle}:{\cal Z}_{i}\rightarrow (-\infty, \infty]$,  bounded linear operators 
 ${\mathfrak L}_{i}:{\cal X}\rightarrow {\cal Z}_{i}$ and weights 
 $\mu_{i}>0$
 $(i=1,\ldots, {\cal M})$.    
 This fact can be understood through the following simple translation 
 (see, e.g., \cite{Yamada2017-SVM,Yamada-Yukawa-Yamagishi2011,gandy2011tensor,combettes2008proximal,pierra76:_method,pierra1984decomposition})
 of 
 \eqref{intro:multiple-involveL}  into the form  (\ref{intro:involveL}) by redefining a new Hilbert space  
\begin{equation}\label{prod-space0}
{\cal Z}:={\cal Z}_{1}\times \cdots \times {\cal Z}_{\cal M}=
\left\{
{\mathbf z}=(z_{1},\ldots, z_{\cal M})\mid z_{i}\in {\cal Z}_{i}\ (i=1,\ldots, {\cal M})\right\}
\end{equation}  
equipped with the addition ${\cal Z}\times {\cal Z}\rightarrow 
{\cal Z}: ({\mathbf x},{\mathbf y})\mapsto 
(x_{1}+y_{1},\ldots, x_{\cal M}+y_{\cal M})$, the scalar multiplication 
${\mathbb R}\times {\cal Z}\rightarrow {\cal Z}:(\alpha, {\mathbf z})\mapsto (\alpha z_{1},\ldots, \alpha z_{\cal M})$, and the inner product $({\mathbf x},{\mathbf y})\mapsto 
\langle {\mathbf x},{\mathbf y}\rangle_{\cal Z}:=\sum_{i=1}^{\cal M}\langle x_{i},y_{i}\rangle_{{\cal Z}_{i}}$, and by introducing a new function 
\begin{equation}\label{prod-space-g}
\hspace{-2mm}\Psi:=\bigoplus_{i=1}^{\cal M}\frac{\mu_{i}}{\mu}\Psi^{\langle i\rangle}:  
{\cal Z}\rightarrow (-\infty,\infty]: {\mathbf z}:=(z_{1},\ldots, z_{\cal M})\mapsto 
\sum_{i=1}^{\cal M}\frac{\mu_{i}}{\mu}\Psi^{\langle i\rangle}(z_{i}), 
\end{equation}
together with a new bounded linear operator 
\begin{equation}\label{prod-space-A}
{\mathfrak L}:{\cal X}\rightarrow {\cal Z}: x\mapsto ({\mathfrak L}_{1}x,\ldots, {\mathfrak L}_{\cal M}x). 
\end{equation}
For example, by letting ${\cal X}={\mathbb R}^{m\times n}$ with 
$
\langle \cdot, \cdot\rangle_{\cal X}:(X,Y)\mapsto {\rm tr} \left(X^{\top}Y\right)$,
${\cal Z}_{i}={\mathbb R}^{M_{i}\times N_{i}}$ with   
$
\langle \cdot, \cdot\rangle_{{\cal Z}_{i}}:(X_{i},Y_{i})\mapsto {\rm tr} \left(X_{i}^{\top}Y_{i}\right)
$, we can promote multiple desired features of $X\in {\mathbb R}^{m\times n}$ flexibly by  the model \eqref{intro:multiple-involveL} with $(\Psi^{\langle i\rangle},{\mathfrak L}_{i}, \mu_{i})$ $(i=1,2,\ldots, {\cal M})$.   
\item[(e)]  ({\em Convexity-preserving nonconvex penalties})  The convexity is certainly a key for global optimization.  Indeed, the popularity of 
		$\|\cdot \|_1$ in (b) and (c) has been supported strongly by the fact 
		that it is a convex envelope of  $\|\cdot\|_{0}$, i.e., $\|\cdot\|_{1}$ is the largest convex minorant of $\|\cdot\|_{0}$,  in a vicinity of $0\in \bbR^{l}$.  However restricting   the choice of function  $\Psi$ within convex functions is not the only realistic compromise for ensuring the convexity of   $J_{\Psi\circ {\mathfrak L}}$ in the problem (\ref{intro:involveL}).  For example, by designing strategically a regularizer $\Psi \circ {\mathfrak L}$ combined with the least-squares term in  (\ref{intro:involveL}), we could have alternative possibility to achieve the overall convexity of  (\ref{intro:involveL}), i.e., the convexity of  $J_{\Psi\circ {\mathfrak L}}$.
		The so-called   convexity-preserving nonconvex penalties were introduced, in late 80's by Blake and Zisserman \cite{Blake1987-visual}, and followed for example by Nikolova \cite{nikolova1998-estimation,nikolova1999-Markovian,Nikolova2015-energy}, as  nonconvex regularizers that can maintain the overall convexity after combined with  
		some convex data-fidelity terms.  
		For recent developments of the convexity-preserving nonconvex penalties, see \cite{ding2015-ArtifactFree,Mollenhoff2015-primal,bayram2016-convergence,carlsson2016-Convexification,Lanza2016-convex,Mohammadi2016-overall_convexity,lanza2017-Nonconvex,selesnick2017-Sparsea,soubies2017-Continuous} and references therein.
		Most of these works rely on certain strong convexity assumptions in the least squares  term, which corresponds to the assumption for the nonsingularity of $A^{\sfT}A$ in the scenario of \eqref{intro:involveL}, where $A^{\sfT}$ stands for the adjoint operator of $A$.
		An exceptional example, which is free from such an assumption, has been introduced by Selesnick \cite{Selesnick2017-sparse}
                as the \emph{generalized minimax concave (GMC) penalty function}\renewcommand{\thefootnote}{\arabic{footnote}}\setcounter{footnote}{0}\footnote{We use the notation $(\|\cdot\|_{1})_B$ 
		in place of its original notation $\Psi_{B}$ used  in \cite{Selesnick2017-sparse} for   
		the  GMC penalty  because  the GMC penalty in \cite{Selesnick2017-sparse} was introduced 
		 as a nonconvex alternative to $\|\cdot\|_{1}$ with  $B \in \bbR^{q \times n}$. In Definition \ref{def:GMC} of the present  paper, we will use $\Psi_{B}$  
		 in much wider sense to denote a nonconvex alternative to 
a general proximable convex function $\Psi$ defined on finite dimensional real Hilbert space. } 
		\begin{equation}
(\|\cdot\|_{1})_{B}(\cdot):=\|\cdot\|_{1}-\min_{v\in {\mathbb R}^{n}}
\left[\|v\|_{1}+\frac{1}{2}\|B(\cdot-v)\|_{{\mathbb R}^{q}}^{2}\right]
\label{GMC-penalty-intro}
\end{equation}
with a parameter $B \in \bbR^{q \times n}$. The GMC penalty function is a parameterized multidimensional extension of the \emph{minimax concave (MC) penalty function} \cite{Zhang2010-mcp} (see also \cite{bayram2015-penalty,fornasier2008-iterative})\footnote{The MC penalty 
			\[
                          \mbox{}^{\beta}|\cdot|_{\mathrm{MC}}
                          :{\mathbb R}\rightarrow {\mathbb R}_{+}:x\mapsto \left\{
			\begin{array}{ll}
			|x|-\frac{1}{2\beta}x^{2},&\mbox{ if }|x|\leq \beta,\\
			\frac{\beta}{2},&\mbox{ otherwise,}
			\end{array}
			\right.
			\]
			where $\beta\in {\mathbb R}_{++}$, was introduced in \cite{Zhang2010-mcp} for achieving a nearly unbiased estimate by minimizing 
			$J_{\mathrm{MC}}:{\mathbb R}^{n}\rightarrow {\mathbb R}:x=(x_{1},\ldots, x_{n})^{\top}\mapsto \frac{1}{2} \| y - A x \|^2 + \mu \sum_{i=1}^{n}\mbox{}^{\beta}|x_{i}|_{\mathrm{MC}}$. In fact, by setting 
			$B^{\sfT}B = \beta {\rm Id}$, the GMC penalty function $(\|\cdot\|_{1})_B$ 			
			reproduces the MC penalty function as  
			$(\|\cdot\|_{1})_B(x)=\sum_{i=1}^{n}\mbox{}^{\beta}|x_{i}|_{\mathrm{MC}}$ \cite[Proposition~12]{Selesnick2017-sparse}.}.  It is known that  (i) the GMC penalty function $(\|\cdot\|_{1})_B$ is nonconvex except for $(\|\cdot\|_{1})_{\rmO_{q \times n}}=\|\cdot\|_{1}$ (see
                       Remark~\ref{rem-prop:exJ_cond}(ii));
                      % Proposition~\ref{prop:exJ_cond});
                       (ii) for any $A \in \bbR^{m \times n}$, $(\|\cdot\|_{1})_B$ can maintain the overall convexity of $J_{(\|\cdot\|_{1})_B \circ {\rm Id}}$ in \eqref{intro:involveL} if $A^{\sfT}A - \mu B^{\sfT}B \succeq \rmO_{n}$ is satisfied (see Proposition~\ref{prop:exJ_cond}(b), Remark~\ref{rem-prop:exJ_cond}(iii), and {\cite[Theorem~1]{Selesnick2017-sparse}}).
	\end{enumerate}
\end{example}

The GMC penalty $(\|\cdot\|_{1})_B$ has  great potential for dealing with many nonconvex variations of $\|\cdot\|_1$ 
under single umbrella of the modern convex analysis.
Indeed,  as will be seen in Example \ref{example1:LiGME},  the  GMC function 
can serve as a parametric penalty which bridges the gap between the direct discrete measure of sparsity and its convex envelope function. 
Moreover, for computing a global minimizer of
\begin{equation}
\minimize_{x \in \bbR^n} \  J_{(\|\cdot\|_{1})_B \circ {\mathfrak L}}(x) := \frac{1}{2} \|y - A x \|_{{\mathbb R}^{m}}^2 + \mu (\|\cdot\|_{1})_B \circ {\rm Id}(x), \ \ \mu > 0,
\label{intro:GMCL}
\end{equation}
an iterative algorithm was presented by Selesnick \cite{Selesnick2017-sparse} (see \ref{app:FB}) but only for a special case satisfying 
$B^{\sfT}B = (\theta / \mu) A^{\sfT}A$ ($0 \leq \theta \leq 1$).  
Despite its great potential of the GMC penalty, so far the  applicability of the algorithm in \cite{Selesnick2017-sparse} is very limited. 
 For example, it is not applicable directly  to most  scenarios in Example \ref{ex-regularizations}(c) and (d).  

To maximize the applicability of the excellent ideas of the MC penalty function \cite{Zhang2010-mcp} followed by the GMC penalty function $(\|\cdot\|_{1})_B$  \cite{Selesnick2017-sparse},  we  are interested in the following questions: 
\begin{enumerate}
  \item[(Q1)] Can we extend the model (\ref{intro:GMCL}) proposed in \cite{Selesnick2017-sparse},  without loosing its inherent computational benefit,  to  
\begin{equation}
\minimize_{x \in {\cal X}} \  J_{\Psi_{B} \circ {\mathfrak L}}(x) := 
\frac{1}{2} \|y - Ax\|_{\cal Y}^2 + \mu \Psi_{B} \circ {\mathfrak L}(x), \ \ \mu > 0,
\label{intro:involveL'}
\end{equation}
where ${\cal X}$, ${\cal Y}$, ${\cal Z}$ and $\widetilde{\cal Z}$ are finite dimensional real Hilbert spaces, $y\in {\cal Y}$, $A\in {\cal B}({\cal X},{\cal Y})$, ${\mathfrak L}\in {\cal B}({\cal X},{\cal Z})$ and 
\begin{equation}
\Psi_{B}(\cdot):=\Psi(\cdot)-\min_{v\in {\cal Z}}\left[\Psi(v)+\frac{1}{2}\|B(\cdot-v)\|_{\widetilde{Z}}^{2}\right]
\label{Moreau-E-penalty}
\end{equation}
with $\Psi\in \Gamma_{0}({\cal Z})$ and  $B\in {\cal B}({\cal Z},\widetilde{\cal Z})$ ? 

\item[(Q2)] For given $A \in {\cal B}({\cal X},{\cal Y})$ and 
	${\mathfrak L}\in {\cal B}({\cal X},{\cal Z})$,  what is the general condition for $B \in {\cal B}({\cal Z},\widetilde{\cal Z})$ and $\mu>0$ to ensure the overall convexity of 
	$J_{\Psi_B \circ {\mathfrak L}}$ in (\ref{intro:involveL'}) ?
	\item[(Q3)]  Can we establish any iterative algorithm of guaranteed convergence to globally optimal solution of  (\ref{intro:involveL'})  under general overall-convexity condition ?
	\item[(Q4)]  For given $A \in {\cal B}({\cal X},{\cal Y})$ and 
	${\mathfrak L}\in {\cal B}({\cal X},{\cal Z})$,  can we choose $B \in {\cal B}({\cal Z},\widetilde{\cal Z})$ and  $\mu>0$ flexibly to ensure the overall-convexity $J_{\Psi_B \circ {\mathfrak L}}$ in (\ref{intro:involveL'}) ?  
\end{enumerate}

\begin{remark}
\item[(a)]  (On Q1) 
The function $\Psi_{B}$  in (\ref{Moreau-E-penalty}) is defined in a way similar to  the GMC penalty function $(\|\cdot\|_{1})_B$  in \eqref{GMC-penalty-intro} and  
can be seen as a nonconvexly enhanced penalty 
 for a given much more general convex penalty $\Psi\in \Gamma_{0}({\cal Z})$ than 
$\|\cdot\|_{1}\in \Gamma_{0}({\mathbb R}^{n})$.
	\item[(b)] (On Q2)  In \cite{Selesnick2017-sparse} specially for $({\cal X},{\cal Z},\Psi,{\mathfrak L}) = ({\mathbb R}^{n},{\mathbb R}^{n},\|\cdot\|_{1},{\rm Id})$,  
	a sufficient condition is found for $B$ and $\mu$ to ensure the convexity of $J_{(\|\cdot\|_{1})_B \circ {\rm Id}}$.   We will see in Remark \ref{rem-prop:exJ_cond} that this sufficient condition is  indeed a necessary and sufficient condition to ensure the convexity of $J_{(\|\cdot\|_{1})_B \circ {\rm Id}}$. 
We consider for general $({\cal X},{\cal Z},\Psi,{\mathfrak L})$ the overall convexity condition of \eqref{intro:involveL'}.  
	\item[(c)] (On Q3)  Any iterative algorithm applicable, under fully general overall-convexity conditions, does not seem to have been reported yet 
	even for $({\cal X},{\cal Z},\Psi,{\mathfrak L}) = ({\mathbb R}^{n},{\mathbb R}^{n},\|\cdot\|_{1},{\rm Id})$. 
As imaginable by the significant effort in the art of proximal splitting  \cite{chambolle05,Combettes-Wajs2005,Bauschke2011-convex,Combettes-Pesquet2011,Condat2013-primal_dual,Vu2013-splitting,Combettes2015-Compositions,yamagishi17:_nonex_lagran,Yamada2017-SVM} for minimizing sum of nonsmooth convex functions, 
	it is not trivial to establish algorithm for (\ref{intro:involveL'})  due to the nonconvexity of $\Psi_B$ for general $({\cal X},{\cal Z},\Psi,{\mathfrak L})$ even under the overall convexity condition. 
	\item[(d)] (On Q4)  For practical applications, it is important to establish a flexible way to design $B$ and $\mu$ under the convexity of $J_{\Psi_B \circ {\mathfrak L}}$.

The GMC penalties in the form of  (\ref{intro:involveL'})  with   
$({\cal X},{\cal Z},\Psi)$ have already been reported (see, e.g., \cite{Zhong2018-novel,du2018minmax}).  However these reports do not present any mathematical analysis related to the above key questions (Q1)-(Q4). 
\end{remark}

This paper considers the questions (Q1)-(Q4) and presents a proximal splitting algorithm for problem \eqref{intro:involveL'} with \eqref{Moreau-E-penalty} under as much general overall-convexity condition for $(A, B, {\mathfrak L}, \mu)$ as possible.  After the preliminary section including short reviews on (i)
the elements of convex analysis and optimization and (ii) fixed point theory of nonexpansive operators, we will present in Proposition~\ref{prop:exJ_cond} useful conditions for the overall convexity of $J_{\Psi_B \circ {\mathfrak L}}$ in \eqref{intro:involveL'}. 
Under the overall convexity condition, we next propose a proximal splitting algorithm (Algorithm 1) for problem \eqref{intro:involveL'}.
The proposed algorithm has theoretical guarantee of convergence to a global minimizer of \eqref{intro:involveL'} (see Theorem~\ref{def:Tprop} in Section~\ref{sec:algorithm}) 
and is designed in a way similar to an idea behind the primal-dual splitting method \cite{Vu2013-splitting,Condat2013-primal_dual,Ono2015-hierarchical} which was established specially for minimization of sum of linearly involved convex terms.
%The proposed algorithm is applicable to much more general GMC penalties than the algorithm introduced in \cite{Selesnick2017-sparse}.
Furthermore, we also present 
a flexible way to design $B$ and $\mu$ in Proposition~\ref{rem:selectionB} for  the convexity of $J_{\Psi_B \circ {\mathfrak L}}$. 
To demonstrate the effectiveness of the proposed algorithm, we present numerical experiments in four different sparsity-rank-aware signal processing scenarios.
%, 
%as a nonconvex variation of TV-based regularization \cite{rudin1992nonlinear}, 
%i.e., by using, as $L$ in \eqref{intro:GMCL}, the first order difference operator $D$ in (\ref{def:D}). 
%The numerical experiments  show that a minimizer of $J_{\Psi_B \circ D}$ has excellent edge-preserving performance in comparison with a minimizer of the convex TV regularization in Example~\ref{ex-regularizations}(c).

Preliminary short versions of this paper were presented at conferences \cite{abe19:_convex_edge_preser_signal_recov,yamada19:_global}.

\section{Preliminaries}
Let $\bbN$, $\bbR$, $\bbR_+$, and $\bbR_{++}$ be the sets of natural numbers, real numbers, nonnegative real numbers, and positive real numbers, respectively.
The superscript $(\cdot)^{\top}$ denotes transpose.
For a vector $\bmx := (x_1, x_2, \dots, x_n) \in \bbR^{n}$, we use $\|\bmx\|_p := (\sum_{i=1}^{n} |x_i|^{p} )^{1/p} \ (0 < p < \infty)$, $\|\bmx\|_{\infty} := \max\{|x_1|, \dots, |x_n| \}$, and $\|\bmx\|_{0}:=\#\{i\in \mathbb{N}\cap[1,n]\mid x_i \not = 0\}$.
$\bmzero_n \in \bbR^n$ stands for the zero vector.
%Inner product spaces $(\bbR^n, \langle \cdot, \cdot \rangle, \| \cdot \|)$ and $(\bbR^n, \langle \cdot, \cdot, \rangle_{A}, \| \cdot \|_{A})$ with $A \succ \rmO_{n}$ are typical examples of the finite dimensional real Hilbert spaces because every finite dimensional inner product space is complete.
In Section~\ref{sec:elem_convex} and Section~\ref{sec:elem_nonexp}, we use finite dimensional real Hilbert spaces $(\Hlb,\langle \cdot, \cdot \rangle_{\Hlb}, \| \cdot \|_{\Hlb})$ and $(\mathcal{K},\langle \cdot, \cdot \rangle_{\mathcal{K}}, \| \cdot \|_{\mathcal{K}})$.
For $S \subset \Hlb$, $\operatorname{cone}(S)$ denotes the conical hull (see, e.g., \cite[Def. 6.1]{Bauschke2011-convex}) of $S$ and $\operatorname{span}(S)$ the span of $S$. $\mathcal{B} (\Hlb, \mathcal{K})$ denotes the set of all bounded linear operators\footnote{
  In real finite dimensional Hilbert space, $\mathcal{B} (\Hlb, \mathcal{K})$ is identical to the set of all linear operators.
} from $(\Hlb, \langle \cdot, \cdot \rangle_{\Hlb}, \| \cdot \|_{\Hlb})$ to $(\mathcal{K}, \langle \cdot, \cdot \rangle_{\mathcal{K}}, \| \cdot \|_{\mathcal{K}})$.
For $L \in \mathcal{B}(\Hlb, \mathcal{K})$, we use $\|L\|_{\rm op}:=\sup_{x \in \Hlb\colon \|x\|_{\Hlb} \leq 1}\|Lx\|_{\mathcal{K}}$.
For $L \in \mathcal{B}(\Hlb, \mathcal{K})$, $L^{\ast} \in \mathcal{B}(\mathcal{K}, \Hlb)$ denotes the adjoint of $L$, i.e., $\langle Lx, y \rangle_{\mathcal{K}}= \langle x, L^*y \rangle_{\Hlb} \ (\forall (x,y) \in \mathcal{H} \times \mathcal{K})$. We also use $\rmId$ to denote the identity operator for general Hilbert spaces.
$\rmO_{\mathcal{B}(\Hlb, \mathcal{K})} \in \mathcal{B}(\Hlb, \mathcal{K})$ and $\rmO_{\Hlb} \in \mathcal{B}(\Hlb, \Hlb)$ stand for the zero operators.
For $L \in \mathcal{B}(\Hlb, \mathcal{K})$, $L^{\dagger} \in \mathcal{B}(\mathcal{K}, \Hlb)$ stands for the Moore-Penrose pseudo inverse of $L$, $\ran(L) := \{ L\bmx \in \mathcal{K} \mid \bmx \in \Hlb \}$ and $\Null(L) := \{ \bmx \in \Hlb \mid L\bmx = \bmzero \}$ denote respectively the range and the null spaces of $L$.
%, and $\rmI_{\Hlb} \in \mathcal{B}(\Hlb, \Hlb)$ the identity operator.
The positive definiteness and positive semidefiniteness of a self-adjoint operator $L \in \mathcal{B}(\Hlb, \Hlb)$ are expressed respectively as $L \succ \rmO_{\Hlb}$ and $L \succeq \rmO_{\Hlb}$.
For $L \succeq \rmO_{\Hlb}$, $\rho(L)$ denotes the maximum eigenvalue of $L$. 
For any $L \succ \rmO_{\Hlb}$, by defining an inner product $\langle \cdot, \cdot \rangle_{L} : \Hlb \times \Hlb \to \bbR\colon (\bmx, \bmy) \mapsto  \langle \bmx, L \bmy \rangle_{\Hlb}$ and its induced norm $\|\bmx\|_{L} := \sqrt{\langle \bmx, \bmx \rangle_{L}}$, $(\Hlb, \langle \cdot, \cdot \rangle_{L}, \|\bmx\|_{L})$ becomes a real Hilbert space. 

Note that, in any real finite dimensional space, a linear operator can be expressed with matrix multiplication and identified with a matrix. We use $\rmI_n \in \mathbb{R}^{n \times n}$ to denote the identity matrix for $\mathbb{R}^n$. $\rmO_{m,n} \in \mathbb{R}^{m \times n}$ and $\rmO_{n} \in \mathbb{R}^{n \times n}$ stand for the zero matrices.

\subsection{Selected elements of convex analysis and optimization}
\label{sec:elem_convex}

The class of proper lower semicontinuous convex functions $f \colon \Hlb \to (-\infty, \infty]$, i.e., $f$ is convex function whose lower level set $\{x \in \Hlb \mid f(x) \leq \alpha\}$ is closed for every $\alpha \in \mathbb{R}$ and $\dom (f) := \{ x \in \Hlb \mid f(x) < \infty \} \neq \varnothing$, is denoted by $\Gamma_0 (\Hlb)$. For convex $C \subset \mathcal{H}$, the relative interior of $C$ is $\operatorname{ri} C:=\left\{x \in \mathcal{H} \mid \operatorname{cone}(C-x) = \operatorname{span}(C-x) \right\}$ (see, e.g., \cite[Def. 6.9]{Bauschke2011-convex}).

\noindent\textbf{(Subdifferential)}
For a function $f \in \Gamma_0(\Hlb)$, the \emph{subdifferential} of $f$ is defined as the set valued operator
\begin{equation*}
\partial f \colon \Hlb \to 2^{\Hlb} \colon x \mapsto \{ u \in \Hlb \mid \langle y - x, u \rangle_{\Hlb} + f(x) \leq f(y), \ \forall y \in \Hlb \}.
\end{equation*}
% As a generalization of G{\^{a}}teaux derivative\footnote{%
% 	\textbf{(G{\^{a}}teaux and Fr{\'{e}}chet derivatives of function)} Let $U$ be an open subset of $\Hlb$.
% 	Then a function $f \colon U \to \bbR$ is said to be \emph{{G{\^{a}}teaux differentiable}} at $x \in U$ if there exists $a(x) \in \Hlb$ such that $\lim_{\delta \to 0} \frac{f(x + \delta h) - f(x)}{\delta} = \langle a(x), h \rangle_{\Hlb} \ (\forall h \in \Hlb)$.
% 	In this case, $\nabla f(x) := a(x)$ is called \emph{G{\^{a}}teaux derivative} (or \emph{gradient}) of $f$ at $x$.
	
% 	On the other hand, a function $f \colon U \to \bbR$ is called \emph{Fr{\'{e}}chet differentiable} over $U$ if for each $u \in U$ there exists $a(u) \in \Hlb$ such that
% 	\begin{equation*}
% 		f(u+h) = f(u) + \langle a(u), h \rangle_{\Hlb} + \o (\| h \|_{\Hlb}) \text{ for all } h \in \Hlb,
% 	\end{equation*}
% 	where $r(h) = \o (\| h \|_{\Hlb})$ means $\lim_{h \to 0} r(h) / \| h \|_{\Hlb} = 0$. In this case, $\nabla f \colon U \to \Hlb$ defined by $\nabla f(u) := a(u)$ is called \emph{Fr{\'{e}}chet derivative} of $f$ over $U$. If $f$ is Fr{\'{e}}chet differentiable over $U$, $f$ is also G{\^{a}}teaux differentiable over $U$ and both derivatives coincide. Moreover, if $f$ is G{\^{a}}teaux differentiable with continuous derivative $\nabla f$ over $U$, then $f$ is called Fr{\'{e}}chet differentiable over $U$.}, 
\noindent Subdifferential has the following properties:
\begin{enumerate}[(a)]
	\item (Fermat's rule \cite[Theorem~16.3]{Bauschke2011-convex}) Let $f \in \Gamma_0(\Hlb)$ and $\bar{x} \in \Hlb$. Then
	\begin{eqnarray}
		\label{eq:Fermat}
		\bar{x} \in \argmin_{x \in \Hlb} f(x) \IFF 0 \in \partial f(\bar{x}).
	\end{eqnarray}
        \item (Sum rule \cite[Corollary~16.48]{Bauschke2011-convex})  Let $f, g \in \Gamma_0(\Hlb)$ with $\dom(g) = \Hlb$. Then
	\begin{align}
		\label{eq:sumrule}
		\partial (f+g) = \partial f+ \partial g.
	\end{align}        
      \item (Chain rule \cite[Corollary~16.53, Fact 6.14(i), Sec. 6.2]{Bauschke2011-convex})        
        Let $g \in \Gamma_0(\Hlb)$
        and  $L \in \mathcal{B}(\Hlb, \mathcal{K})$ satisfy $0_{\Hlb} \in \operatorname{ri}\left(\operatorname{dom}(g) - \ran L \right)$. Then
	\begin{align}
		\label{eq:chainrule}
		\partial (g \circ L) = L^{\ast} \circ (\partial g) \circ L.
	\end{align}
	\item (\cite[Proposition~17.31]{Bauschke2011-convex}) Let $f \in \Gamma_0(\Hlb)$, let $x \in \dom(f)$, and suppose that $f$ is (G{\^{a}}teaux) differentiable at $x$. Then $\partial f(x) = \{ \nabla f(x) \}$.
\end{enumerate}

% \begin{fact}[{\textbf{Sum Rule for Regular/Fr{\'{e}}chet subdifferential {\cite[Exercise~8.8]{Rockafellar2009-variational}}}}]
% 	Let $f \in \Gamma_0(\Hlb)$. Assume $g \colon \Hlb \to \bbR$ is continuously Fr{\'{e}}chet differentiable in a neighborhood of $\bar{x} \in \dom(f)$ and $f+g \in \Gamma_0(\Hlb)$.
% 	Then
% 	\begin{equation*}
% 		\partial (f + g)(\bar{x}) = \hat{\partial} (f + g)(\bar{x}) = \hat{\partial} f (\bar{x}) + \nabla g(\bar{x}) = \partial f(\bar{x}) + \nabla g(\bar{x}),
% 	\end{equation*}
	% where, for $\phi \in \Gamma_0(\Hlb)$, the \emph{regular subdifferential} of $\phi$ is defined, as in \cite[8.B.]{Rockafellar2009-variational} by
	% \begin{equation*}
	% 	\hat{\partial}\phi \colon \Hlb \to 2^{\Hlb} \colon \bar{x} \mapsto \left\{ u \in \Hlb \relmiddle| \liminf_{\scriptstyle x \to \bar{x} \atop\scriptstyle x \neq \bar{x}} \frac{\phi(x) - \phi(\bar{x}) - \langle u, x - \bar{x} \rangle_{\Hlb}}{\| x - \bar{x} \|_{\Hlb}} \geq 0 \right\}.
	% \end{equation*} \label{item:subgrad_diff_sum}
%\end{fact}

\noindent\textbf{(Legendre-Fenchel conjugate)}
For any $f \in \Gamma_0(\Hlb)$, the function defined by
\begin{equation*}
	f^{\ast} \colon \Hlb \to (-\infty, \infty] \colon y \mapsto \sup_{x \in \Hlb} \{ \langle x, y \rangle_{\Hlb} - f(x) \}
\end{equation*}
satisfies $f^{\ast} \in \Gamma_0(\Hlb)$. This function is called the \emph{conjugate} (also named \emph{Legendre-Fenchel conjugate}) of $f$.
Let $f \in \Gamma_0 (\Hlb)$. Then, for any $(x, u) \in \Hlb \times \Hlb$,
\begin{align}
\label{eq:subdifferentialinversion}
u \in \partial f(x) \IFF x \in \partial f^{\ast} (u).
\end{align}

\subsection{Selected elements of fixed point theory of nonexpansive operators}
\label{sec:elem_nonexp}
\noindent\textbf{(Nonexpansive operator)}
An operator $T \colon \Hlb \to \Hlb$ is said to be $\kappa$-\emph{Lipschitzian} with constant $\kappa > 0$ if
\begin{equation*}
(\forall x, y \in \Hlb) \quad \| T(x) - T(y) \|_{\Hlb} \leq \kappa \| x - y \|_{\Hlb}.
\end{equation*}
In particular, an operator $T \colon \Hlb \to \Hlb$ is said to be \emph{nonexpansive} if it is 1-Lipschitzian, i.e.,
\begin{equation*}
(\forall x, y \in \Hlb) \quad \| T(x) - T(y) \|_{\Hlb} \leq \| x - y \|_{\Hlb}.
\end{equation*}
For $\alpha \in (0, 1)$, a nonexpansive operator $T$ is called \emph{$\alpha$-averaged} if there exists a nonexpansive operator $\widehat{T} \colon \Hlb \to \Hlb$ such that
\begin{equation*}
	T = (1 - \alpha) \rmId + \alpha \widehat{T},
\end{equation*}
i.e., $T$ is a convex combination of the identity operator $\rmId$ and some nonexpansive operator $\widehat{T}$.

\begin{fact}[Compositions of averaged nonexpansive operators {\cite{ogura02:_non}\cite[Proposition~2.4]{Combettes2015-Compositions}}]
  \label{fact:averaged}
	Suppose that each $T_i \colon \Hlb \to \Hlb$ $(i = 1,2)$ is $\alpha_i$-averaged nonexpansive for some $\alpha_i \in (0,1)$. Then $T_1 \circ T_2$ is $\alpha$-averaged nonexpansive for $\alpha := \frac{\alpha_1 + \alpha_2 - 2\alpha_1\alpha_2}{1-\alpha_1\alpha_2} \in (0,1)$.
\end{fact}

\begin{fact}[Krasnosel'ski{\u{\i}}-Mann iteration for finding a fixed point of averaged nonexpansive operator {\cite[Section~5.2]{Bauschke2011-convex}\cite{Groetsch1972-mann}}]
	\label{fact:KM}
	For a nonexpansive operator $T \colon \Hlb \to \Hlb$ with $\Fix(T) := \{ x \in \Hlb \mid T(x) = x \} \neq \varnothing$ and any initial point $x_0 \in \Hlb$, the sequence $(x_k)_{k \in \bbN} \subset \Hlb$ generated by
	\begin{equation}
		x_{k+1} = [(1-\alpha_k) \rmId + \alpha_k T] (x_k)
		\label{eq:Mann_proc}
	\end{equation}
	converges weakly to a point in $\Fix(T)$ if $(\alpha_k)_{k \in \bbN} \subset [0,1]$ satisfies $\sum_{k \in \bbN} \alpha_k(1-\alpha_k) = \infty$.
	In particular, if $T$ is $\alpha$-averaged for some $\alpha \in (0,1)$, a simple iteration
	\begin{equation*}
		x_{k+1} = T(x_{k})
	\end{equation*}
	converges weakly to a point in $\Fix(T)$.	
\end{fact}

\noindent\textbf{(Monotone operator)}
A set-valued operator $T \colon \Hlb \to 2^{\Hlb}$ is said to be \emph{monotone} if
\begin{equation*}
	(\forall (x,u) \in \gra(T))(\forall (x',u') \in \gra(T)) \quad \langle x - x', u - u' \rangle_{\Hlb} \geq 0,
\end{equation*}
where $\gra(T) := \{ (x, u) \in \Hlb \times \Hlb \mid u \in T(x) \}$ is the graph of $T$.
In particular, $T$ is called \emph{maximally monotone} if, for every $(x, u) \in \Hlb \times \Hlb$,
\begin{equation*}
(x, u) \in \gra(T) \IFF (\forall (x', u') \in \gra(T)) \quad \langle x - x', u - u' \rangle_{\Hlb} \geq 0.
\end{equation*}
For a given $f \in \Gamma_0(\Hlb)$, $\partial f \colon \Hlb \to 2^{\Hlb}$ is maximally monotone. Furthermore, $T \colon \Hlb \to 2^{\Hlb}$ is maximally monotone if and only if the \emph{resolvent} $R_T := (\rmId + T)^{-1} \colon \Hlb \to 2^{\Hlb} \colon u \mapsto \{ x \in \Hlb \mid u \in x + T(x)\}$ is single-valued $(1/2)$-averaged nonexpansive operator.

\noindent\textbf{(Proximity operator)}
The \emph{proximity operator} of $f \in \Gamma_0(\Hlb)$ is defined by
\begin{equation*}
\Prox_{f} \colon \Hlb \to \Hlb : x \mapsto \argmin_{y \in \Hlb} \left[ f(y) + \frac{1}{2} \|x - y \|_{\Hlb}^2  \right].
\end{equation*}
Note that $\Prox_f (x) \in \Hlb$ is well-defined for all $x \in \Hlb$ due to the coercivity and the strict convexity of $f(\cdot) + \frac{1}{2}\| x - \cdot \|_{\Hlb}^2 \in \Gamma_0(\Hlb)$.
It is also well known that $\Prox_{f}$ is nothing but the resolvent of $\partial f$, i.e., $\Prox_{f} = (\rmId + \partial f)^{-1} = R_{\partial f}$, which implies that
\begin{align}
&\bar{x} \in \Fix(\Prox_f) \IFF \Prox_{f}(\bar{x}) = \bar{x} \IFF (\rmId + \partial f)^{-1} (\bar{x}) = \bar{x} \\
&\IFF \bar{x} \in (\rmId + \partial f)(\bar{x}) \IFF 0 \in \partial f(\bar{x}) \IFF \bar{x} \in \argmin_{x \in \Hlb} f(x).
\end{align}
 The proximity operator of $\Psi^*$ can be expressed as
  $\Prox_{ \Psi^*}= \rmId -  \Prox_{\Psi} $ (see e.g. \cite[Theorem 14.3(ii)]{Bauschke2011-convex}).

\noindent\textbf{(Moreau envelope)}
For $f \in \Gamma_0(\Hlb)$,
\begin{equation}
\label{eq:Moreauenvelope}
{}^{\gamma}\! f \colon \Hlb \to \bbR \colon x \mapsto \min_{y \in \Hlb} \left[ f(y) + \frac{1}{2\gamma} \| x - y \|_{\Hlb}^2  \right],
\end{equation}
is called the \emph{Moreau envelope} of $f$ of index $\gamma > 0$.
The Moreau envelope of $f \in \Gamma_0(\Hlb)$ converges pointwise to $f$ on $\dom(f)$ as $\gamma \downarrow 0$, i.e. $\lim_{\gamma \downarrow 0} {}^{\gamma}\! f(x) = f(x)$ for every $x \in \dom(f)$.
The function ${}^{\gamma}\! f$ is Fr{\'{e}}chet differentiable convex function with ($1/\gamma$)-Lipschitzian gradient
\begin{align}
\label{eq:gradientMoreauenvelope}
\nabla {}^{\gamma}\! f \colon \Hlb \to \Hlb \colon x \mapsto \frac{x - \Prox_{\gamma f}(x)}{\gamma}. 
\end{align}

\section{Linearly involved Generalized-Moreau-Enhanced (LiGME) model and proximal splitting algorithm}
In this section, after introducing LiGME model (see Definition \ref{def:GMC}), we then presents a proximal splitting type algorithm of guaranteed convergence to a globally optimal solution of the model under an overall convexity condition (see Theorem~\ref{def:Tprop}).
% We shall introduce \emph{generalized-Moreau-Enhanced (GME) penalty}, as an extension of the GMC penalty \cite{Selesnick2017-sparse} (see Definition~\ref{def:GMC}(a)). We also introduce the linearly involved variant of the GME penalty, say Linearly involved GME (LiGME) penalty (see Definition~\ref{def:GMC}(b)). This variant significantly expands the applicability of the GME penalty because it enables to cover much more general penalties (see Proposition~\ref{prop:productspce}). In addition, we introduce \emph{LiGME model} as the penalized least squares minimization by the LiGME penalty (see Definition~\ref{def:GMC}(c)). For the LiGME model under overall convexity condition (see Problem~\ref{prob:optim}), we propose an iterative algorithm with global convergence guarantee to a solution. 

\subsection{Linearly involved Generalized-Moreau-Enhanced (LiGME) Model } %
%In this section, we define the \emph{generalized-Moreau-enhanced (GME) penalty}. % established in \cite{Selesnick2017-sparse}.
% Interesting use of the (generalized) Moreau-envelope in designing a nonconvex alternative to $\|\cdot\|_1$ was introduced in \cite{Selesnick2017-sparse}.
% The alternative, called the GMC penalty, is defined as the difference of $\|\cdot\|_1$ and its generalized Moreau envelope. This strategy is also applied to designing a nonconvex alternative to the standard total variation penalty \cite{selesnick2017total}.
                                                                                                                     %                                                                                                                      We extend this strategy to the cases of a general convex function defined over the real Hilbert space.
We impose the relatively strong assumption $\dom{\Psi}=\mathcal{Z}$ for $\Psi$ in \eqref{Moreau-E-penalty}, to reduce technical complexity in the later discussion, although there would be many ways to relax. 
\begin{definition}[Linearly involved Generalized-Moreau-Enhanced (LiGME) Model]
  \label{def:GMC}
  Let $(\calX, \langle \cdot, \cdot \rangle_{\calX}, \|\cdot\|_{\calX})$,
  $(\calY, \langle \cdot, \cdot \rangle_{\calY}, \|\cdot\|_{\calY})$,
  $(\calZ, \langle \cdot, \cdot \rangle_{\calZ}, \|\cdot\|_{\calZ})$,
  and $(\widetilde{\calZ}, \langle \cdot, \cdot \rangle_{\widetilde{\calZ}}, \|\cdot\|_{\widetilde{\calZ}})$ be finite dimensional real Hilbert spaces, $\Psi \in \Gamma_0(\mathcal{Z})$ coercive with $\dom{\Psi}=\mathcal{Z}$, $B \in \mathcal{B}(\mathcal{Z}, \widetilde{\mathcal{Z}})$, $\mathfrak{L} \in \mathcal{B}(\mathcal{X},\mathcal{Z})$, and $(A,\mathfrak{L},\mu) \in \mathcal{B}(\mathcal{X}, \mathcal{Y}) \times \mathcal{B}(\mathcal{X}, \mathcal{Z}) \times \bbR_{+}$.   Then:
	%For  coercive  %with $\operatorname{dom} \Psi=\mathcal{Z}$
        %and for ,
  \\
  (a) \emph{GME penalty function} $\Psi_B \in \Gamma_0(\mathcal{Z})$ is defined as
%        the difference of $\Psi$ and its generalized Moreau envelope, i.e.,
	\begin{equation}
	\Psi_B(\cdot) := \Psi( \cdot ) - \min_{v \in \mathcal{Z}} \left[ \Psi(v) + \frac{1}{2} \|B(\cdot - v)\|^2_{\widetilde{\mathcal{Z}}} \right].
	\label{eq:defGMC}
      \end{equation}
      (b) Linearly involved Generalized-Moreau-Enhanced (LiGME) penalty is defined as $\Psi_B \circ \mathfrak{L} \colon \mathcal{X} \to (-\infty, \infty]$. \\
      (c) LiGME model is defined as the minimization of
    \begin{align}
      \label{eq:J}
      J_{\Psi_B \circ \mathfrak{L}}\colon \mathcal{X} \to \mathbb{R}\colon x \mapsto \frac{1}{2} \| y - A x \|^2_{\mathcal{Y}} + \mu \Psi_B \circ \mathfrak{L}(x).
    \end{align}
      \end{definition}

\begin{example} \label{example1:LiGME}(LiGME penalty bridges the gap between the direct discrete measures and their convex envelopes)
 \begin{enumerate}
 \item[(a)] (Normalized MC penalty) 
By letting ${\cal X}={\cal Z}={\mathbb R}$, $\Psi=|\cdot|$, $\mathfrak{L}=1$,  $B=\frac{1}{\sqrt{\gamma}}$ for $\gamma\in {\mathbb R}_{++}$ and $\mu=\frac{2}{\gamma}$, the function $\mu\Psi_{B}\circ \mathfrak{L}$ in 
(\ref{eq:J}) reproduces 
 \begin{equation}
\frac{2}{\gamma} \left({}^{\gamma}|\cdot|_{\rm MC}\right):{\mathbb R}\rightarrow {\mathbb R}:
x\mapsto  
\left\{
\begin{array}{ll}
\frac{2}{\gamma}|x|-\frac{1}{\gamma^{2}}x^{2}, &\mbox{\rm if } |x|\leq {\gamma}\ ; \\
1,&\mbox{\rm otherwise,}
\end{array}
\right.
\label{normal-MC}
\end{equation}
which satisfies 
\begin{equation}
\lim_{\gamma \downarrow 0}\frac{2}{\gamma} \left({}^{\gamma}|x|_{\rm MC}\right)
=
\left\{
\begin{array}{ll}
0, &\mbox{\rm if } x=0\ ; \\
1,&\mbox{\rm otherwise.}
\end{array}
\right.
\label{lim-normal-MC}
\end{equation}
 \item[(b)] (LiGME penalty bridges the gap between $\|\cdot\|_0$ and $\|\cdot\|_1$ ) 
Let ${\cal X}={\cal Z}={\mathbb R}^{n}$, $\Psi=\|\cdot\|_{1}$, $\mathfrak{L}={\rm Id}$,  $B=\frac{1}{\sqrt{\gamma}}{\rm Id}$ for $\gamma\in {\mathbb R}_{++}$ and $\mu=\frac{2}{\gamma}$. 
Then the function $\mu\Psi_{B}\circ \mathfrak{L}$ in 
(\ref{eq:J}) reproduces 
 \begin{equation}
\frac{2}{\gamma} (\|\cdot\|_{1})_{\frac{1}{\sqrt{\gamma}}{\rm Id}}:
{\mathbb R}^{n}\rightarrow {\mathbb R}:
(x_{1},\ldots, x_{n})\mapsto  
\sum_{i=1}^{n}\frac{2}{\gamma} \left({}^{\gamma}|x_{i}|_{\rm MC}\right)
\label{n-normal-MC}
\end{equation}
which satisfies for $(x_{1},\ldots, x_{n})\in {\mathbb R}^{n}$
\begin{equation}
\lim_{\gamma \downarrow 0}
\frac{2}{\gamma} (\|\cdot\|_{1})_{\frac{1}{\sqrt{\gamma}}{\rm Id}}(x_{1},\ldots, x_{n})=\|(x_{1},\ldots, x_{n})\|_{0}.
\label{lim-normal-MC2}
\end{equation}
This fact together with 
$(\|\cdot\|_{1})_{\rmO_{m,n}}(x_{1},\ldots, x_{n})=\|(x_{1},\ldots, x_{n})\|_{1}$ 
validates that the LiGME penalty can serve as  a parametrized bridge between $\|\cdot\|_0$ and $\|\cdot\|_1$. 
 \item[(c)]  (LiGME penalty bridges the gap between ${\rm rank}(\cdot)$ and $\|\cdot\|_{\rm nuc}$) 
 Let ${\cal X}={\cal Z}={\mathbb R}^{m\times n}$, $\Psi=\|\cdot\|_{\rm nuc}$, $\mathfrak{L}={\rm Id}$,  $B=\frac{1}{\sqrt{\gamma}}{\rm Id}$ for $\gamma\in {\mathbb R}_{++}$ and $\mu=\frac{2}{\gamma}$, where  $\|\cdot\|_{\rm nuc}:{\mathbb R}^{m\times n}\rightarrow {\mathbb R}: X\mapsto \sum_{i=1}^{r}\sigma_{i}(X)$ with 
 $r={\rm rank}(X)$ and $i$-th largest singular value $\sigma_{i}(X)$ ($i=1,2,\ldots,r$) of $X$. 
 It is well-known that $\|\cdot\|_{\rm nuc}$ is a convex envelope of ${\rm rank}(\cdot)$, i.e.,
the largest convex minorant of ${\rm rank}(\cdot)$, in a vicinity of $\rmO_{m,n}$.
By \cite[Prop. 24.68]{Bauschke2011-convex}, the function $\mu\Psi_{B}\circ \mathfrak{L}$ in 
(\ref{eq:J}) reproduces 
\begin{equation}
\frac{2}{\gamma} (\|\cdot\|_{\rm nuc})_{\frac{1}{\sqrt{\gamma}}{\rm Id}}:
{\mathbb R}^{m\times n}\rightarrow {\mathbb R}:
X\mapsto  
\sum_{i=1}^{r}\frac{2}{\gamma} \left({}^{\gamma}|\sigma_{i}(X)|_{\rm MC}\right)
\label{n-normal-MC-rank}
\end{equation}
which satisfies for $X\in {\mathbb R}^{m\times n}$
\begin{equation}
\lim_{\gamma \downarrow 0}
\frac{2}{\gamma} (\|\cdot\|_{\rm nuc})_{\frac{1}{\sqrt{\gamma}}{\rm Id}}(X)=\|(\sigma_{1}(X),\ldots, \sigma_{r}(X))\|_{0}={\rm rank}(X).
\label{lim-normal-MC4}
\end{equation}
This fact together with 
$(\|\cdot\|_{\rm nuc})_{\rmO_{m,n}}(X)=\|X\|_{\rm nuc}$ 
validates that the LiGME penalty can serve as  a parametrized bridge between ${\rm rank}(\cdot)$ and $\|\cdot\|_{\rm nuc}$.  
\end{enumerate} 
 \end{example}     

\begin{example}
  \label{prop:productspce}(The sum of multiple LiGME penalties can be expressed as a single LiGME penalty on product space)
  Let $\mathcal{Z}_i, \widetilde{\mathcal{Z}}_i \ (i= 1,2,\ldots, \mathcal{M})$, 
  $\mathcal{Z}=\mathcal{Z}_1 \times \mathcal{Z}_2 \times \ldots \times \mathcal{Z}_{\mathcal{M}}$,
  and 
  $\widetilde{\mathcal{Z}}=\widetilde{\mathcal{Z}}_1 \times \widetilde{\mathcal{Z}}_2 \times \ldots \times \widetilde{\mathcal{Z}}_{\mathcal{M}}$
  be real Hilbert spaces. For coercive $\Psi^{\langle i\rangle} \in \Gamma_0(\mathcal{Z}_i)$ with $\dom{\Psi^{\langle i\rangle}}=\mathcal{Z}_i$, $B^{\langle i\rangle}  \in \mathcal{B}(\mathcal{Z}_i, \widetilde{\mathcal{Z}}_i)$ and $\mathfrak{L}_i \in \mathcal{B}(\mathcal{X}, \mathcal{Z}_i)$ $\ (i= 1,2,\ldots, \mathcal{M})$, let $\Psi:=\mu_1\Psi^{\langle 1\rangle}\oplus \mu_2\Psi^{\langle 2\rangle}\oplus \ldots \oplus \mu_{\mathcal{M}}\Psi^{\langle\mathcal{M}\rangle}$, $B\colon \calZ \to \widetilde{\calZ} \colon (z_1,\ldots,z_M)\mapsto \left(\sqrt{\mu_1}B^{\langle 1\rangle} z_1,\ldots, \sqrt{\mu_{\mathcal{M}}}B^{\langle \mathcal{M}\rangle}z_{\mathcal{M}}\right)$, and $\mathfrak{L}\colon\mathcal{X} \to \mathcal{Z} \colon x \mapsto (\mathfrak{L}_ix)_{i=1}^{\mathcal{M}}$. Then we have
  \begin{align}
    \label{eq:ex3PBL}
    \Psi_B \circ \mathfrak{L}=\sum_{i=1}^{\mathcal{M}} \mu_i (\Psi^{\langle i\rangle})_{B^{\langle i\rangle}}\circ \mathfrak{L}_i,
    \end{align}
where $
    (\Psi^{\langle i\rangle})_{B^{\langle i\rangle}}(\cdot) = \Psi^{\langle i\rangle}(\cdot)-\min_{v \in \mathcal{Z}_i}\left[
    \Psi^{\langle i\rangle}(v) + \frac{1}{2}\| B^{\langle i\rangle}(\cdot - v)\|_{\widetilde{\mathcal{Z}}_i}^2
    \right]$.
\end{example}

\begin{remark}
  % The LiGME reproduces existing (generalized-)Moreau-enhanced penalties (as well as Proposition ~\ref{prop:exJ_cond}(b') reproduces its associated minimization model and convexity condition). For example it reproduces:
  % the minimax-concave penalty in \cite{Zhang2010-mcp} as the case of $\mathcal{X}=\mathcal{Z}=\widetilde{\mathcal{Z}}=\mathbb{R}^l$, $B=\rmId$, $\mathfrak{L}=\rmId$, and $\Psi=\|\cdot\|_1$; the GMC penalty in \cite{Selesnick2017-sparse} as the case of $\mathcal{X}=\mathcal{Z}=\mathbb{R}^l$, $\mathfrak{L}=\rmId$, and $\Psi=\|\cdot\|_1$;
  The LS-CNC penalty function in \cite[Definition 2]{yin2019stable} is reproduced as an LiGME penalty by setting
  %(see also \cite[(49) and Lemma 3]{yin2019stable} for its associated minimization model and convexity condition) as the case of
  $\mathcal{X}=\mathcal{Z}=\mathbb{R}^{m \times n} \oplus \mathbb{R}^{m \times n}$, $\mathfrak{L}=\rmId$, and $\Psi=\Psi_1\oplus\Psi_2$ with $\Psi_1:=\alpha\|\cdot\|_{\rm nuc}$ and $\Psi_2:= \beta\|\cdot\|_1$ in \eqref{prod-space-g}, where $\alpha, \beta \geq 0$. Moreover, the LiGME penalty in Example~\ref{prop:productspce} can also be utilized to enhance the so-called morphologicl component analysis in \cite{Starck-Murtagh-Fadili15}.

% The so-called morphological component analysis \cite{} which aims to decompose a given signal as the sum of signals each of which is repesented as a linear combination of the corresponding dictionary with sparse coefficients is in general interpreted as the least-squares problem peneralized with the sum $\sum_{i=1}^{\mathcal{M}} \mu_i \Psi^{\langle i\rangle}\circ \mathfrak{L}_i$ with sparsity promoting penalties $\Psi^{\langle i\rangle}$ and dictionary matrices $\mathfrak{L}_i$. Thus it is enhanced immediately by replacing the sum by the one in Example~\ref{prop:productspce}.
\end{remark}

%      \noindent Then, we shall consider, for any $L \in \mathcal{B}(\mathcal{X},\mathcal{Z})$, its linearly involved variant $\Psi_B \circ L \in \mathcal{X} \to (-\infty, \infty]$ because it significantly extends the applicability of the GME penalty function (see Proposition~\ref{prop:productspce}).
%      In the following, we call it as Linearly involved Generalized-Moreau-Enhanced (LiGME).

% \begin{definition}[Generalized-Moreau-Enhanced (GME) penalty function $\Psi_B$]
%   \label{def:GMC}
%   Let $\mathcal{Z}$ and $\widetilde{\mathcal{Z}}$ be real Hilbert spaces.
% 	For  coercive $\Psi \in \Gamma_0(\mathcal{Z})$ %with $\operatorname{dom} \Psi=\mathcal{Z}$
%         and for $B \in \mathcal{B}(\mathcal{Z}, \widetilde{\mathcal{Z}})$, (a) the \emph{GME penalty function} $\Psi_B \in \Gamma_0(\mathcal{Z})$ is defined as
%         the difference of $\Psi$ and its generalized Moreau envelope, i.e.,
% 	\begin{equation}
% 	\Psi_B(\cdot) := \Psi( \cdot ) - \min_{v \in \mathcal{Z}} \left[ \Psi(v) + \frac{1}{2} \|B(\cdot - v)\|^2 \right].
% 	\label{eq:defGMC}
% 	\end{equation}
%       \end{definition}
%       \noindent Then, we shall consider, for any $L \in \mathcal{B}(\mathcal{X},\mathcal{Z})$, its linearly involved variant $\Psi_B \circ L \in \mathcal{X} \to (-\infty, \infty]$ because it significantly extends the applicability of the GME penalty function (see Proposition~\ref{prop:productspce}).
%       In the following, we call it as Linearly involved Generalized-Moreau-Enhanced (LiGME).

\begin{proposition}[Overall convexity condition for the LiGME model]
  \label{prop:exJ_cond}
  The GME penalty function $\Psi_B$ in Definition~\ref{def:GMC} has the following properties:
  \begin{enumerate}[(a)]
  \item[\rm (a)] $\Psi_B\circ \mathfrak{L}(x) = \Psi(\mathfrak{L}x) - \left[\Psi(0_{\mathcal{Z}})+\frac{1}{2}\|B\mathfrak{L}x\|_{\widetilde{\mathcal{Z}}}^2\right]$ if and only if $B^{\sfT}B\mathfrak{L}x \in \operatorname{argmin}(\Psi^*)$.
  \item[\rm (b)] Let $(A,\mathfrak{L},\mu) \in \mathcal{B}(\mathcal{X}, \mathcal{Y}) \times \mathcal{B}(\mathcal{X}, \mathcal{Z}) \times \bbR_{++}$. Then, for the three conditions (C$_1$) $A^{\sfT}A - \mu \mathfrak{L}^{\sfT}B^{\sfT}B\mathfrak{L} \succeq \rmO_{\mathcal{X}}$, (C$_2$)
    $J_{\Psi_B \circ \mathfrak{L}} \in \Gamma_0(\mathcal{X})$ for any $y \in \mathcal{Y}$,
    and (C$_3$) $J_{\Psi_B \circ \mathfrak{L}}^{(0)}:=\frac{1}{2}\|A\cdot\|_{\calY}^2+\mu \Psi_B\circ \mathfrak{L} \in \Gamma_0(\calX)$,
    the relation $(C_1) \Rightarrow (C_2) \IFF (C_3)$ holds.
  \end{enumerate}
  In particular, if $\Psi$ is a certain norm, say $\opnorm{\cdot}$, over the vector space $\mathcal{Z}$, these properties are enhanced as:
  \begin{enumerate}[(a)]
  \item[\rm (a')] $(\opnorm{\cdot})_B\circ \mathfrak{L}(x) = \opnorm{\mathfrak{L}x} - \frac{1}{2}\|B\mathfrak{L}x\|_{\widetilde{\mathcal{Z}}}^2$ if and only if $\opnorm{B^{\sfT}B\mathfrak{L}x}_* \leq 1$, where $\opnorm{\cdot}_*\colon \mathcal{Z} \to \mathbb{R}\colon v \mapsto \sup_{w \in \mathcal{Z}\colon \opnorm{w}\leq 1}|\langle w, v  \rangle|$ is the dual norm\footnote{
      See, e.g., \cite[Def. 5.4.12]{horn2012matrix}, \cite[Def. 2.10.3]{kreyszig1978introductory}, and \cite[Example 3.26]{Boyd-Vandenberghe04}.
      } of $\opnorm{\cdot}$ .
  \item[\rm (b')] The equivalence $(C_1) \IFF (C_2) \IFF (C_3)$ holds.
  \end{enumerate}
\end{proposition}
\begin{proof}
  See \ref{app:prop1}.
%See \ref{app:conv} for selected elements, of convex analysis and optimization, utilized in the following proof. \\
\end{proof}

% \begin{fact}[Properties of GMC penalty function \cite{Selesnick2017-sparse}]
% 	\label{fact:GMCprops}
% 	The GMC penalty function $\Psi_B$ in Definition~\ref{def:GMC} has following properties:
% 	\begin{enumerate}[(a)]
% 		\item ({\cite[Corollary~2]{Selesnick2017-sparse}})
% 		$\Psi_B(\bmx) = \|\bmx\|_1 - \frac{1}{2}\|B\bm{x}\|^2$ if and only if $\|B^{\sfT}B\bm{x}\|_{\infty} \leq 1$. \label{item:nonconvexity}
% 		\item ({\cite[Theorem~1]{Selesnick2017-sparse}})
% 		Let $(A, B, \bmy, \mu) \in \bbR^{m\times n} \times \bbR^{q \times n} \times \bbR^{m} \times \bbR_{++}$ and suppose that $A^{\sfT}A - \mu B^{\sfT}B \succeq \rmO_{n}$. Then $J_{\Psi_B}(\cdot) := \frac{1}{2} \|\bmy - A(\cdot) \|^2 + \mu \Psi_B (\cdot)\in \Gamma_0(\bbR^n)$. \label{item:preserving}
% 		% \item ({\cite[Lemma~2]{Selesnick2017-sparse}})
% 		% If $B \in \bbR^{q \times n}$ is of $\rank(B) = q$, then
% 		% \begin{equation*}
% 		% \Psi_B(\bmx) = \| \bmx \|_1 - {}^{1}(\funcd \circ B^{\dagger}) \circ B \bmx,
% 		% \end{equation*}
% 		% where $\funcd \colon \bbR^n \to \bbR_{+}$ is defined as
% 		% \begin{equation}
% 		% \funcd(\bmx) := \min_{\bmw \in \Null(B)} \| \bmx - \bmw \|_1, \label{eq:def_d}
% 		% \end{equation}
% 		% and ${}^{1}(\funcd \circ B^{\dagger})$ is the Moreau envelope of $\funcd \circ B^{\dagger}$ of index $\gamma = 1$ (see \eqref{eq:Moreauenvelope}).
% 		% \label{item:Moreau}
% 	\end{enumerate}
% \end{fact}

\begin{remark}\label{rem-prop:exJ_cond}
  %For $\Psi=\opnorm{\cdot}$, $B^{\sfT}B\mathfrak{L}x \in \operatorname{argmin}(\Psi^*) \IFF \opnorm{B^{\sfT}B\mathfrak{L}x}_* \leq 1$ holds.
  (i) Proposition~\ref{prop:exJ_cond}(a') for special case $(\mathcal{X},\mathcal{Z}, \opnorm{\cdot},\mathfrak{L})=(\mathbb{R}^n, \mathbb{R}^n, \|\cdot\|_1, \rmId)$ reproduces {\cite[Corollary~2]{Selesnick2017-sparse}} (i.e. $(\|x\|_1)_B = \|\bmx\|_1 - \frac{1}{2}\|Bx\|_{\widetilde{\mathcal{Z}}}^2$ if and only if $\|B^{\sfT}Bx\|_{\infty} \leq 1$) because the dual norm of $\|\cdot \|_1$ is $\|\cdot\|_{\infty}$. For the whole shape of the graph of $(\|\cdot\|_1)_B$,
  see the graphs in \cite[Figs. 3, 8, and 9]{Selesnick2017-sparse} of the GMC penalty.  \\
  (ii) Proposition~\ref{prop:exJ_cond}(b') specialized for $\mu >0$ and $A =\rmO_{\mathcal{B}(\mathcal{X}, \mathcal{Y})}$ yields
  \begin{equation}
    \label{eq:AO}
  B = \rmO_{\mathcal{B}(\mathcal{Z}, \widetilde{\mathcal{Z}})} \left(\IFF - B^{\sfT}B \succeq \rmO_{\mathcal{Z}}\right) \IFF (\opnorm{\cdot})_B = \opnorm{\cdot} \in \Gamma_0(\mathcal{X}).
\end{equation}
(iii) (C$_1$) $\Rightarrow$ (C$_2$) is found in {\cite[Theorem~1]{Selesnick2017-sparse}} but only for special case $\Psi=\|\cdot\|_1$
(compare this with Proposition~\ref{prop:exJ_cond}(b) and Proposition~\ref{prop:exJ_cond}(b')).

%is not only a generalization but also a refinement of {\cite[Theorem~1]{Selesnick2017-sparse}}.
% because it implies
% \begin{equation}
% \label{eq:cond_iff}
% A^{\sfT}A - \mu B^{\sfT}B \succeq \rmO_{n} \Leftrightarrow \frac{1}{2} \|\bmy - A(\cdot) \|_{\mathcal{Y}}^2 + \mu (\|\cdot\|_1)_B \in \Gamma_0(\bbR^n).
% \end{equation}
% Note that the equivalence \eqref{eq:cond_iff} specialized for $A=O_{m \times n}$ yields
% \begin{equation*}
%   B = \rmO_{q \times n} [\IFF - B^{\sfT}B \succeq \rmO_{n}] \IFF (\|\cdot\|_1)_B \in \Gamma_0(\bbR^n),
% \end{equation*}
%   i.e.,
% only the $\ell_1$ norm $\|\cdot\|_1=(\|\cdot\|_1)_{\rmO_{q \times n}} $ falls in the intersection of the set of all GMC penalty functions and the set of all convex penalty functions. % (see Figure~\ref{fig:umbrella}).
\end{remark}

\subsection{A proximal splitting algorithm for the LiGME model and its global convergence property}
\label{sec:algorithm}
Our target is the following convex optimization problem:
\begin{problem}[LiGME model in Definition~\ref{def:GMC} under an overall convexity condition]
	\label{prob:optim}
        Assume that $\Psi \in \Gamma_0(\calZ)$ satisfies the even symmetry\footnote{
          In this case, for $B=\rmO_{\cal{Z}}$, we have 
          $\Psi_B(\cdot)=\Psi(\cdot) - \Psi(0_{\calZ})$ (See also \eqref{eq:AO}).
          } $\Psi\circ(-\rmId)=\Psi$ and is \emph{proximable}, i.e., $\operatorname{prox}_{\gamma \Psi}$ is available  as a computable operator for every $\gamma \in \mathbb{R}_{++}$.
        Then 
        for $(A, \mathfrak{L}, B, y,\mu) \!\in\! \mathcal{B}(\mathcal{X}, \mathcal{Y}) \times \mathcal{B}(\mathcal{X}, \mathcal{Z}) \times \mathcal{B}(\mathcal{Z}, \widetilde{\mathcal{Z}}) \times \mathcal{Y} \times \bbR_{++}$ satisfying $A^{\sfT}A - \mu \mathfrak{L}^{\sfT}B^{\sfT}B\mathfrak{L} \!\succeq\! \rmO_{\mathcal{X}}$,
	\begin{equation}  
          \text{find } x^{\star} \in \mathcal{S} := \argmin_{x \in \mathcal{X}} J_{\Psi_B \circ \mathfrak{L}}(x).
          %\text{ [see the LiGME model \eqref{eq:J} for $J_{\Psi_B \circ \mathfrak{L}}$]}.
	\label{prob:involve}
	\end{equation}
\end{problem}
                                                                                                           %                                                                                                            We observe that existing proximal-splitting techniques in \cite{Bauschke2011-convex} cannot be applied directly to Problem~\ref{prob:optim} because $\Psi_B$ is nonsmooth and nonconvex.

We will use a technical lemma below.
\begin{lemma}
  \label{lem:qualification}
  In Definition~\ref{def:GMC}, if $\Psi$ satisfies $\Psi\circ(-\rmId)=\Psi$, we have
  \begin{align}
    \label{eq:qualification}
    0_{\mathcal{Z}} \in \operatorname{ri}\left(\dom \left(\left(\Psi+ \frac{1}{2} \|B \cdot \|_{\widetilde{\mathcal{Z}}}^2\right)^*\right) - \ran(B^*)  \right).
  \end{align}
\end{lemma}
\begin{proof}
See \ref{app:qualification}.
\end{proof}

In the next theorem, (a) and (b) show that the set $\mathcal{S}$ of all globally optimal solutions of Problem~1 can be expressed in terms of the fixed-point set of a computable averaged nonexpansive operator in a certain real Hilbert space,
and (c) presents an iterative algorithm, for Problem~\ref{prob:optim}, based on the Krasnosel'ski{\u{\i}}-Mann iteration in Fact~\ref{fact:KM}.

\begin{theorem}[Nonexpansive operator $\Tlcp$ and iterative algorithm for Problem~\ref{prob:optim}]
  \label{def:Tprop}
  \ \\
  In Problem~\ref{prob:optim}, let
  $(\Hlb:= \mathcal{X} \times \mathcal{Z} \times \mathcal{Z}, \langle \cdot, \cdot \rangle_{\Hlb}, \| \cdot\|_{\Hlb})$ be a real Hilbert space whose inner product $\langle \cdot, \cdot \rangle_{\Hlb}$ is defined as the one for the product space in Example~\ref{ex-regularizations}(d), and 
  % $\calR := \mathcal{X} \times \mathcal{Z} \times \mathcal{Z}$ be a real Hilbert space designed in a way similar to the product space in Example~\ref{ex-regularizations}(d)
  % equipped with $\langle \cdot, \cdot \rangle_{\mathcal{H}}:=
  % \langle \cdot, \cdot \rangle_{\mathcal{X}}
  % +\langle \cdot, \cdot \rangle_{\mathcal{Z}}
  % +\langle \cdot, \cdot \rangle_{\mathcal{Z}}$
  % and 
  define $\Tlcp : \calR \to \calR : (\bmx,\bmv,\bmw) \mapsto (\bmxi,\bmzeta,\bmeta)$, with $(\sigma, \tau) \in \bbR_{++} \times \bbR_{++}$, by
	\begin{align}
	\bmxi &:= \left[\rmId - \frac{1}{\sigma}(A^{\sfT}A-\mu \mathfrak{L}^{\sfT}B^{\sfT}B\mathfrak{L})\right]\bmx -\frac{\mu}{\sigma}\mathfrak{L}^{\sfT}B^{\sfT}B\bmv-\frac{\mu}{\sigma}\mathfrak{L}^{\sfT}\bmw+\frac{1}{\sigma}A^{\sfT}\bmy,  \nonumber \\
	\bmzeta &:= \Prox_{\frac{\mu}{\tau} \Psi} \left[ \frac{2\mu}{\tau}B^{\sfT}B\mathfrak{L}\bmxi - \frac{\mu}{\tau} B^{\sfT}B\mathfrak{L}\bmx + \left( \rmId - \frac{\mu}{\tau} B^{\sfT}B \right) \bmv \right], \nonumber \\
          \bmeta &:=  \Prox_{\Psi^*}\left( 2 \mathfrak{L}\bmxi -  \mathfrak{L}\bmx + \bmw \right).  \nonumber
	\end{align}
        Then
	\begin{enumerate}[(a)]
        \item[\rm (a)] the solution set $\mathcal{S}$ of  Problem 1 can be expressed as
          \begin{align}
            \mathcal{S} = \Xi (\Fix(\Tlcp)):=\{\Xi(\bmx^{\star},\bmv^{\star},\bmw^{\star})\in \mathcal{X} \mid (\bmx^{\star},\bmv^{\star},\bmw^{\star}) \in \Fix(\Tlcp) \}
          \end{align}
          with $\Xi : \calR \to \mathcal{X} : (\bmx,\bmv,\bmw) \mapsto \bmx$.
          % if
          % \begin{align}
          %   \label{eq:qualification}
          %   0_{\mathcal{Z}} \in \operatorname{ri}\left(\dom \left(\Psi+ \frac{1}{2} \|B \cdot \|_{\widetilde{\mathcal{Z}}}^2\right)^* - \ran(B^*)  \right).
          % \end{align}
		\item[\rm (b)] Choose $(\sigma, \tau, \kappa) \in \bbR_{++} \times \bbR_{++}  \times (1,\infty)$ satisfying\footnote{
                    For example,
                    \eqref{eq:stepsize_condition} is satisfied by any $\kappa > 1$ and 
%                    the triple $(\sigma, \tau, \kappa)$
%                    \in  \bbR_{++} \times \bbR_{++} \times (1,\infty)$ chosen as
			\begin{equation*}
			\left[ \begin{array}{l}
			\sigma := \left\|\frac{\kappa}{2} A^{\sfT}A + \mu \mathfrak{L}^{\sfT}\mathfrak{L}\right\|_{\rm op} + (\kappa - 1), \\
			\tau := (\frac{\kappa}{2} + \frac{2}{\kappa})\mu \|B\|_{\rm op}^2 + (\kappa - 1).%, \\
%			\kappa -1> \epsilon,
			\end{array} \right.
                      %\label{eq:stepsize_ex}
			\end{equation*}
			%satisfies \eqref{eq:stepsize_condition}.
		}
		\begin{equation}
		\left[ \begin{array}{l}
		%				\sigma \rmI_n - \frac{\mu^2}{\tau} L^{\sfT}(B^{\sfT}B)^2L - \mu L^{\sfT}L \succ \rmO_n, \\
		%				\sigma \rmI_n - \frac{\kappa}{2} A^{\sfT}A - \mu L^{\sfT}L \succeq \rmO_n, \\
                         \sigma \rmId - \frac{\kappa}{2} A^{\sfT}A - \mu \mathfrak{L}^{\sfT}\mathfrak{L} \succ \rmO_{\mathcal{X}}, \\
                         \tau \geq \left( \frac{\kappa}{2} + \frac{2}{\kappa} \right) \mu \|B\|_{\rm op}^2.
		\end{array} \right.
		\label{eq:stepsize_condition}
		\end{equation}
		Then
		\begin{equation}
		\mathfrak{P} := \begin{bmatrix}
		\sigma \rmId & -\mu \mathfrak{L}^{\sfT} B^{\sfT}B & -\mu \mathfrak{L}^{\sfT} \\
		-\mu B^{\sfT}B\mathfrak{L} & \tau \rmId & \rmO_{\mathcal{Z}} \\
		-\mu \mathfrak{L} & \rmO_{\mathcal{Z}} & \mu \rmId
		\end{bmatrix} \succ \rmO_{\calR}
		\label{eq:S_def}
              \end{equation}
		and $\Tlcp$ is $\frac{\kappa}{2\kappa - 1}$-averaged nonexpansive in the Hilbert space $(\calR, \langle \cdot, \cdot \rangle_{\mathfrak{P}}, \| \cdot\|_{\mathfrak{P}})$.
              \item[\rm (c)] \label{item:thalgorithm} Assume
                %\eqref{eq:qualification} and
                $(\sigma, \tau, \kappa) \in \bbR_{++} \times \bbR_{++} \times (1,\infty)$ satisfies \eqref{eq:stepsize_condition}. Then, for any initial point $(\bmx_0, \bmv_0,\bmw_0) \in \calR$, the sequence $(\bmx_k, \bmv_k, \bmw_k)_{k \in \mathbb{N}} \subset \calR$ generated by
		\begin{equation}
		(\bmx_{k+1}, \bmv_{k+1}, \bmw_{k+1}) = \Tlcp(\bmx_k, \bmv_k, \bmw_k)
		\label{eq:algorithm_1step}
		\end{equation}
		converges weakly to a point $(\bmx^{\star}, \bmv^{\star}, \bmw^{\star}) \in \Fix(\Tlcp)$ and
		\begin{equation*}
		\lim_{k \to \infty} \bmx_k = \bmx^{\star} \in \mathcal{S}.
              \end{equation*}     
	\end{enumerate}
      \end{theorem}
\begin{proof}
See \ref{app:theo1}.
\end{proof}
Detailed description of the algorithm proposed in Theorem~\ref{def:Tprop} is shown in Algorithm~1.

\begin{remark}[Algorithm~1 versus existing algorithms]
  (a) The derivation of Algorithm~1 is inspired by Condat's primal-dual algorithm \cite{Condat2013-primal_dual} and is essentially based on the so-called forward-backward splitting method (see also \eqref{eq:Tlcpex} demonstrating that $\Tlcp$ is a forward-backward operator).
%  The proposed algorithm achieves wider applicability than Condat's primal-dual algorithm which was proposed for minimization of sum of linearly involved convex terms.
  Since Condat's primal-dual algorithm was proposed for minimization of sum of linearly involved convex terms, it is not directly applicable to the LiGME model involving nonconvex functions.
  
  \noindent (b)
  The proposed algorithm in \eqref{eq:algorithm_1step} differs clearly from 
  Combettes-Pesquet primal-dual algorithm \cite{combettes2012primal} which is for monotone inclusion problems and based on the so-called forward-backward-forward splitting method (or Tseng's method), i.e., requires an extra forward step compared with the so-called forward-backward splitting method.
  
  \noindent (c) Vu's primal-dual algorithm \cite{Vu2013-splitting} for monotone inclusion is also based on the so-called forward-backward splitting method. However, to the best of the authors' knowledge,
  the strongly monotone assumption (of $D_i$) in \cite[Problem 1.1]{Vu2013-splitting} prevents from applying directly the Vu's primal-dual algorithm to Problem~1 if $\operatorname{null}(B)\not = \{0_{\mathcal{Z}}\}$. Algorithm~1 is applicable to general $B \in \mathcal{B}(\mathcal{Z}, \widetilde{\mathcal{Z}})$.

\end{remark}

\begin{table}[t]
	\centering
	\begin{small}
		\begin{tabular}{p{12cm}}
			\hline \hline \\[-4mm]
			\hspace{-2mm} {\normalsize {\bf Algorithm 1} for Problem~\ref{prob:optim}.} \\
			\hline
                  \hspace{-2mm} Choose ($\bmx_0$, $\bmv_0$, $\bmw_0$) $\in \calR(=\mathcal{X} \times \mathcal{Z} \times \mathcal{Z})$. \\
			\hspace{-2mm} Let $(\sigma, \tau, \kappa) \in \mathbb{R}_{++} \times \mathbb{R}_{++} \times (1,\infty)$ satisfying \eqref{eq:stepsize_condition}.
                                                                                                           %                                                                                                            \ \Comment{See the footnote for Theorem~\ref{def:Tprop}(b) for selections of $(\sigma, \tau, \kappa)$.}
                  \\
			\hspace{-2mm} Define $\mathfrak{P}$ as \eqref{eq:S_def}. \\
			\hspace{-2mm} $k \gets 0$. \\
			\hspace{-2mm} {\bf Do} \\
			\hspace{2mm}$\bmx_{k+1}  \gets  \left[\rmId - \frac{1}{\sigma}(A^{\sfT}A-\mu \mathfrak{L}^{\sfT}B^{\sfT}B\mathfrak{L})\right]\bmx_k - \frac{\mu}{\sigma}\mathfrak{L}^{\sfT}B^{\sfT}B\bmv_k-\frac{\mu}{\sigma}\mathfrak{L}^{\sfT}\bmw_k+\frac{1}{\sigma}A^{\sfT}\bmy$ \\[2mm]
			\hspace{2mm}$\bmv_{k+1}  \gets  \Prox_{\frac{\mu}{\tau}\Psi}  \left[  \frac{2\mu}{\tau}B^{\sfT}B\mathfrak{L}\bmx_{k+1} - \frac{\mu}{\tau} B^{\sfT}B\mathfrak{L}\bmx_k + \left( \rmId - \frac{\mu}{\tau} B^{\sfT}B \right) \bmv_{k}  \right]$ \\
                  %\Comment{$\Prox_{\frac{\mu}{\tau}\|\cdot\|_1}$ can be calculated by \eqref{eq:soft_def}.} \\[2mm]
			\hspace{2mm}$\bmw_{k+1}  \gets  \Prox_{\Psi^{\ast}} \left( 2\mathfrak{L}\bmx_{k+1} - \mathfrak{L}\bmx_k + \bmw_k \right)$
			% \Comment{$\Prox_{\|\cdot\|_1^{\ast}}$ can be calculated by \eqref{eq:proj_def}.}
                  \\[2mm]
			\hspace{2mm}$k \gets k+1$ \\
			\hspace{-2mm} {\bf while} $\|(\bmx_k, \bmv_k, \bmw_k) - (\bmx_{k-1}, \bmv_{k-1}, \bmw_{k-1})\|_{\mathfrak{P}}$ is not sufficiently small \\
			\hspace{-2mm} {\bf return} $\bmx_k$ \\
			\hline
		\end{tabular}
	\end{small}
	\ \\[-4mm]
\end{table}

\subsection{How to choose $B$ to ensure overall-convexity of $J_{\Psi_B \circ L}$}
\label{sec:choose_B}
Choices of $B$ to guarantee $J_{\Psi_{B} \circ L} \in \Gamma_0(\bbR^n)$ are given, e.g., as follows.
\renewcommand{\sfT}{\top} %\mathsf{T}}
\begin{proposition}[A design of $B$ to ensure the overall-convexity condition in Proposition~\ref{prop:exJ_cond}(b)]
  \label{rem:selectionB}
  In Definition~\ref{def:GMC}, let $(\calX,\calY,\calZ) = (\bbR^{n},\bbR^{m},\bbR^{l})$, $(A, \mathfrak{L}, \mu) \in \bbR^{m \times n} \times \bbR^{l \times n} \times \bbR_{++}$, and $\rank(\mathfrak{L}) = l$. Choose a nonsingular $\tilde{\mathfrak{L}} \in \bbR^{n \times n}$ satisfying\footnote{
    Such a choice is always possible. See Corollary~\ref{col:selectionBmultiple} and
    numerical experiments in four different scenarios in Section~\ref{sec:NumericalExperiment}.
  }$
	[\rmO_{l \times (n-l)}\ \ \rmI_{l}] \tilde{\mathfrak{L}} = \mathfrak{L}$.
	Then
	\begin{equation}
	\label{eq:Btheta}
	B_{\theta} := \sqrt{\theta / \mu} \Lambda^{1/2} U^{\sfT} \in \mathbb{R}^{l \times l}, \quad \theta \in [0,1],
	\end{equation}
	ensures $J_{\Psi_{B_{\theta}} \circ \mathfrak{L}} \in \Gamma_0(\bbR^n)$, where               
        \begin{align}
	\label{eq:tildeAdef}
	[\tilde{A}_1\ \ \tilde{A}_2]:=A(\tilde{\mathfrak{L}})^{-1}
	\end{align}
        and $U\Lambda U^{\sfT} := \tilde{A}_2^{\sfT} \tilde{A}_2 - \tilde{A}_2^{\sfT} \tilde{A}_1 (\tilde{A}_1^{\sfT}\tilde{A}_1)^{\dagger} \tilde{A}_1^{\sfT} \tilde{A}_2 \in \bbR^{l \times l}$ is an eigendecomposition.
      \end{proposition}
      \renewcommand{\sfT}{*}
\begin{proof}
         See \ref{app:prop2}.
       \end{proof}

%Proposition~\ref{rem:selectionB} will be utilized to design concrete $B$ in Section~\ref{sec:NumericalExperiment}.
       The next corollary presents a way of design $B^{\langle i \rangle} \in \mathcal{B}(\calZ,\widetilde{\calZ}) \ (i=1,2,\ldots, \mathcal{M})$ in Example~\ref{prop:productspce} for $\Psi_B\circ \mathfrak{L}$ in \eqref{eq:ex3PBL} to ensure the overall-convexity condition in Proposition~\ref{prop:exJ_cond}(b).

%       The following proposition shows that, in the case of using multiple LiGME penalties described in Example~3, the rank condition $\rank({\mathfrak{L}}) = l$ in Proposition~\ref{rem:selectionB} can be relaxed to $\rank({\mathfrak{L}_i}) = l_i$ ($i = 1,2,\dots,\mathcal{M}$).
\renewcommand{\sfT}{\top} %\mathsf{T}}
\begin{corollary}[A design of $B^{\langle i \rangle}$ in Example~\ref{prop:productspce} to ensure the overall-convexity condition in Proposition~\ref{prop:exJ_cond}(b)]
	\label{col:selectionBmultiple}
	In Example~\ref{prop:productspce}, let $(\calX,\calY,\calZ_i) = (\bbR^{n},\bbR^{m},\bbR^{l_i})$, $(A, \mathfrak{L}_i, \mu) \in \bbR^{m \times n} \times \bbR^{l_i \times n} \times \bbR_{++}$,
        and $\rank(\mathfrak{L}_i) = l_i$ $(i = 1, 2, \dots, \mathcal{M})$.
%        
                                                                                                           %                                                                                                            let $A \in \bbR^{m \times n}$,  $\mathcal{X} = \bbR^n$, and $\mathcal{Z}_i = \bbR^{l_i}$ $(i = 1, 2, \dots, \mathcal{M})$, suppose $\rank(\mathfrak{L}_i) = l_i$ $(i = 1, 2, \dots, \mathcal{M})$.
        Choose nonsingular $\tilde{\mathfrak{L}_i} \in \bbR^{n \times n}$ satisfying $
	[\rmO_{l_i \times (n-l_i)}\ \ \rmI_{l_i}] \tilde{\mathfrak{L}_i} = \mathfrak{L}_i \ (i = 1,2,\dots,\mathcal{M})$
        and $\omega_i \in \mathbb{R}_{++}$ ($i=1,2,\ldots, \mathcal{M}$) satisfying $\sum_{i=1}^{\mathcal{M}} \omega_i = 1$.
        For each $i=1,2,\ldots, \mathcal{M}$, apply Proposition~\ref{rem:selectionB} to $\left(\sqrt{\frac{\omega_i}{\mu}} A, \mathfrak{L}_i, \mu_i\right)$
        to obtain $B_{\theta_i}^{\langle i \rangle} \in \mathbb{R}^{l_i \times l_i}$ satisfying \\
$\left(\sqrt{\frac{\omega_i}{\mu}} A\right)^{\sfT}\left(\sqrt{\frac{\omega_i}{\mu}} A\right) - \mu_i \mathcal{L}_i^{\sfT} {B_{\theta_i}^{\langle i \rangle}}^{\sfT} B_{\theta_i}^{\langle i \rangle} \mathfrak{L}_i  \succeq \rmO_{n \times n}$. Then 
%
%Let, for every $i \in \{1, 2, \dots, \mathcal{M}\}$, $B_{\theta}^{\langle i \rangle}$ be given by applying Proposition~\ref{rem:selectionB} to $\left(\sqrt{\frac{\omega_i}{\mu}} A, \mathfrak{L}_i, \mu_i\right)$.
                                                                                                           %                                                                                                            Then
$B_{\theta}\colon \bbR^{l_1} \times \bbR^{l_2} \times \ldots \times~\bbR^{l_\mathcal{M}} \to \bbR^{l_1} \times \bbR^{l_2} \times \ldots \times\bbR^{l_\mathcal{M}}  \colon (z_1,\ldots,z_M)\mapsto \left(\sqrt{\mu_1}B_{\theta_1}^{\langle 1\rangle} z_1,\ldots, \sqrt{\mu_{\mathcal{M}}}B_{\theta_{\mathcal{M}}}^{\langle \mathcal{M}\rangle}z_{\mathcal{M}}\right)$
%$B_{\theta} := \sqrt{\mu_1}B_{\theta_i}^{\langle 1 \rangle} \oplus \sqrt{\mu_2}B_{\theta_i}^{\langle 2 \rangle} \oplus \cdots \oplus \sqrt{\mu_{\mathcal{M}}}B_{\theta_i}^{\langle \mathcal{M} \rangle}$
ensures $J_{\Psi_{B_{\theta}} \circ \mathfrak{L}} \in \Gamma_0(\bbR^n)$.
\end{corollary}
\begin{proof}
  Verified by
	\begin{align*} 
          A^{\sfT}A - \mu \mathfrak{L}^{\sfT} B_{\theta}^{\sfT} B_{\theta} \mathfrak{L} &= A^{\sfT}A - \mu \sum_{i=1}^{\mathcal{M}} \mu_i \mathfrak{L}_i^{\sfT} {B_{\theta_i}^{\langle i \rangle}}^{\sfT} B_{\theta_i}^{\langle i \rangle} \mathfrak{L}_i  \\
          &= \mu \sum_{i=1}^{\mathcal{M}} \left( \frac{\omega_i}{\mu}A^{\sfT}A - \mu_i \mathcal{L}_i^{\sfT} {B_{\theta_i}^{\langle i \rangle}}^{\sfT} B_{\theta_i}^{\langle i \rangle} \mathfrak{L}_i \right) \succeq \rmO_{n \times n}.
	\end{align*}
\end{proof}
\renewcommand{\sfT}{*}

\section{Numerical Experiments}
\label{sec:NumericalExperiment}
To demonstrate the effectiveness of the proposed penalties (LiGME penalties) and the proposed algorithm for the LiGME model (see Algorithm~1), we present numerical experiments in four sparsity-rank-aware signal processing scenarios: (i) recovering a piecewise constant 1-d signal, (ii) deburring a piecewise constant image, (iii) filling missing entries of a low-rank matrix, which is a task so-called the matrix completion, (iv) filling missing entries of low-rank as well as piecewise constant matrix by handling two different LiGME penalties.

\renewcommand{\sfT}{\top}
\subsection{Piecewise constant 1-d signal recovery}
\label{section:exam1}
In this section, we present a numerical experiment in a scenario of edge-preserving signal recovery by considering Problem~\ref{prob:optim} with $(\calX,\calY,\calZ) = (\bbR^{N},\bbR^{M},\bbR^{N-1})$, $(N,M) := (128, 100)$, $\Psi = \| \cdot \|_1$, and $\frakL$ being the first order difference operator, i.e.,
\begin{equation}
	\label{def:D}
	\frakL = D := \begin{bmatrix}
		-1&1&&\\
		&\ddots&\ddots&\\
		&&-1&1
	\end{bmatrix} \in \bbR^{(N-1) \times N}.
\end{equation}
In this experiment, entries of $A \in \bbR^{M \times N}$ are drawn from i.i.d. zero-mean white Gaussian noise with unit variance.
The observation $\bmy \in \bbR^{M}$ is generated by $\bmy = A \bmx^{\star} + \varepsilon$, where $\bmx^{\star} \in \bbR^{N}$ is a piecewise constant signal (Figure~\ref{fig:exam1_estimates}: dotted) and $\varepsilon \in \bbR^{M}$ is additive white Gaussian noise. The signal-to-noise ratio (SNR) is -5dB, which is defined as
\begin{equation}
	\label{def:SNR}
	\text{SNR: } 10 \log_{10} \frac{\|\bmx^{\star}\|_{\calX}^2}{\|\varepsilon\|_{\calY}^2} \ \mathrm{[dB]}.
\end{equation}
We compared minimizers of Problem~1, estimated by Algorithm 1, with two penalties: one is the standard convex total variation (TV), i.e., $(\| \cdot \|_1)_{B_{0}} \circ D=(\| \cdot \|_1)_{\rmO_{\calZ}} \circ D = \|\cdot\|_1 \circ D$, the other is a nonconvex LiGME penalty $(\|\cdot\|_1)_{B_{\theta}} \circ D$ whose $B_{\theta} \in \bbR^{(N-1) \times (N-1)}$ is obtained by Proposition~\ref{rem:selectionB} with
\begin{equation}
	\tilde{L} = \tilde{D} := {\left[\,e_1\,|\,D^{\sfT}\,\right]}^{\sfT} \in \bbR^{N \times N},
	\label{eq:choice_tildeD}
\end{equation}
where $\theta = 0.99$ and $e_1=(1,0,\ldots,0)^{\sfT} \in \bbR^N$. Algorithm~1 with $\kappa = 1.001$ and $(\sigma,\tau)$ given in the footnote for Theorem~\ref{def:Tprop}(b) is applied to the minimization problems, where the common initial estimate is set as $(x_0, v_0, w_0)=(0_{\calX}, 0_{\calZ}, 0_{\calZ})$ for all experiments.

In Algorithm~1, $\Prox_{\gamma \| \cdot \|_1}$ for $\gamma \in \bbR_{++}$ can be calculated by the soft-thresholding whose $i$-th component is 
\begin{equation}
	\label{eq:Proxl1}
	\left[\Prox_{\gamma \| \cdot \|_1} \right]_i\colon \bbR^{N-1} \to \bbR \colon z=(z_1,\ldots, z_{N-1})^{\top} \mapsto
         \begin{cases}
		0, & \text{if } |z_i| \leq \gamma, \\
		(|z_i| - \gamma)\frac{z_i}{|z_i|} & \text{otherwise}.
	\end{cases}
\end{equation}

\begin{figure}[ht]
	\centering
	\subfloat[][]{\includegraphics[clip,width=0.45\hsize]{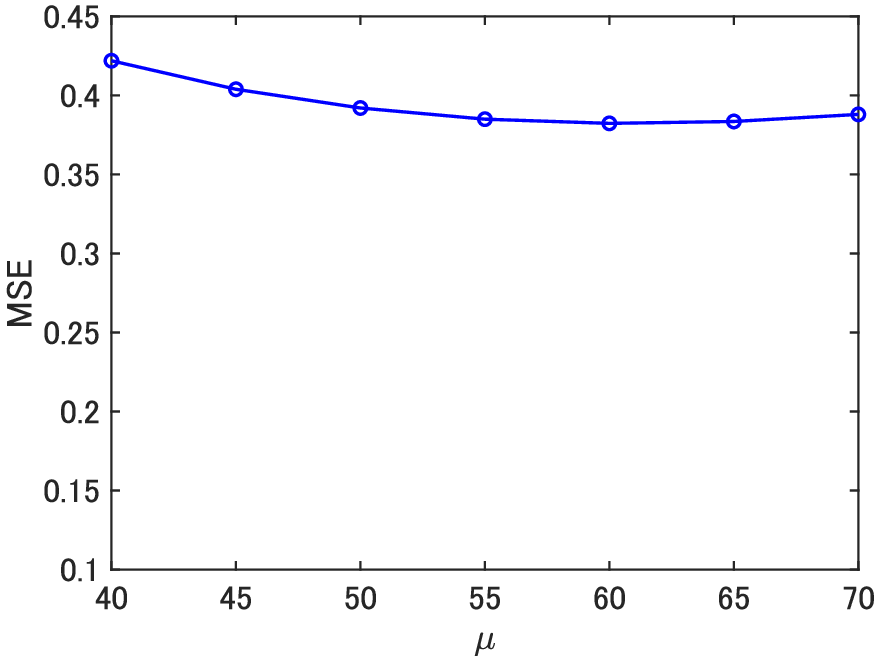}\label{subfig:exam1_tuneTV}} \quad
	\subfloat[][]{\includegraphics[clip,width=0.45\hsize]{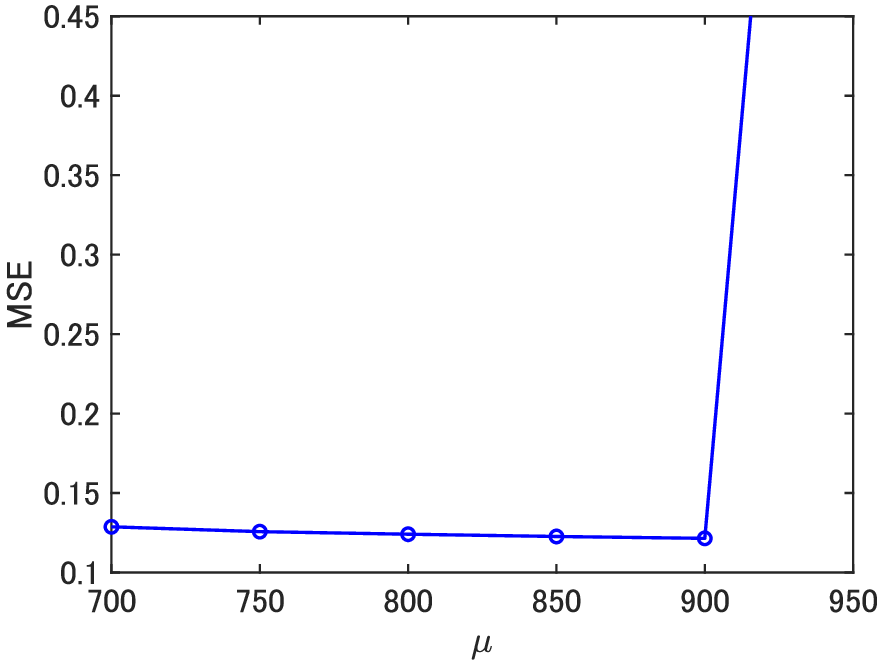}\label{subfig:exam1_tuneLiGME}}
	\caption{{MSE versus $\mu$ in Problem~1 at $k = 15,000$ iteration for (a) the standard convex TV penalty $\| \cdot \|_1 \circ D$ and (b) LiGME penalty $(\| \cdot \|_1)_{B_{\theta}} \circ D$.}}
	\label{fig:exam1_tunes}
\end{figure}
\begin{figure}[ht]
	\centering
	\includegraphics[clip,width=0.60\hsize]{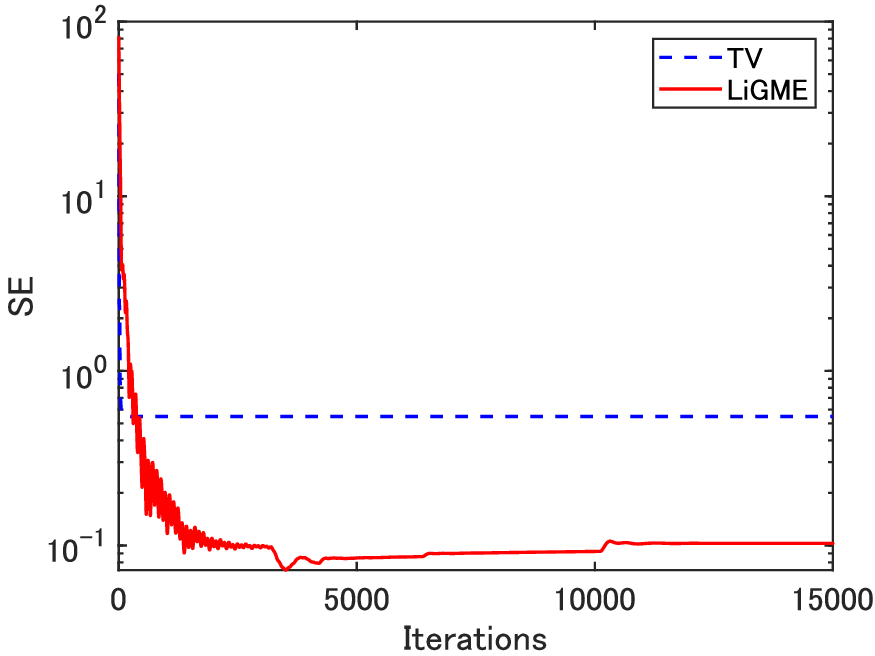}
	\caption{SE versus iterations for TV (dotted blue) and LiGME (solid red).}
	\label{fig:exam1_trace}
\end{figure}
\begin{figure}[ht]
	\centering
	\includegraphics[clip,width=0.60\hsize]{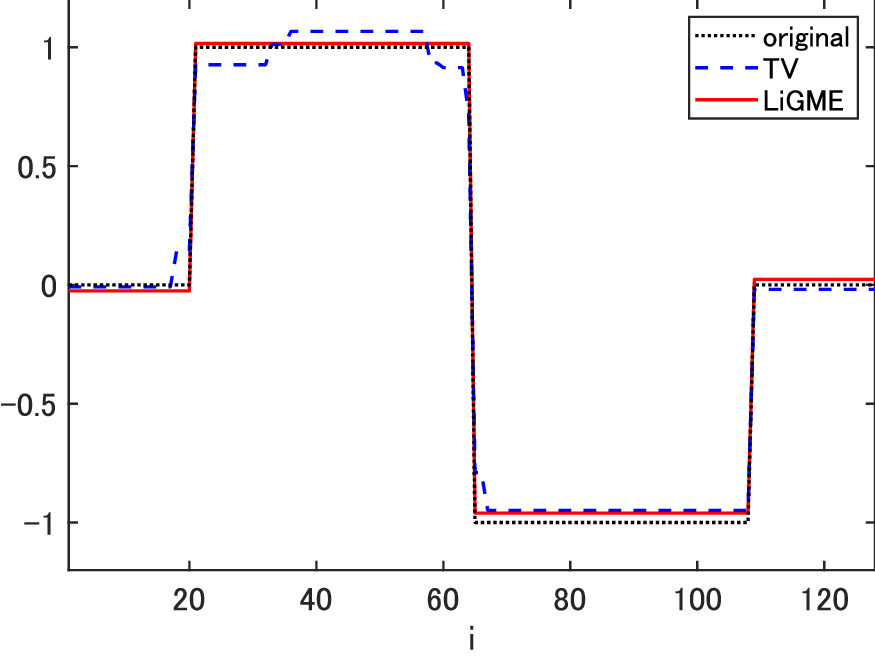}
	\caption{Entries in original piecewise constant signal ($\bmx^{\star}$: dotted black), recovered by the TV penalty ($\bmx_{\mathrm{TV}}$: dashed blue), and by the LiGME penalty ($\bmx_{\mathrm{LiGME}}$: solid red).}
	\label{fig:exam1_estimates}
\end{figure}
\begin{figure}[ht]
	\centering
	\includegraphics[clip,width=0.60\hsize]{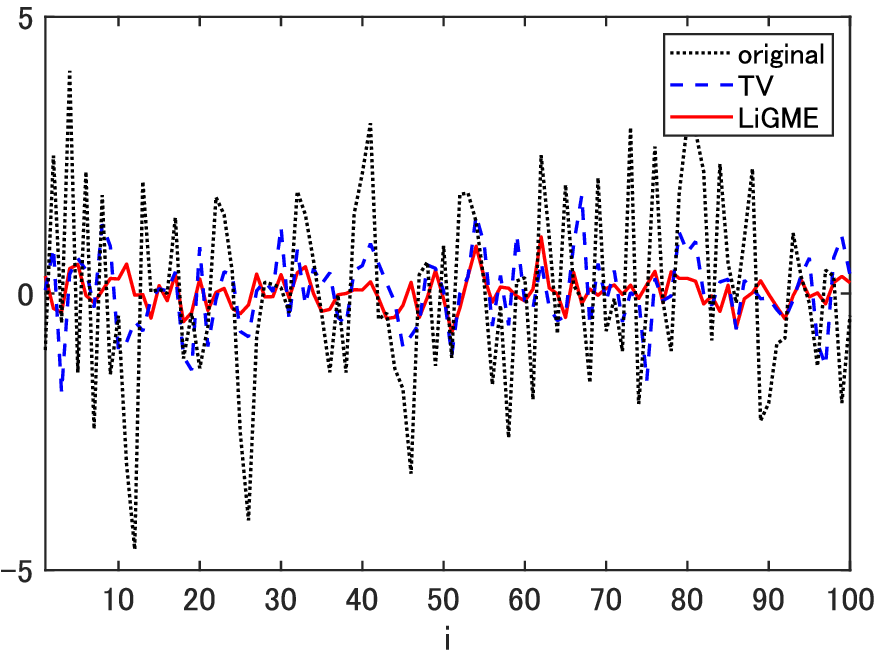}
	\caption{Entries in $\bmy - A\bmx^{\star}$ (dotted black), $A\bmx_{\mathrm{TV}} - A\bmx^{\star}$ (dashed blue), and $A\bmx_{\mathrm{LiGME}} - A\bmx^{\star}$ (solid red), for $\bmx^{\star}$, $\bmx_{\mathrm{TV}}$, and $\bmx_{\mathrm{LiGME}}$ in Figure~\ref{fig:exam1_estimates}.}
	\label{fig:exam1_suppress}
\end{figure}

Figure~\ref{fig:exam1_tunes} shows dependency of recovering performance on the parameter $\mu$ in Problem~\ref{prob:optim}.
The performance is measured by mean squared error (MSE) defined as the average of 
\begin{align}
\label{eq:SE}
\text{squared error (SE): \ } \| \bmx_k - \bmx^{\star}  \|_{\calX}^2
\end{align}
over $100$ independent realizations of the additive noise.
From Figure~\ref{fig:exam1_tunes}, we can see that (i) the best weights of the penalties are respectively $\mu_{\mathrm{TV}} := 60$ for $\| \cdot \|_1 \circ D$ and $\mu_{\mathrm{LiGME}} := 900$ for $(\| \cdot \|_1)_{B_{\theta}} \circ D$ and (ii) the estimation by LiGME penalty with $\mu_{\mathrm{LiGME}}$ outperforms the standard convex TV penalty with $\mu_{\mathrm{TV}}$ in the context of MSE.

Figure~\ref{fig:exam1_trace} shows dependency of the SE on the number of iterations under weights $(\mu_{\mathrm{TV}}, \mu_{\mathrm{LiGME}})$.
The accuracy of the approximation by the LiGME penalty becomes higher than the TV penalty after 400 iterations and SE for LiGME reaches 18.8\% of SE for TV in the end.

Figure~\ref{fig:exam1_estimates} shows the original signal and recovered signals by the penalties at $15,000$ iteration.
The estimation by LiGME $(\| \cdot \|_1)_{B_{\theta}} \circ D$ restores much more successfully the sharp edges than the standard convex TV, which also results in efficient noise suppression at 15,000 iteration depicted in Figure~\ref{fig:exam1_suppress}.

\subsection{Piecewise constant image deblurring}
\label{section:exam2}
We present a numerical experiment in a scenario of image deblurring for piecewise constant  $N$-by-$N$ image by considering Problem~\ref{prob:optim} and Example~3 with $(\mathcal{M}, \calX,\calY, \calZ_1,\calZ_2) = (2, \bbR^{N^2},\bbR^{N^2}, \bbR^{N(N-1)},\bbR^{N(N-1)})$, $N = 16$, $\Psi^{\langle 1 \rangle} = \Psi^{\langle 2 \rangle} = \| \cdot \|_1$, $\mu_1 = \mu_2 = 1$, and $\frakL = \bar{D} := [D_{\mathrm{V}}^{\top}, D_{\mathrm{H}}^{\top}]^{\top}$,
where the vertical difference operator $D_V \in \bbR^{N(N-1) \times N^2}$ and the horizontal difference operator $D_H \in \bbR^{N(N-1) \times N^2}$ are respectively defined as
\begin{equation}
	\label{eq:defD2d}
	\hspace{-1cm}\hspace{-1cm} D_{\mathrm{V}} := \begin{bmatrix}
		D &&&\\
		&D&&\\
		&&\ddots&\\
		&&&D
	\end{bmatrix}, \quad D_{\mathrm{H}} := \begin{bmatrix}
		-1 & \bmzero_{N-1}^{\top} & 1 && \\
		& \ddots & \ddots & \ddots & \\
		&& -1 & \bmzero_{N-1}^{\top} & 1
	\end{bmatrix}
\end{equation}
with $D \in \bbR^{(N-1) \times N}$ in \eqref{def:D}.
The blur matrix\footnote{
	The blur matrix used in this experiment is more ill-conditioned than the random matrix used in Section~4.1.
	The condition number, i.e., the ratio of the maximum singular value to the minimum singular value, of the blur matrix in \eqref{eq:blurmatrix} is about 593, and of the random matrix is about 12.4.
}
$A \in \bbR^{N^2 \times N^2}$ is designed by
\begin{equation}
  \label{eq:blurmatrix}
	A = \bar{A} \otimes \bar{A},
\end{equation}
where $\otimes$ denotes the Kronecker product and the $(i,j)$-entry of the matrix $\bar{A} \in \bbR^{N \times N}$ is given by
\begin{equation}
	\bar{A}_{i, j} := \begin{cases}
		\frac{1}{\sqrt{1.62\pi}} \exp \left( - \frac{{|i-j|}^2}{1.62} \right), & \text{if } |i-j| < 6, \\
		0, & \text{otherwise}.
	\end{cases}
\end{equation}
The observation $\bmy \in \bbR^{N^2}$ (Figure~\ref{fig:exam2_estimates}(b)) is generated by $\bmy = A \bmx^{\star} + \varepsilon$, where $\bmx^{\star} \in \bbR^{N^2}$ is given by the vectorization\footnote{
	The \emph{vectorization} of a matrix (or an image) is the mapping:
	\begin{equation*}
          %\label{eq:defvec}
		\mathrm{vec} \colon \bbR^{m \times n} \to \bbR^{mn} \colon A \mapsto [a_1^{\top}, \cdots, a_n^{\top}]^{\top},
	\end{equation*}
	where, for $i \in \{1, \cdots, n\}$, $a_i \in \bbR^{m}$ is the $i$-th column vector of $A$. The inverse mapping of the \emph{vectorization} $\mathrm{vec}$ is denoted by $\mathrm{vec}^{-1}\colon  \bbR^{mn} \to \bbR^{m \times n}$.
        
} of a piecewise constant image (Figure~\ref{fig:exam2_estimates}(a)) and $\varepsilon \in \bbR^{N^2}$ is additive white Gaussian noise. The signal-to-noise ratio (SNR) defined in \eqref{def:SNR} is 20dB.
We compared minimizers of Problem~1, estimated by Algorithm 1, with two penalties: one is the anisotropic TV, i.e.,
$$
%(\|\cdot\|_1)_{B_{0}} \circ \bar{D} =
(\|\cdot\|_1)_{\rmO_{\calZ}} \circ D_{\mathrm{V}} + (\|\cdot\|_1)_{\rmO_{\calZ}} \circ D_{\mathrm{H}} = \|\cdot\|_1 \circ D_{\mathrm{V}} + \|\cdot\|_1 \circ D_{\mathrm{H}},
$$
the other is a LiGME penalty $(\|\cdot\|_1)_{B_{\theta}} \circ \bar{D}$ whose $B_{\theta}=\begin{bmatrix}
  B_{\theta_1} & \rmO_{N(N-1)} \\
  \rmO_{N(N-1)} & B_{\theta_2} 
\end{bmatrix}
\in \bbR^{2N(N-1) \times 2N(N-1)}$ is obtained by Corollary~\ref{col:selectionBmultiple} with $\theta_1=\theta_2 = 0.99$, $\omega_1=\omega_2=1/2$, and $(\tilde{\mathfrak{L}}_1, \tilde{\mathfrak{L}}_2)$ given as
\begin{equation}
\label{eq:choices_tildeD2}
\tilde{\mathfrak{L}}_1 = \tilde{D}_{\mathrm{V}} := \begin{bmatrix}
E \\
D_{\mathrm{V}}
\end{bmatrix} \in \bbR^{N^2 \times N^2}, \quad \tilde{\mathfrak{L}}_2 = \tilde{D}_{\mathrm{H}} := \begin{bmatrix}
\rmI_{N} & \rmO_{N \times N(N-1)} \\
\multicolumn{2}{c}{D_{\mathrm{H}}}
\end{bmatrix} \in \bbR^{N^2 \times N^2},
\end{equation}
%\begin{equation}
%	\label{eq:choice_tildeD2}
%	\tilde{\frakL} = \tilde{\bar{D}} := \left[ \tilde{D}_\mathrm{V}^{\sfT} \ \tilde{D}_{\mathrm{H}}^{\sfT} \right]^{\sfT},
%\end{equation}
%where 
%\begin{equation}
%	\tilde{D}_{\mathrm{V}} := \begin{bmatrix}{c}
%		E \\
%		D_{\mathrm{V}}
%	\end{bmatrix}, \quad \tilde{D}_{\mathrm{H}} := \begin{bmatrix}{cc}
%		\rmI_{N} & \rmO_{N \times N(N-1)} \\
%		\multicolumn{2}{c}{D_{\mathrm{H}}}
%	\end{bmatrix},
%\end{equation}
where the $(i,j)$-entry of $E \in \bbR^{N \times N^2}$ is defined as
\begin{equation}
	E_{i,j} := \begin{cases}
		1, & \text{if } (i-1)N+1 = j \\
		0, & \text{otherwise}.
	\end{cases}
\end{equation}
Algorithm~1 with $\kappa = 1.001$ and $(\sigma,\tau)$ given in the footnote for Theorem~\ref{def:Tprop}(b) is applied to the minimization problems, where the common initial estimate is set as $(x_0, v_0, w_0)=(0_{\calX}, 0_{\calZ}, 0_{\calZ})$ for all experiments. The operator $\Prox_{\gamma \|\cdot\|_1}$ for $\gamma \in \bbR_{++}$ in Algorithm~1 can be calculated by \eqref{eq:Proxl1}.

\begin{figure}[ht]
	\centering
	\subfloat[][]{\includegraphics[clip,width=0.45\hsize]{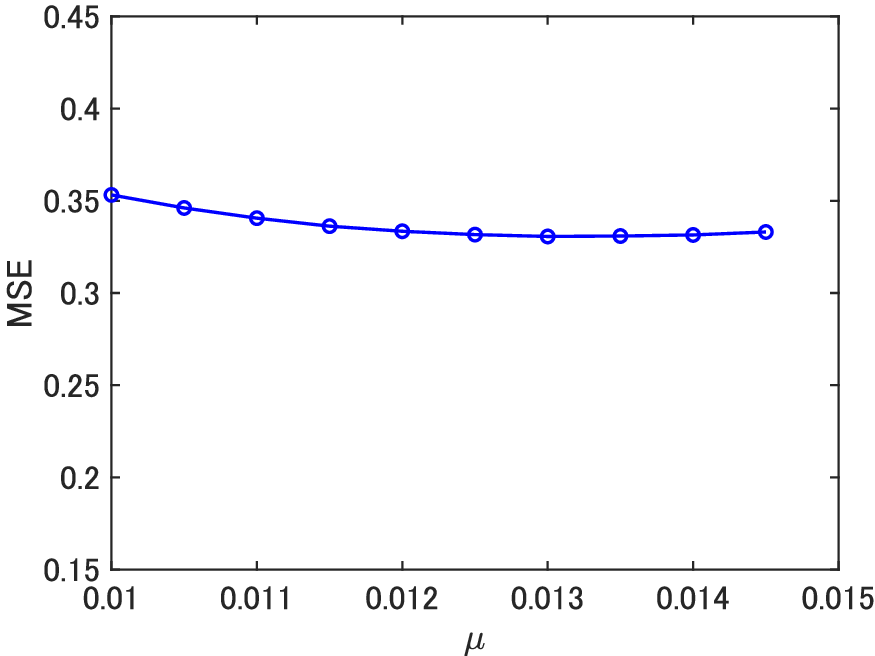}\label{subfig:exam2_tuneTV}} \quad
	\subfloat[][]{\includegraphics[clip,width=0.45\hsize]{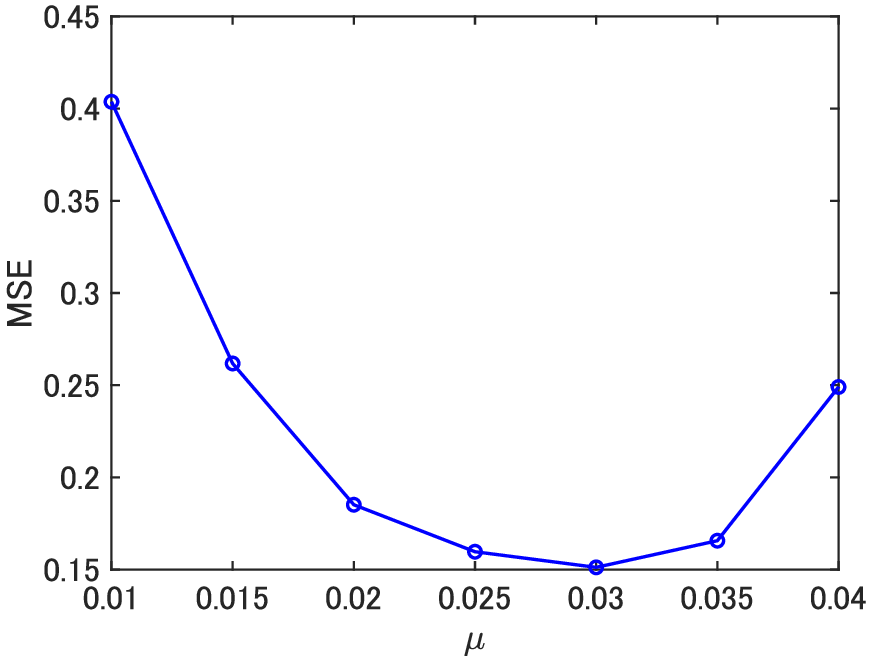}\label{subfig:exam2_tuneLiGME}}
	\caption{{MSE versus $\mu$ in Problem~1 at $k = 5,000$ iteration for (a) the anisotropic TV penalty $\| \cdot \|_1 \circ \bar{D}$ and (b) LiGME penalty $(\| \cdot \|_1)_{B_{\theta}} \circ \bar{D}$.}}
	\label{fig:exam2_tunes}
\end{figure}
\begin{figure}[ht]
	\centering
	\includegraphics[clip,width=0.60\hsize]{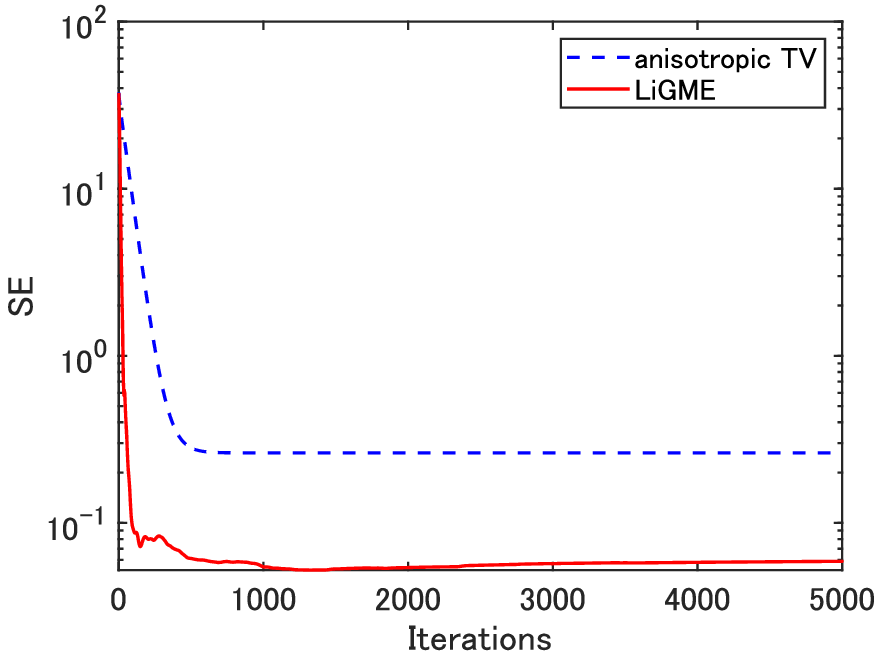}
	\caption{SE versus iterations for anisotropic TV (dotted blue) and LiGME (solid red).}
	\label{fig:exam2_trace}
\end{figure}
\begin{figure}[ht]
	\centering
	\subfloat[][]{\includegraphics[clip,width=0.45\hsize]{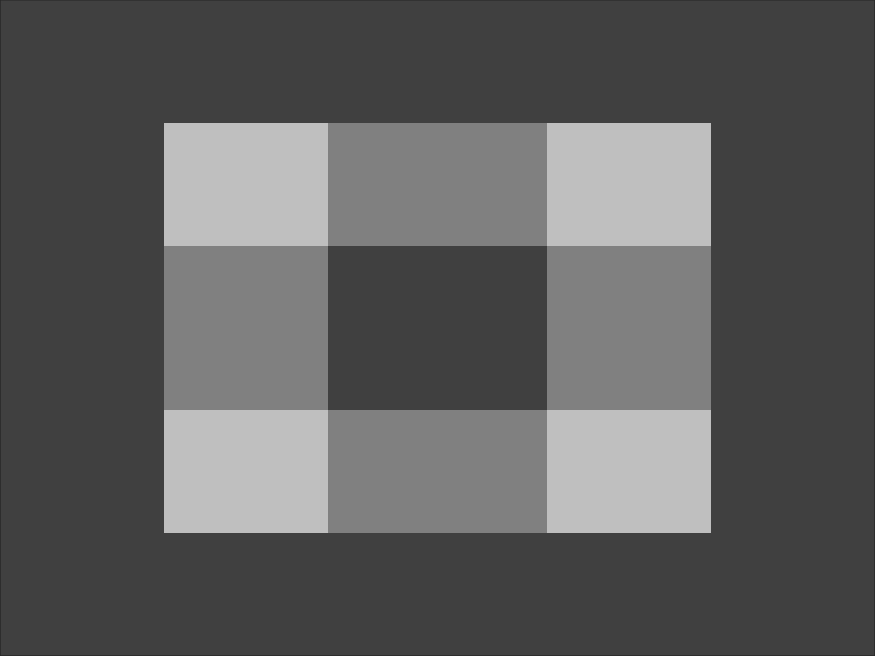}\label{subfig:exam2_x}} \quad
	\subfloat[][]{\includegraphics[clip,width=0.45\hsize]{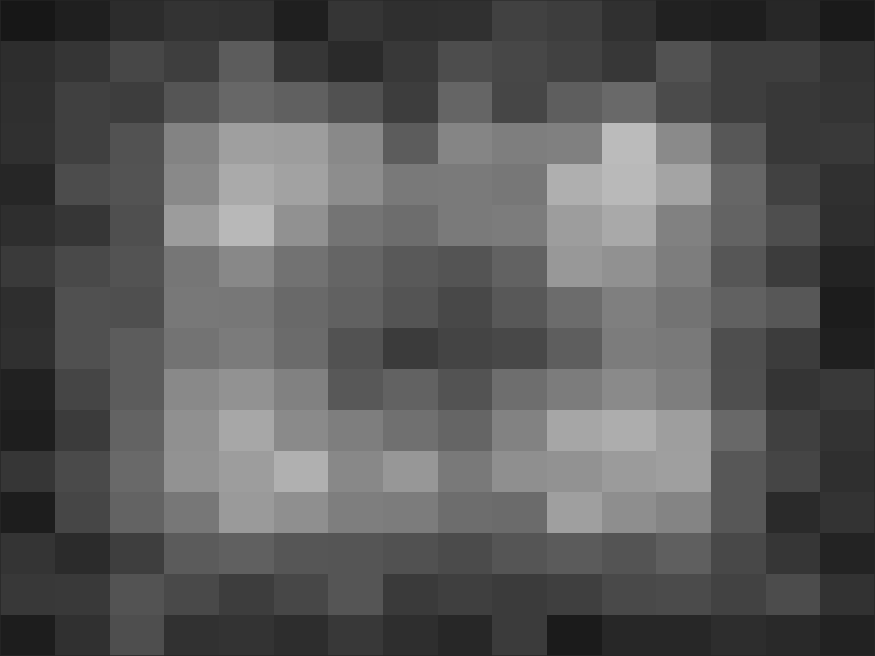}\label{subfig:exam2_y}} \quad
	\subfloat[][]{\includegraphics[clip,width=0.45\hsize]{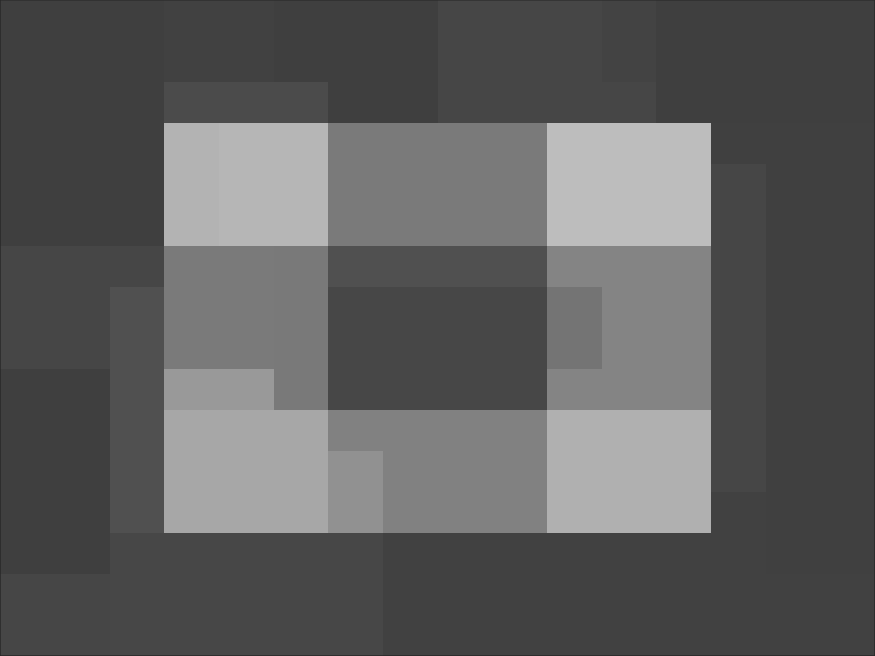}\label{subfig:exam2_estimateTV}} \quad
	\subfloat[][]{\includegraphics[clip,width=0.45\hsize]{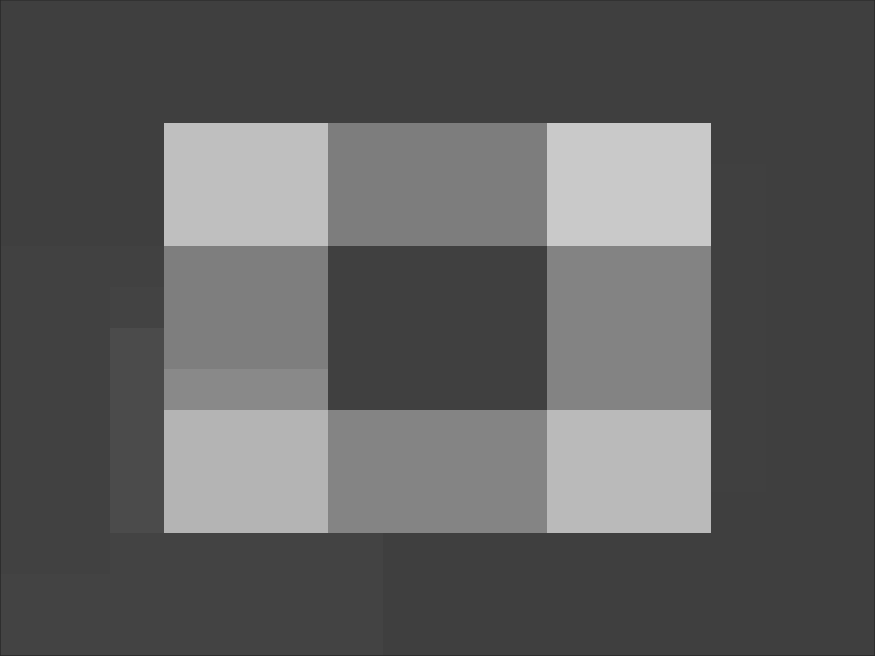}\label{subfig:exam2_estimateLiGME}} \quad
	\caption{{(a) original piecewise constant image whose pixels belong to $\{0.25, 0.50, 0.75\}$, (b) a noisy blurred image, (c) estimated image by using anisotropic TV penalty at $k = 5,000$ iteration, (d) estimated image by using LiGME penalty at $k = 5,000$ iteration. Each pixel is assigned a real value and displayed with under -0.2 in black and over 1.2 in white.}}
	\label{fig:exam2_estimates}
\end{figure}
Figure~\ref{fig:exam2_tunes} shows dependency of recovering performance on the parameter $\mu$ in Problem~\ref{prob:optim}.
The performance is measured by mean squared error (MSE) defined as the average of SE in \eqref{eq:SE} over $100$ independent realizations of the additive noise.
From Figure~\ref{fig:exam2_tunes}, we can see that (i) the best weights of the penalties are respectively $\mu_{\mathrm{TV}} := 0.013$ for $\| \cdot \|_1 \circ \bar{D}$ and $\mu_{\mathrm{LiGME}} := 0.03$ for $(\| \cdot \|_1)_{B_{\theta}} \circ \bar{D}$ and (ii) the estimation by LiGME penalty with $\mu_{\mathrm{LiGME}}$ outperforms the anisotropic TV penalty with $\mu_{\mathrm{TV}}$ in the context of MSE.

Figure~\ref{fig:exam2_trace} shows dependency of the SE on the number of iterations under weights $(\mu_{\mathrm{TV}}, \mu_{\mathrm{LiGME}})$.
The accuracy of the approximation by the LiGME penalty becomes higher than the anisotropic TV penalty from the beginning and SE for LiGME reaches 22.4\% of SE for anisotropic TV in the end.

Figure~\ref{fig:exam2_estimates} shows the original image, an observed image, and recovered images by the penalties at $5,000$ iteration.
The deblurring by LiGME $(\| \cdot \|_1)_{B_{\theta}} \circ \bar{D}$ restores much more successfully the sharp edges than the anisotropic TV.

\subsection{Matrix completion by promoting low-rankness}
\label{section:exam3}
We present a numerical experiment in a scenario of matrix completion by considering Problem~\ref{prob:optim} with $(\calX,\calY,\calZ) = (\bbR^{N^2},\bbR^{N^2},\bbR^{N^2})$, $N = 16$, $\Psi = \|{\rm vec}^{-1} (\cdot)\|_{\rm nuc}$ defined in Example~\ref{example1:LiGME}(c), and $\frakL = \rmId$.
In this experiment, the $(i,j)$-entry of $A$ is given by
\begin{equation}
	\label{eq:defmissing}
	A_{i,j} = \begin{cases}
		1, & \text{if } i = j \in \Omega, \\
		0, & \text{otherwise},
	\end{cases}
\end{equation}
where $\Omega \subset \{1,\dots,N^2\}$ satisfies $\#\Omega = N^2-M$ with $M = 64$, i.e., $25\%$ of entries are missing.
The matrix $A^{\sfT}A$ is singular because $\mathrm{rank}(A) = N^2 - M$.
The observation $\bmy \in \bbR^{N^2}$ (Figure~\ref{fig:exam3_estimates}(b)) is generated by $\bmy = A \bmx^{\star} + \varepsilon$, where $\bmx^{\star} \in \bbR^{N^2}$ is given by the vectorization of a low-rank matrix (Figure~\ref{fig:exam3_estimates}(a)) and $\varepsilon \in \bbR^{N^2}$ is additive white Gaussian noise. The signal-to-noise ratio (SNR) defined in \eqref{def:SNR} is 30dB.
%TODO: modify the || . ||_{nuc} to roman expression.
We compared minimizers of Problem~1, estimated by Algorithm 1, with two penalties: one is the nuclear norm, i.e., $(\|{\rm vec}^{-1} (\cdot)\|_{\rm nuc})_{B_{0}}=(\|{\rm vec}^{-1} (\cdot)\|_{\rm nuc})_{\rmO_{\calZ}} = \|{\rm vec}^{-1} (\cdot)\|_{\rm nuc}$, the other is a LiGME penalty $(\|{\rm vec}^{-1}(\cdot)\|_{\rm nuc})_{B_{\theta}}$ whose $B_{\theta} \in \bbR^{N^2 \times N^2}$ is obtained by Proposition~\ref{rem:selectionB} with $\theta = 0.99$ and $\tilde{\frakL} = \rmId$.
Algorithm~1 with $\kappa = 1.001$ and $(\sigma,\tau)$ given in the footnote for Theorem~\ref{def:Tprop}(b) is applied to the minimization problems, where the common initial estimate is set as $(x_0, v_0, w_0)=(0_{\calX}, 0_{\calZ}, 0_{\calZ})$ for all experiments. In Algorithm~1, the operator $\Prox_{\gamma \|{\rm vec}^{-1}(\cdot) \|_{\rm nuc}}$ for $\gamma \in \bbR_{++}$ can be calculated by
\begin{equation}
	\label{eq:Proxnuclear}
	\Prox_{\gamma \|{\rm vec}^{-1}(\cdot) \|_{\rm nuc}}(z) = \mathrm{vec} \left( U \mathrm{diag}(\Prox_{\gamma \| \cdot \|_1} ([\sigma_1, \dots, \sigma_N]^{\sfT})) V^{\sfT} \right),
\end{equation}
where $U \mathrm{diag}([\sigma_1, \dots, \sigma_N]) V^{\sfT} \ (\sigma_1 \geq \cdots \geq \sigma_N \geq 0)$ is a singular value decomposition of $\mathrm{vec}^{-1}(z) \in \bbR^{N \times N}$.

\begin{figure}[t]
	\centering
	\subfloat[][]{\includegraphics[clip,width=0.45\hsize]{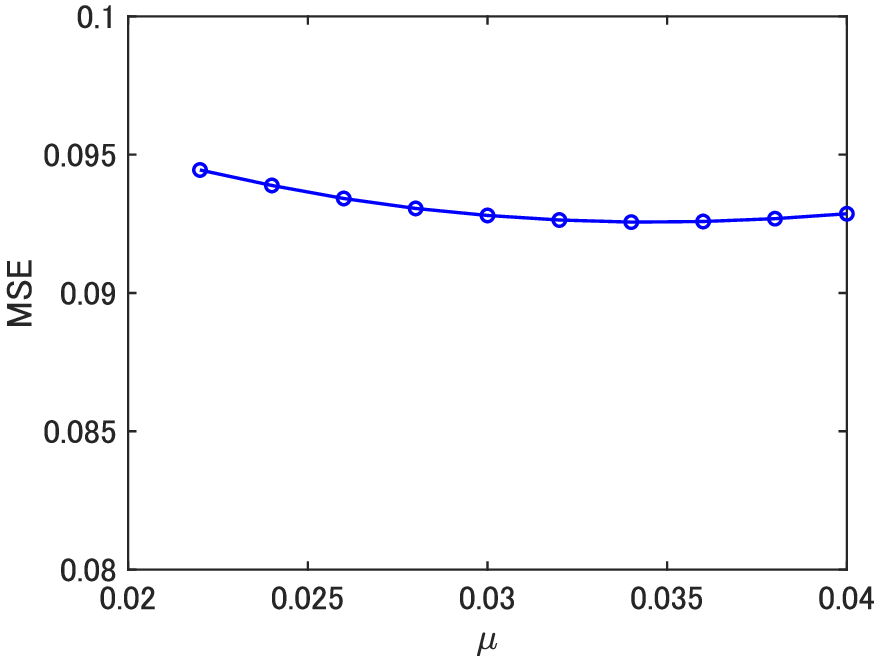}\label{subfig:exam3_tuneNC}} \quad
	\subfloat[][]{\includegraphics[clip,width=0.45\hsize]{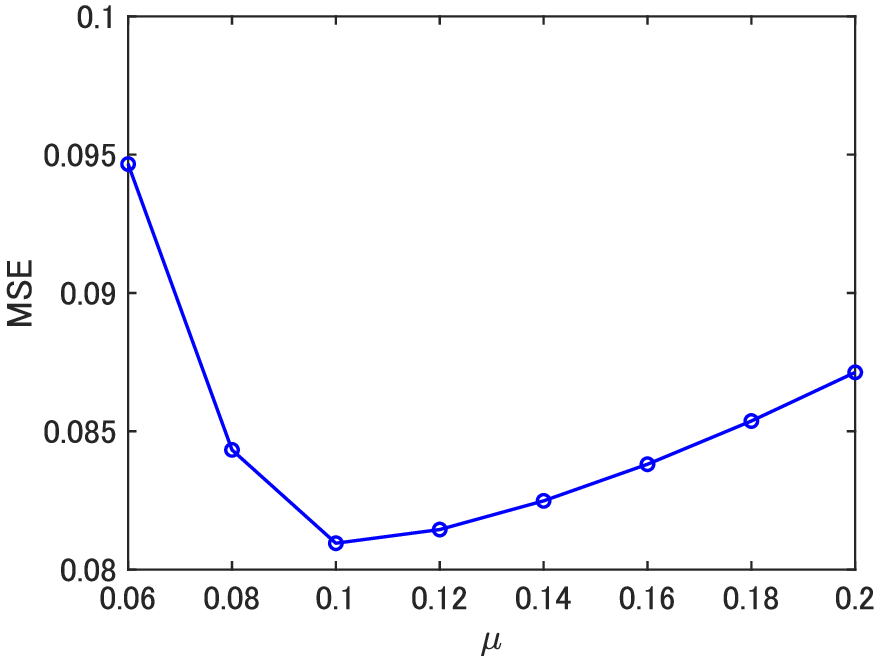}\label{subfig:exam3_tuneLiGME}} \quad
	\caption{{MSE versus $\mu$ in Problem~1 at $k = 500$ iteration for (a) the nuclear norm penalty $\| \cdot \|_{\rm nuc}$ and (b) LiGME penalty $(\| \cdot \|_{\rm nuc})_{B_{\theta}}$.}}
	\label{fig:exam3_tunes}
\end{figure}
\begin{figure}[t]
	\centering
	\includegraphics[clip,width=0.60\hsize]{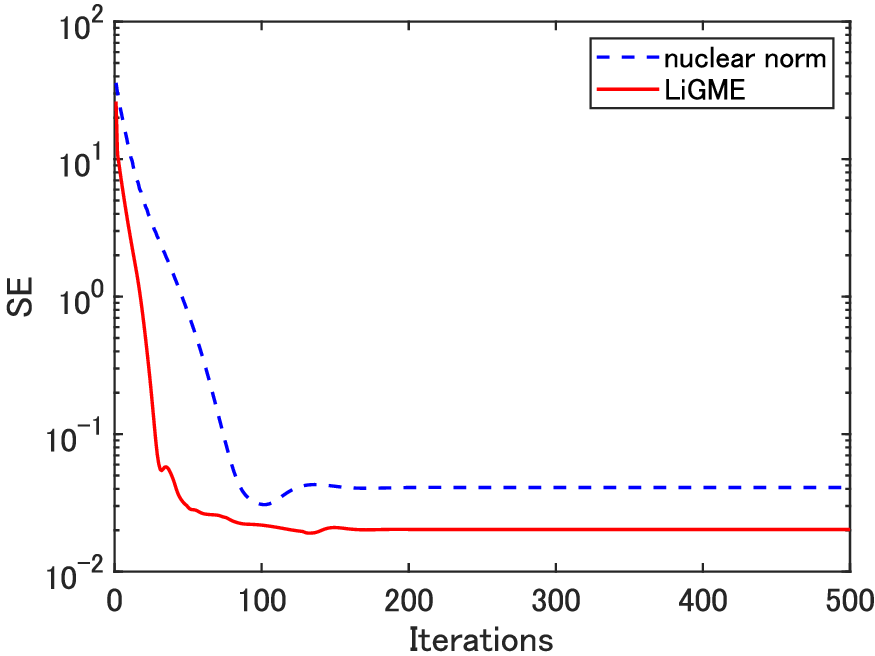}
	\caption{SE versus iterations for the nuclear norm (dotted blue) and LiGME (solid red).}
	\label{fig:exam3_trace}
\end{figure}
\begin{figure}[ht]
	\centering
	\subfloat[][]{\includegraphics[clip,width=0.45\hsize]{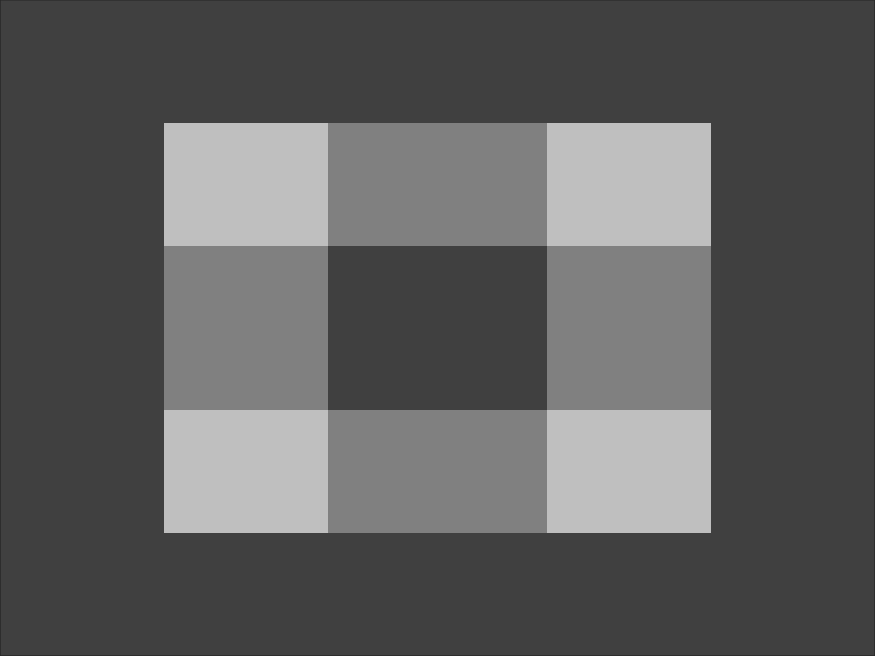}\label{subfig:exam3_x}} \quad
	\subfloat[][]{\includegraphics[clip,width=0.45\hsize]{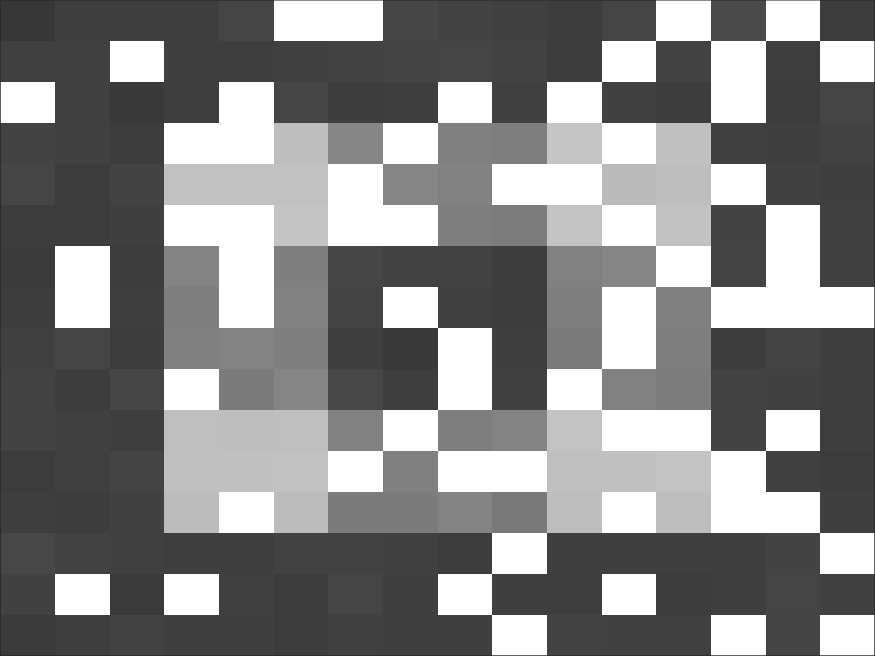}\label{subfig:exam3_y}} \quad
	\subfloat[][]{\includegraphics[clip,width=0.45\hsize]{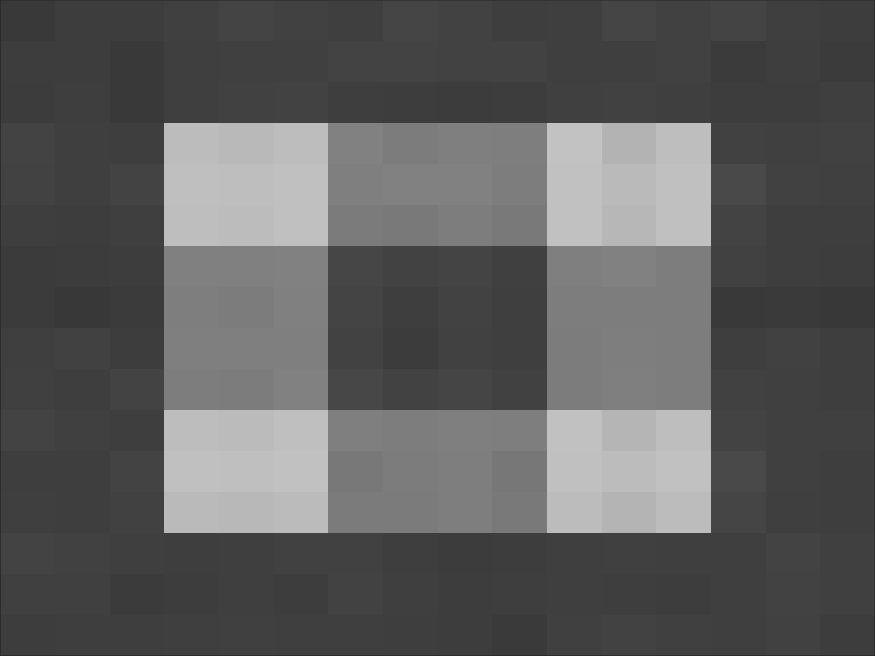}\label{subfig:exam3_estimateONE}} \quad
	\subfloat[][]{\includegraphics[clip,width=0.45\hsize]{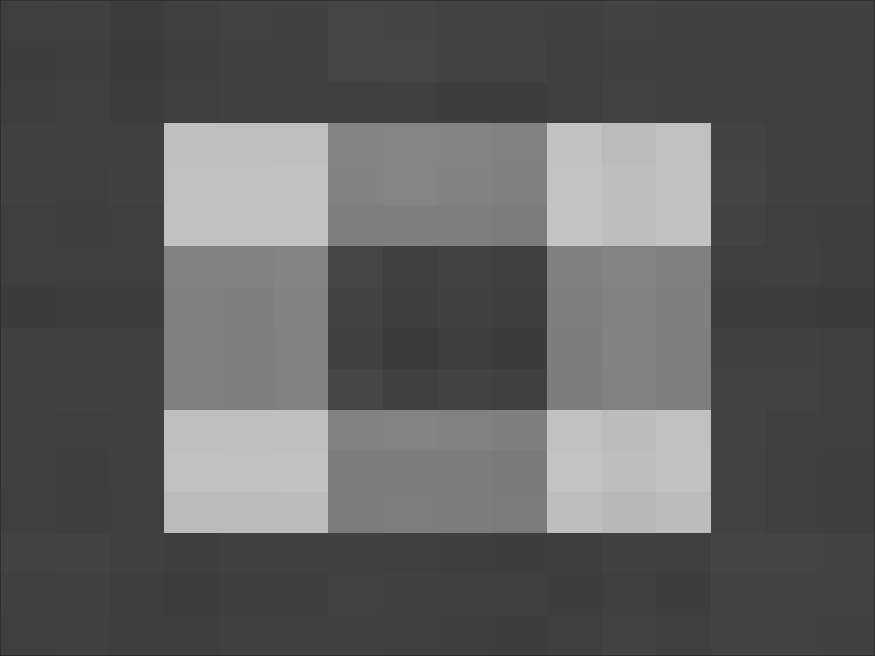}\label{subfig:exam3_estimateGMC}} \quad
	\caption{{(a) Original low-rank matrix whose rank is 3, (b) observed matrix whose missing entries are displayed in white, (c) estimated matrix by using the nuclear norm penalty at $k = 500$ iteration, (d) estimated matrix by using LiGME penalty at $k = 500$ iteration. Each entry is displayed with under -0.2 in black and over 1.2 in white.}}
	\label{fig:exam3_estimates}
\end{figure}

\begin{table}[ht]
	\centering
	\caption{Singular values $\sigma_1 \geq \cdots \geq \sigma_{16} \geq 0$ of the original and estimated matrices in Figure~\ref{fig:exam3_estimates} and the numerical rank (num-rank) which is defined as the number of singular values greater than $10^{-8}$.}
	\label{table:exam3_singularvalues}
        {\small
	\begin{tabular}{lccccccc} \hline \hline
		singular values & $\sigma_1$ & $\sigma_2$ & $\sigma_3$ & $\sigma_4$ & $\cdots$ & $\sigma_{16}$ & num-rank \\ \hline
		original & $6.48 \times 10^{0}$ & $9.01 \times 10^{-1}$ & $3.85 \times 10^{-1}$ & 0 & $\cdots$ & 0 & 3 \\
		nuclear norm & $6.42 \times 10^{0}$ & $8.55 \times 10^{-1}$ & $3.38 \times 10^{-1}$ & $6.66 \times 10^{-2}$ & $\cdots$ & $6.52 \times 10^{-11}$ & 8 \\
		LiGME & $6.48 \times 10^{0}$ & $9.11 \times 10^{-1}$ & $3.89 \times 10^{-1}$ & $1.10 \times 10^{-14}$ & $\cdots$ & $8.26 \times 10^{-17}$ & 3 \\ \hline 
	\end{tabular}
        }
\end{table}
Figure~\ref{fig:exam3_tunes} shows dependency of recovering performance on the parameter $\mu$ in Problem~\ref{prob:optim}.
The performance is measured by mean squared error (MSE) defined as the average of SE in \eqref{eq:SE} over $100$ independent realizations of the additive noise.
From Figure~\ref{fig:exam3_tunes}, we can see that (i) the best weights of the penalties are respectively $\mu_{\mathrm{nuc}} := 0.034$ for $\| \cdot \|_{\rm nuc}$ and $\mu_{\mathrm{LiGME}} := 0.1$ for $(\| \cdot \|_{\rm nuc})_{B_{\theta}}$ and (ii) the estimation by LiGME penalty with $\mu_{\mathrm{LiGME}}$ outperforms the nuclear norm penalty with $\mu_{\mathrm{nuc}}$ in the context of MSE.

Figure~\ref{fig:exam3_trace} shows dependency of the SE on the number of iterations under weights $(\mu_{\mathrm{nuc}}, \mu_{\mathrm{LiGME}})$.
The accuracy of the approximation by the LiGME penalty becomes higher than the nuclear norm penalty from the  beginning and SE for LiGME reaches 49.5\% of SE for nuclear norm in the end.

Figure~\ref{fig:exam3_estimates} shows the original matrix, an observed matrix, and recovered matrices by the penalties at $500$ iteration and Table~\ref{table:exam3_singularvalues} shows the singular values of the original matrix and the recovered matrices in Figure~\ref{fig:exam3_estimates}.
In the context of singular values in Table~\ref{table:exam3_singularvalues}, the recovered matrix by the LiGME penalty more accurately approximates the original than by the nuclear norm penalty.
Especially, the number of singular values greater than $10^{-8}$ (num-rank) of the recovered matrix by the LiGME is equal to of the original.

\subsection{Matrix completion by promoting low-rankness and smoothness}
\label{section:exam4}
We present a numerical experiment in a scenario of matrix completion by considering Problem~\ref{prob:optim} and Example~3 with $(\mathcal{M}, \calX, \calY, \calZ_1, \calZ_2, \calZ_3) = (3, \bbR^{N^2}, \bbR^{N^2}, \bbR^{N(N-1)},$ $\bbR^{N(N-1)}, \bbR^{N^2})$, $N=16$, $(\Psi^{\langle1\rangle}, \Psi^{\langle2\rangle}, \Psi^{\langle3\rangle}) = ({\| \cdot \|_1}, {\| \cdot \|_1}, {\| {\rm vec}^{-1} (\cdot) \|_{\rm nuc}})$, $\mu = 1$, $(\frakL_1, \frakL_2, \frakL_3) = (D_{\mathrm{V}}, D_{\mathrm{H}}, \rmId)$ defined in \eqref{eq:defD2d}.
In this experiment, for $\Omega \subset \{1,\dots,N^2\}$ with $\#\Omega = N^2-M$ and $M = 64$, the $(i,j)$-entry of $A$ is given by \eqref{eq:defmissing}, which satisfies $\mathrm{rank}(A) = N^2-M$.
The observation $\bmy \in \bbR^{N^2}$ (Figure~\ref{fig:exam4_estimates}(b)) is generated by $\bmy = A \bmx^{\star} + \varepsilon$, where $\bmx^{\star} \in \bbR^{N^2}$ is given by the vectorization of a piecewise constant image (Figure~\ref{fig:exam4_estimates}(a)) and $\varepsilon \in \bbR^{N^2}$ is additive white Gaussian noise. The signal-to-noise ratio (SNR) defined in \eqref{def:SNR} is 20dB, which is lower than the SNR set in Section~4.3.

We compared minimizers of Problem~1, estimated by Algorithm~1, with four penalties:
\begin{align*}
  \text{(i) } \Psi_{\rm \numI} \circ \frakL & := \mu_{a} \left[ (\| \cdot \|_1)_{B^{\langle 1\rangle}_{0}} \circ D_{\mathrm{V}} + (\| \cdot \|_1)_{B^{\langle 2\rangle}_{0}} \circ D_{\mathrm{H}}\right]  + \mu_{b} (\| {\rm vec}^{-1} (\cdot) \|_{\rm nuc})_{B^{\langle 3\rangle}_{0}} \\
  \phantom{\text{(i) } \Psi_{\rm \numI} \circ \frakL} & =
                                                        \mu_a \left[ (\| \cdot \|_1)_{\rmO_{\calZ_1}} \circ D_{\mathrm{V}} + (\| \cdot \|_1)_{\rmO_{\calZ_2}} \circ D_{\mathrm{H}} \right] + \mu_b (\| {\rm vec}^{-1} (\cdot) \|_{\rm nuc})_{\rmO_{\calZ_3}}, \\
 \text{(ii) } \Psi_{\rm \numII} \circ \frakL & :=
                                               \mu_a \left[ (\| \cdot \|_1)_{B^{\langle 1\rangle}_{\theta_1}} \circ D_{\mathrm{V}} + (\| \cdot \|_1)_{B^{\langle 2\rangle}_{\theta_2}} \circ D_{\mathrm{H}} \right] + \mu_b (\| {\rm vec}^{-1} (\cdot) \|_{\rm nuc})_{B^{\langle 3\rangle}_{0}} \\
  \phantom{ \text{(ii) } \Psi_{\rm \numII} \circ \frakL } & =
                                                            \mu_a \left[(\| \cdot \|_1)_{B^{\langle 1\rangle}_{\theta_1}} \circ D_{\mathrm{V}}  + (\| \cdot \|_1)_{B^{\langle 2\rangle}_{\theta_2}} \circ D_{\mathrm{H}} \right] + \mu_b (\| {\rm vec}^{-1} (\cdot) \|_{\rm nuc})_{\rmO_{\calZ_3}}, \\
  \text{(iii) } \Psi_{\rm \numIII} \circ \frakL & := \mu_a \left[ (\| \cdot \|_1)_{B^{\langle 1\rangle}_{0}} \circ D_{\mathrm{V}} + (\| \cdot \|_1)_{B^{\langle 2\rangle}_{0}} \circ D_{\mathrm{H}} \right] + \mu_b (\| {\rm vec}^{-1} (\cdot) \|_{\rm nuc})_{B^{\langle 3\rangle}_{\theta_3}} \\
  \phantom{  \text{(iii) } \Psi_{\rm \numIII} \circ \frakL } & =
                                                               \mu_a \left[ (\| \cdot \|_1)_{\rmO_{\calZ_1}} \circ D_{\mathrm{V}} + (\| \cdot \|_1)_{\rmO_{\calZ_2}} \circ D_{\mathrm{H}} \right] + \mu_b (\| {\rm vec}^{-1} (\cdot) \|_{\rm nuc})_{B^{\langle 3\rangle}_{\theta_3}}, \\
  \text{(iv) } \Psi_{\rm \numIV} \circ \frakL & := \mu_a \left[ (\| \cdot \|_1)_{B^{\langle 1\rangle}_{\theta_1}} \circ D_{\mathrm{V}} + (\| \cdot \|_1)_{B^{\langle 2\rangle}_{\theta_2}} \circ D_{\mathrm{H}} \right] + \mu_b (\| {\rm vec}^{-1} (\cdot) \|_{\rm nuc})_{B^{\langle 3\rangle}_{\theta_3}}.
\end{align*}
% (ii) $\Psi_{\rm \numII} \circ \frakL :=
% \mu_1 (\| \cdot \|_1)_{B^{\langle 1\rangle}_{\theta}} \circ \bar{D} + \mu_2 (\| {\rm vec}^{-1} (\cdot) \|_{\rm nuc})_{B^{\langle 2\rangle}_{0}}=
% \mu_1 (\| \cdot \|_1)_{B^{\langle 1\rangle}_{\theta}} \circ \bar{D} + \mu_2 (\| {\rm vec}^{-1} (\cdot) \|_{\rm nuc})_{\rmO_{\calZ_2}}$, \\
% (iii) $\Psi_{\rm \numIII} \circ \frakL := \mu_1 (\| \cdot \|_1)_{B^{\langle 1\rangle}_{0}} \circ \bar{D} + \mu_2 (\| {\rm vec}^{-1} (\cdot) \|_{\rm nuc})_{B^{\langle 2\rangle}_{\theta}}=
% \mu_1 (\| \cdot \|_1)_{\rmO_{\calZ_1}} \circ \bar{D} + \mu_2 (\| {\rm vec}^{-1} (\cdot) \|_{\rm nuc})_{B^{\langle 2\rangle}_{\theta}}$, \\
%(iv) $\Psi_{\rm \numIV} \circ \frakL := \mu_1 (\| \cdot \|_1)_{B^{\langle 1\rangle}_{\theta}} \circ \bar{D} + \mu_2 (\| {\rm vec}^{-1} (\cdot) \|_{\rm nuc})_{B^{\langle 2\rangle}_{\theta}}$.
In each penalty, $B^{\langle i \rangle}_{\theta_i}$ $(i=1,2,3)$ are obtained by Corollary~\ref{col:selectionBmultiple} with $\mu_1=\mu_2=\mu_a$, $\mu_3=\mu_b$, $\omega_1=\omega_2=\omega_3=1/3$, $\theta_1=\theta_2=\theta_3=0.99$, $(\tilde{\frakL}_1, \tilde{\frakL}_2, \tilde{\frakL}_3) = (\tilde{D}_{\mathrm{V}}, \tilde{D}_{\mathrm{H}}, \rmId)$ and $(\tilde{D}_{\mathrm{V}}, \tilde{D}_{\mathrm{H}})$ defined in \eqref{eq:choices_tildeD2}.
Algorithm~1 with $\kappa = 1.001$ and $(\sigma,\tau)$ given in the footnote for Theorem~\ref{def:Tprop}(b) is applied to the minimization problems, where the common initial estimate is set as $(x_0, v_0, w_0)=(0_{\calX}, 0_{\calZ}, 0_{\calZ})$ for all experiments. The operator $\Prox_{\gamma \Psi}$ for $\gamma \in \bbR_{++}$ can be calculated by
\begin{align}
	\hspace{-15mm}\Prox_{\gamma \Psi} &\colon \calZ_1 \times \calZ_2 \times \calZ_3 \to \calZ_1 \times \calZ_2 \times \calZ_3 \\ \nonumber
	&\colon (z_1, z_2, z_3) \mapsto (\Prox_{\gamma \| \cdot \|_1}(z_1), \Prox_{\gamma \| \cdot \|_1}(z_2), \Prox_{\gamma \| {\rm vec}^{-1} (\cdot) \|_{\rm nuc}}(z_3)),  \label{eq:Proxprod}
\end{align}
where $\Prox_{\gamma \| \cdot \|_1}$ and $\Prox_{\gamma \| {\rm vec}^{-1} (\cdot) \|_{\rm nuc}}$ are given by \eqref{eq:Proxl1} and \eqref{eq:Proxnuclear} respectively.

\begin{figure}[t]
	\centering
	\subfloat[][]{\includegraphics[clip,width=0.45\hsize]{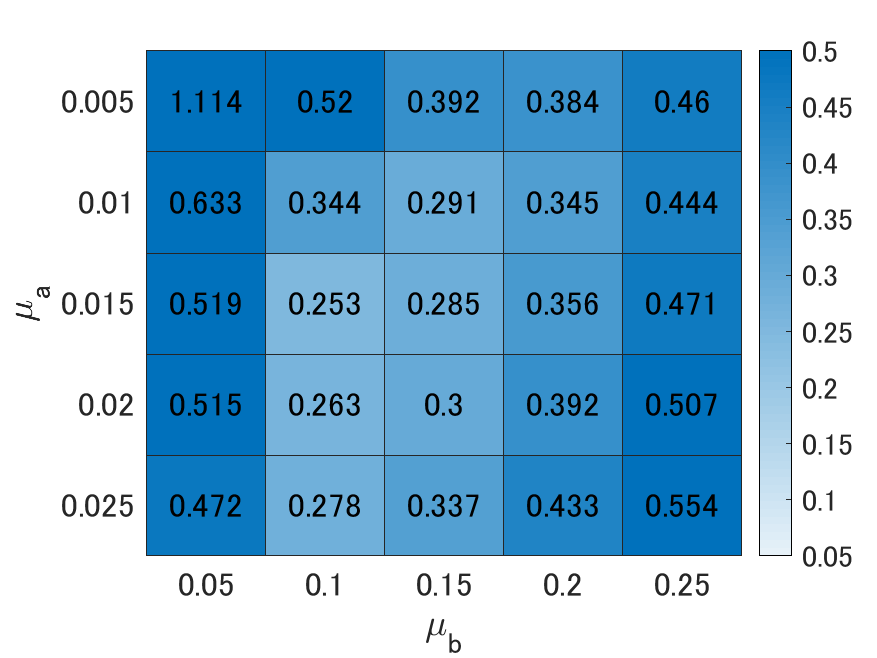}\label{subfig:exam4_tuneONE}} \quad
	\subfloat[][]{\includegraphics[clip,width=0.45\hsize]{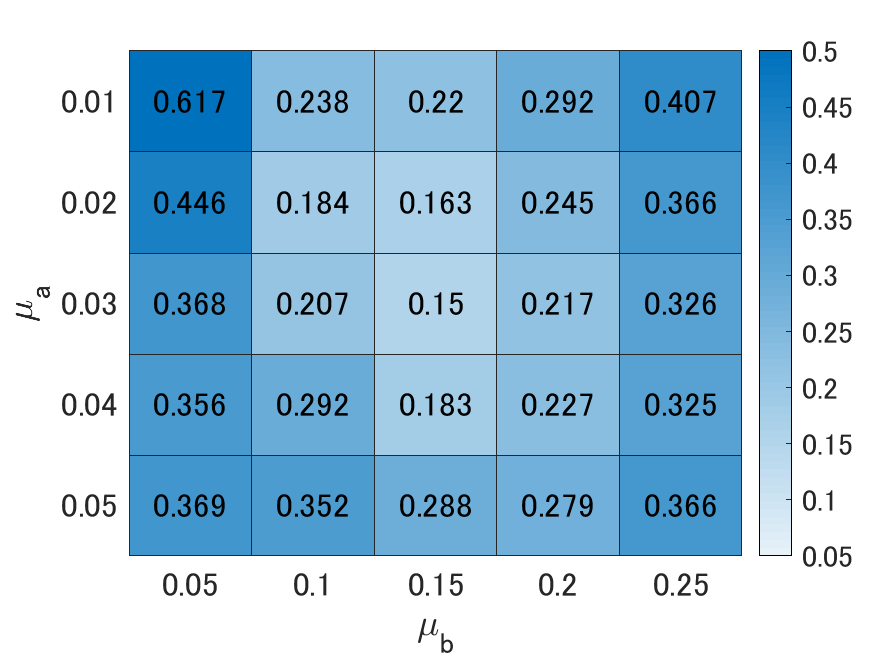}\label{subfig:exam4_tuneDO}} \quad
	\subfloat[][]{\includegraphics[clip,width=0.45\hsize]{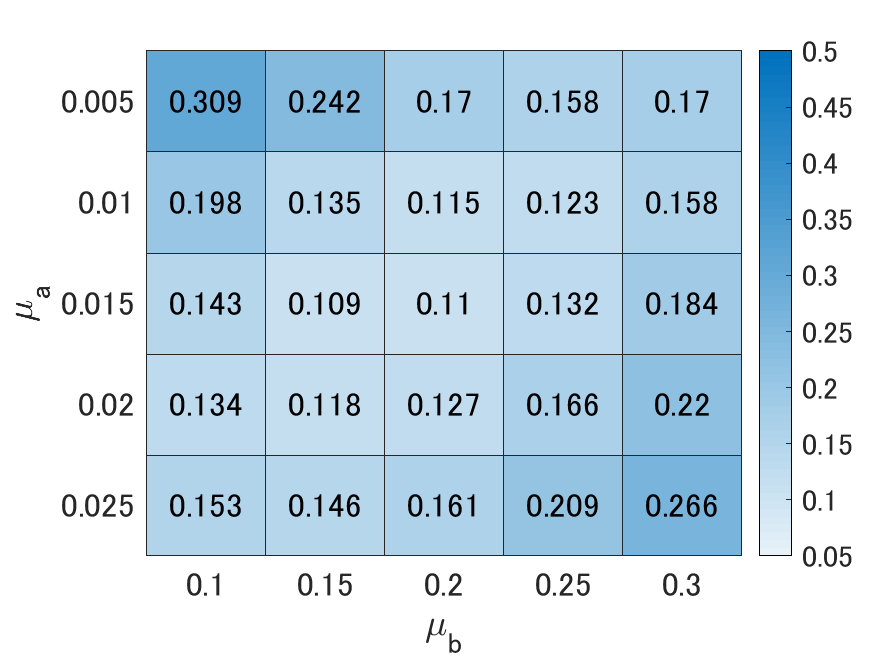}\label{subfig:exam4_tuneOL}} \quad
	\subfloat[][]{\includegraphics[clip,width=0.45\hsize]{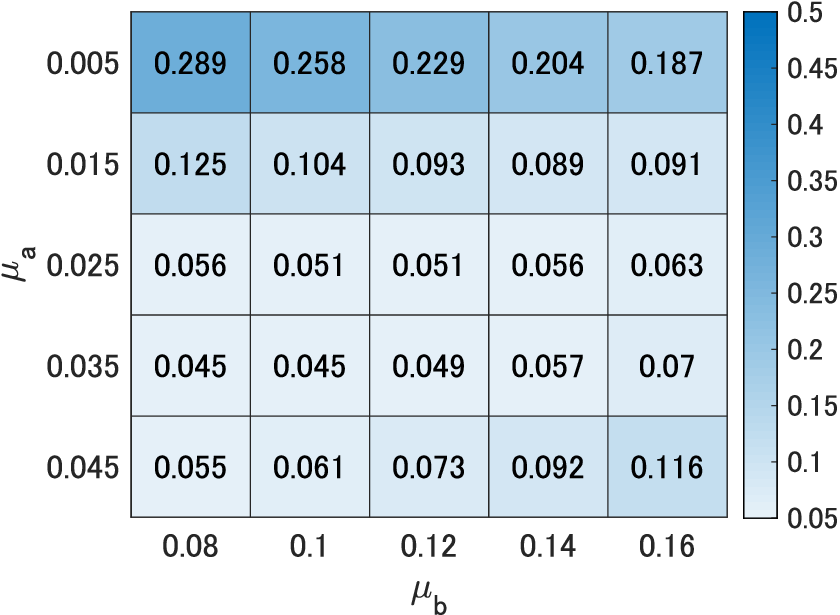}\label{subfig:exam4_tuneLiGME}} \quad
	\caption{{MSE versus $(\mu_a, \mu_b)$ in Problem~1 at $k = 1,000$ iteration for penalties (a) $\Psi_{\rm \numI} \circ \frakL := \mu_a [ \| \cdot \|_1 \circ D_{\mathrm{V}} + \| \cdot \|_1 \circ D_{\mathrm{V}} ] + \mu_b \| {\rm vec}^{-1} (\cdot) \|_{\rm nuc}$, (b) $\Psi_{\rm \numII} \circ \frakL := \mu_a [ (\| \cdot \|_1)_{B^{\langle 1\rangle}_{\theta}} \circ D_{\mathrm{V}} + (\| \cdot \|_1)_{B^{\langle 2\rangle}_{\theta}} \circ D_{\mathrm{H}} ] + \mu_b \| {\rm vec}^{-1} (\cdot) \|_{\rm nuc}$, (c) $\Psi_{\rm \numIII} \circ \frakL := \mu_a [ \| \cdot \|_1 \circ D_{\mathrm{V}} + \| \cdot \|_1 \circ D_{\mathrm{H}}] + \mu_b (\| {\rm vec}^{-1} (\cdot) \|_{\rm nuc})_{B^{\langle 3\rangle}_{\theta}}$, (d) $\Psi_{\rm \numIV} \circ \frakL := \mu_a [ (\| \cdot \|_1)_{B^{\langle 1\rangle}_{\theta}} \circ D_{\mathrm{V}} + (\| \cdot \|_1)_{B^{\langle 2\rangle}_{\theta}} \circ D_{\mathrm{H}} ] + \mu_b (\| {\rm vec}^{-1} (\cdot) \|_{\rm nuc})_{B^{\langle 3\rangle}_{\theta}}$.}}
	\label{fig:exam4_tunes}
\end{figure}
\begin{figure}[ht]
	\centering
	\includegraphics[clip,width=0.60\hsize]{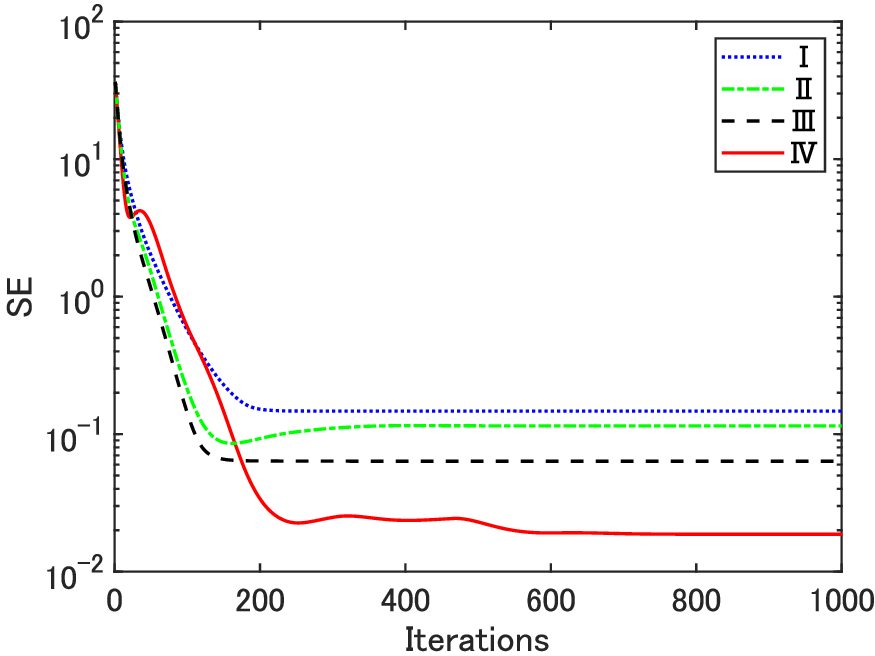}
	\caption{SE versus iterations for penalties $\Psi_{\rm \numI} \circ \frakL$(dotted blue), $\Psi_{\rm \numII} \circ \frakL$(dash-dotted green), $\Psi_{\rm \numIII} \circ \frakL$(dashed black), and $\Psi_{\rm \numIV} \circ \frakL$(solid red).}
	\label{fig:exam4_trace}
\end{figure}
\begin{figure}[ht]
	\centering
	\subfloat[][]{\includegraphics[clip,width=0.45\hsize]{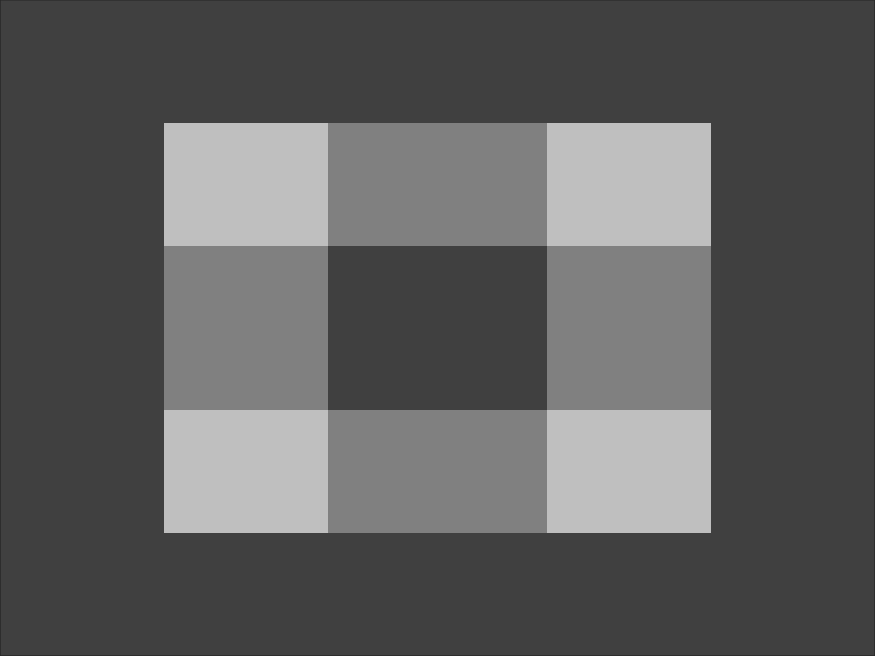}\label{subfig:exam4_x}} \quad
	\subfloat[][]{\includegraphics[clip,width=0.45\hsize]{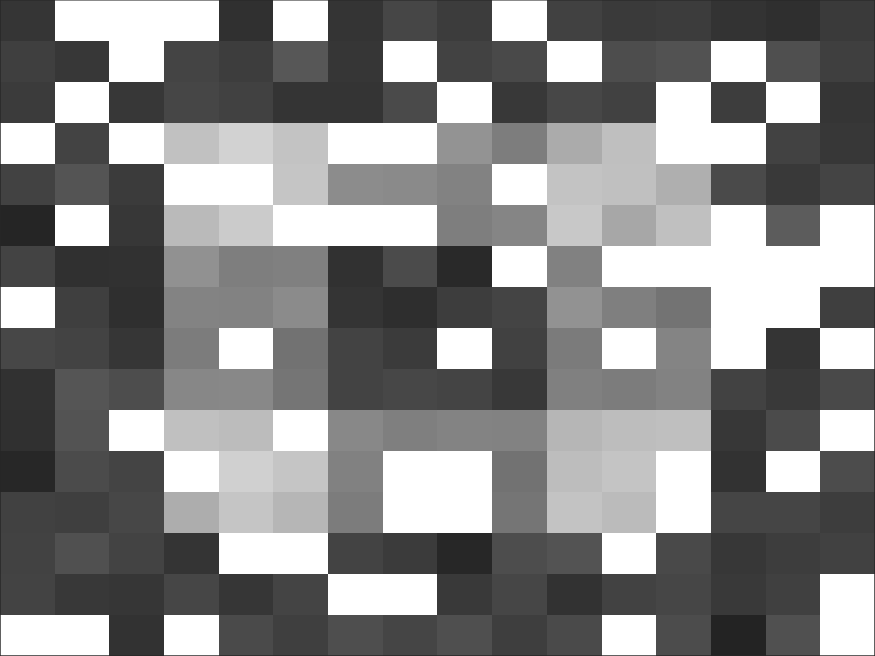}\label{subfig:exam4_y}} \quad
	\subfloat[][]{\includegraphics[clip,width=0.45\hsize]{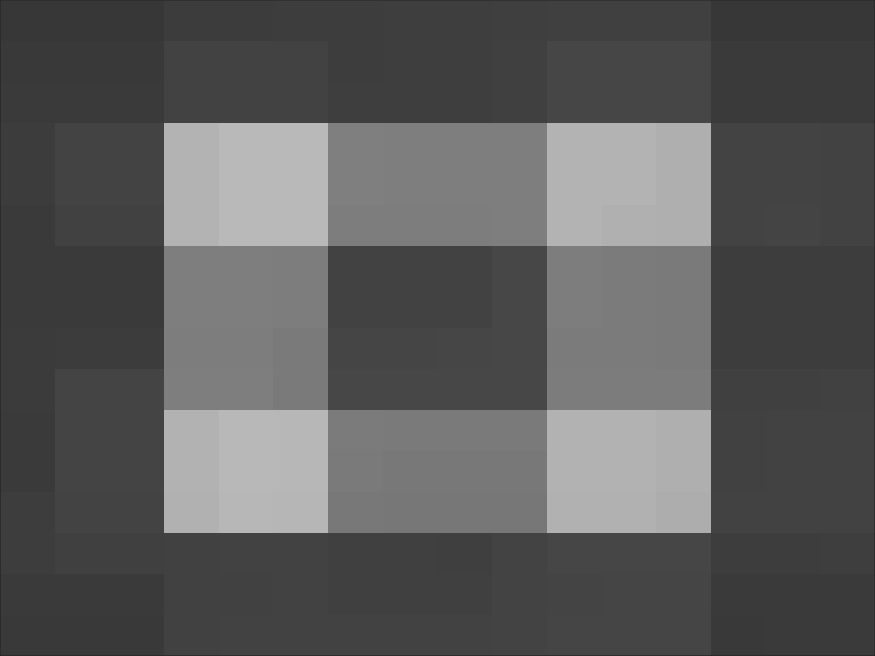}\label{subfig:exam4_estimateONE}} \quad
	\subfloat[][]{\includegraphics[clip,width=0.45\hsize]{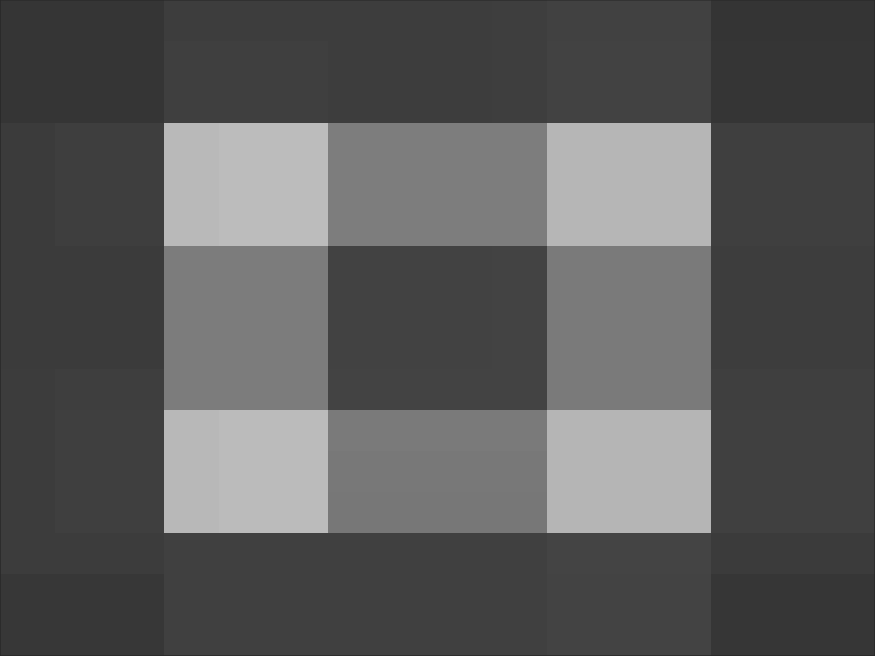}\label{subfig:exam4_estimateDO}} \quad
	\subfloat[][]{\includegraphics[clip,width=0.45\hsize]{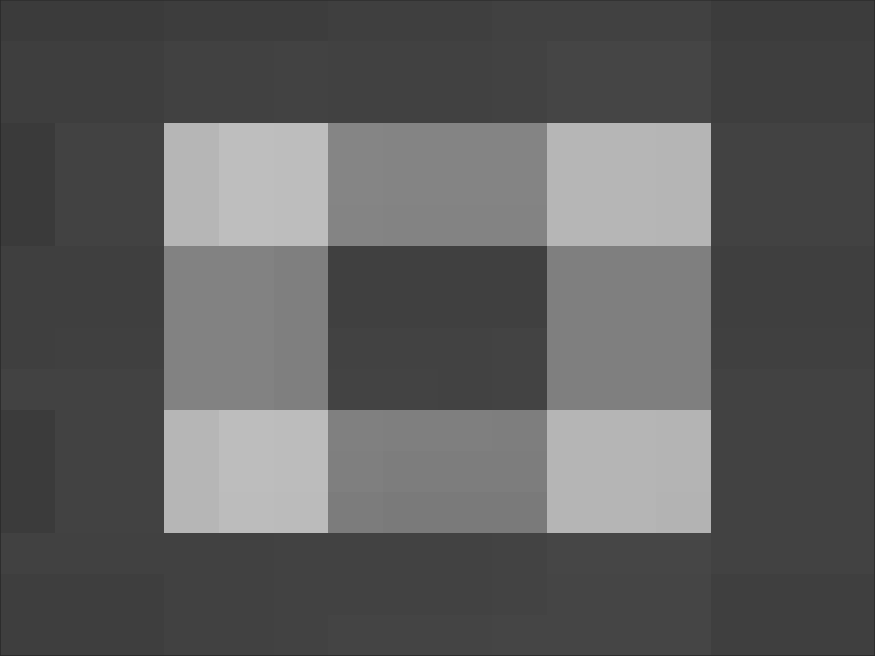}\label{subfig:exam4_estimateOL}} \quad
	\subfloat[][]{\includegraphics[clip,width=0.45\hsize]{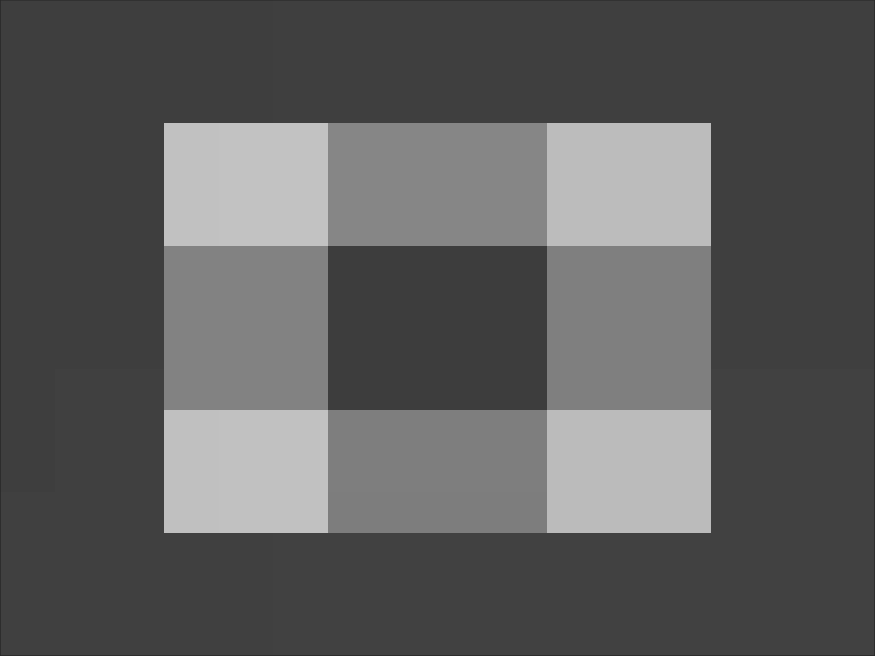}\label{subfig:exam4_estimateLiGME}}
	\caption{{(a) original low-rank and piecewise-constant matrix which is the same as Figure~\ref{fig:exam3_estimates}(a), (b) observed matrix whose missing entries are displayed in white, (c) estimated matrix by using $\Psi_{\rm \numI} \circ \frakL$ at $k = 1,000$ iteration, (d) estimated matrix by using $\Psi_{\rm \numII} \circ \frakL$ at $k = 1,000$ iteration, (e) estimated matrix by using $\Psi_{\rm \numIII} \circ \frakL$ at $k = 1,000$ iteration, (f) estimated matrix by using $\Psi_{\rm \numIV} \circ \frakL$ at $k = 1,000$ iteration. Each entry is displayed with under -0.2 in black and over 1.2 in white.}}
	\label{fig:exam4_estimates}
\end{figure}
\begin{figure}[ht]
	\centering
	\subfloat[][]{\includegraphics[clip,width=0.45\hsize]{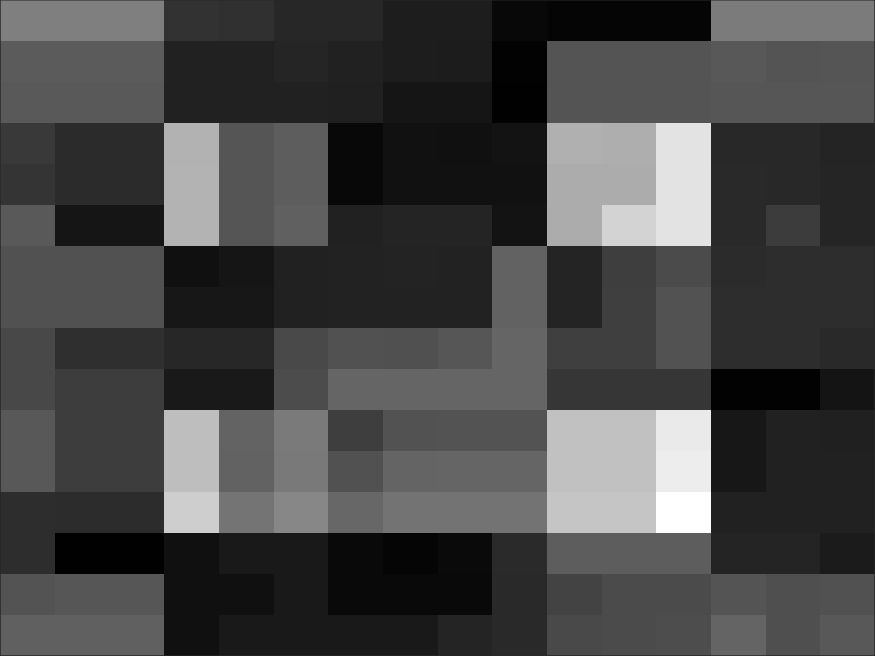}\label{subfig:exam4_diffONE}} \quad
	\subfloat[][]{\includegraphics[clip,width=0.45\hsize]{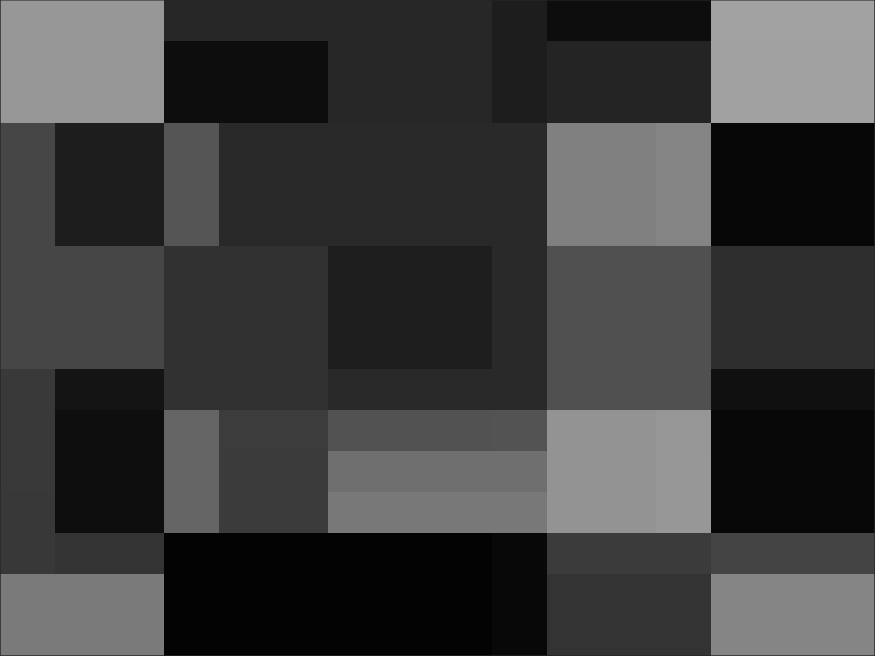}\label{subfig:exam4_diffDO}} \quad
	\subfloat[][]{\includegraphics[clip,width=0.45\hsize]{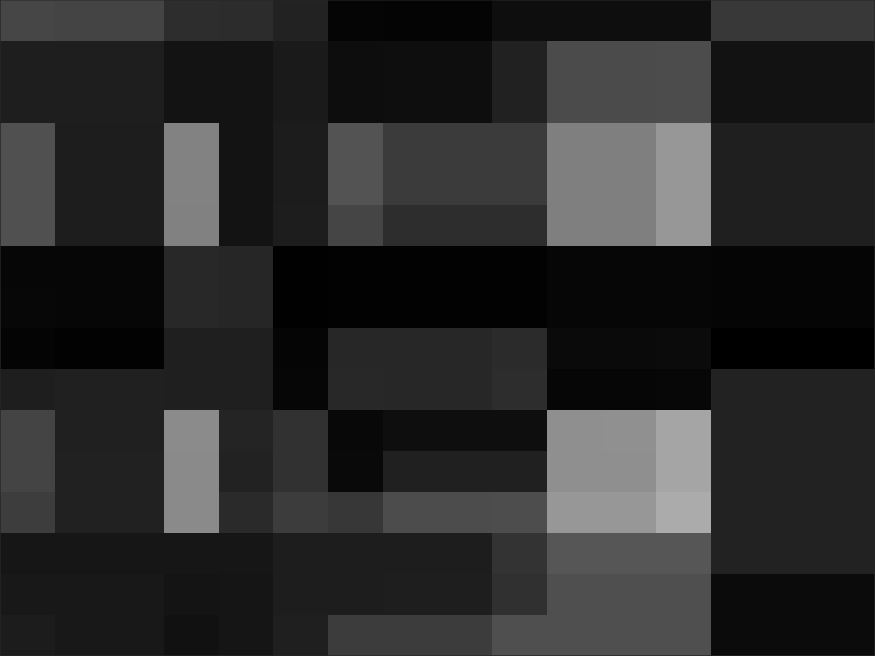}\label{subfig:exam4_diffOL}} \quad
	\subfloat[][]{\includegraphics[clip,width=0.45\hsize]{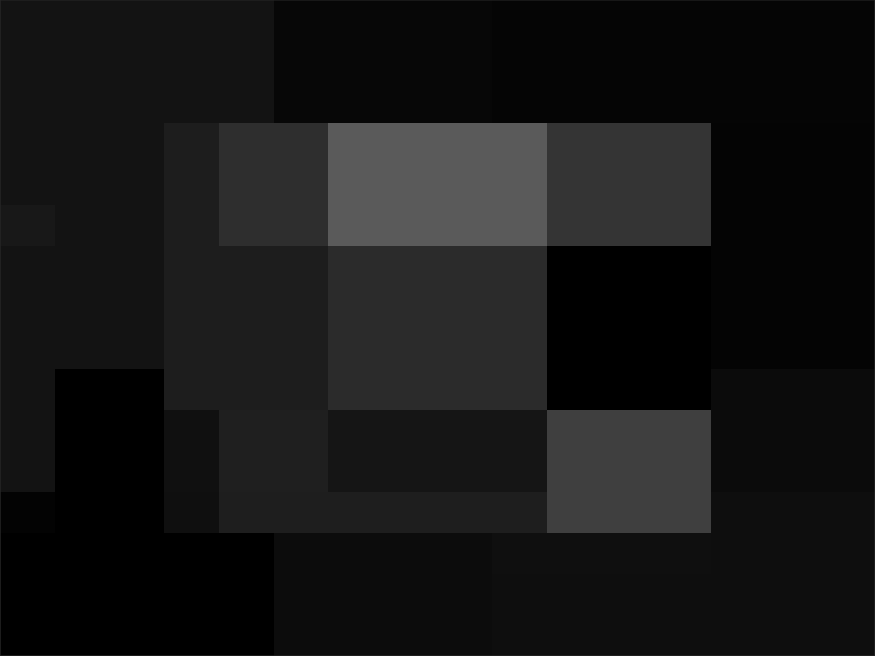}\label{subfig:exam4_diffGMC}}
	\caption[]{{The differences between original matrix and (a) the matrix in Figure~\ref{fig:exam4_estimates}(c), (b) the matrix in Figure~\ref{fig:exam4_estimates}(d), (c) the matrix in Figure~\ref{fig:exam4_estimates}(e), and (d) the matrix in Figure~\ref{fig:exam4_estimates}(f). Each entry has absolute value of the difference and is displayed with 0 in black and 0.07 in white.}}
	\label{fig:exam4_diffs}
\end{figure}

%\begin{table}[ht]
%	\centering
%	\caption{Singular values $\sigma_1 \geq \cdots \geq \sigma_{16} \geq 0$ of the estimated images in Figure~...}
%	\begin{tabular}{lccccccr} \hline \hline
%		penalty & $\sigma_1$ & $\sigma_2$ & $\sigma_3$ & $\sigma_4$ & $\cdots$ & $\sigma_{16}$ & rank \\ \hline
%		${}^{1}\Psi := \mu_1^{<1>}\| \cdot \|_1 + \mu_2^{<1>}\| \cdot \|_{\ast}$ & 6.269692 & 0.705257 & 0.304678 & 0.037810 & $\cdots$ & 0.000000 & 16 \\
%		$\mu_1^{<2>} (\| \cdot \|_1)_{B_1^{<2>}} + \mu_2^{<2>} \| \cdot \|_{\ast}$ & 6.268287 & 0.708376 & 0.317428 & 0.008368 & $\cdots$ & 0.000000 & 16 \\
%		$\mu_1^{<3>}\| \cdot \|_1 + \mu_2^{<3>}(\| \cdot \|_{\ast})_{B_2^{<3>}}$ & 6.412580 & 0.812003 & 0.411140 & 0.002035 & $\cdots$ & 0.000000 & 16 \\
%		$\mu_1^{<4>}(\| \cdot \|_1)_{B_1^{<4>}} + \mu_2^{<4>}(\| \cdot \|_{\ast})_{B_2^{<4>}}$ & 6.481346 & 0.893754 & 0.449620 & 0.005904 & $\cdots$ & 0.000000 & 16 \\ \hline
%	\end{tabular}
%\end{table}
Figure~\ref{fig:exam4_tunes} shows dependency of recovering performance on the parameter $(\mu_a, \mu_b)$ in Problem~\ref{prob:optim} and Example~3.
The performance is measured by mean squared error (MSE) defined as the average of SE in \eqref{eq:SE} over $100$ independent realizations of the additive noise.
From Figure~\ref{fig:exam4_tunes}, we can see that (i) the best weights of the penalties are respectively $(\mu_{a,{\rm \numI}}, \mu_{b, {\rm \numI}}) := (0.015, 0.1)$ for $\Psi_{\rm \numI} \circ \frakL$, $(\mu_{a, {\rm \numII}}, \mu_{b, {\rm \numII}}) := (0.03, 0.15)$ for $\Psi_{\rm \numII} \circ \frakL$, $(\mu_{a, {\rm \numIII}}, \mu_{b, {\rm \numIII}}) := (0.015, 0.15)$ for $\Psi_{\rm \numIII} \circ \frakL$, and $(\mu_{a, {\rm \numIV}}, \mu_{b, {\rm \numIV}}) := (0.035, 0.1)$ for $\Psi_{\rm \numIV} \circ \frakL$ and (ii) the estimations by $\Psi_{i} \circ \frakL$ with $(\mu_{a,i}, \mu_{b,i})$ $(i = {\rm \numII}, {\rm \numIII}, {\rm \numIV})$ outperform the convex penalty $\Psi_{\rm \numI} \circ \frakL$ with $(\mu_{a, {\rm \numI}}, \mu_{b, {\rm \numI}})$ in the context of MSE.

Figure~\ref{fig:exam4_trace} shows dependency of the SE on the number of iterations under weights $(\mu_{a,i}, \mu_{b,i})$ for $\Psi_{i} \circ \frakL$ $(i = {\rm \numI}, {\rm \numII}, {\rm \numIII}, {\rm \numIV})$.
The accuracy of the approximations by $\Psi_{i} \circ \frakL$ $(i = {\rm \numII}, {\rm \numIII}, {\rm \numIV})$ penalties become higher than the convex penalty $\Psi_{\rm \numI} \circ \frakL$ after 110 iterations and SE for $\Psi_{\rm \numII} \circ \frakL$, $\Psi_{\rm \numIII} \circ \frakL$, and $\Psi_{\rm \numIV} \circ \frakL$ reaches respectively 78.1\%, 43.1\%, and 12.7\% of SE for $\Psi_{\rm \numI} \circ \frakL$ in the end.

Figure~\ref{fig:exam4_estimates} shows the original matrix, an observed matrix, and recovered matrices by the penalties at $1,000$ iteration and Figure~\ref{fig:exam4_diffs} shows the difference between the original matrix and recovered matrices.
From Figure~\ref{fig:exam4_diffs}, the recovered matrix by $\Psi_{\rm \numIV} \circ \frakL$ approximates most accurately the original matrix.

\renewcommand{\sfT}{{*}}
\section{Conclusion}
In this paper, we have proposed
the Linearly involved Generalized Moreau Enhanced (LiGME) model as a unified extension of the ideas in [Zhang'10, Selesnick'17, Yin, Parekh, Selesnick'19] for
exploiting nonconvex penalties in the regularized least-squares models without losing their overall convexities. The proposed model can admit multiple nonconvex penalties without losing its overall convexity and thus is applicable to much broader scenarios including sparsity-rank-aware signal processing and machine learning.
We have also proposed a proximal splitting type algorithm for the LiGME model under an overall-convexity condition. The proposed algorithm is guaranteed to converge to a globally optimal solution. Numerical experiments in four different sparsity-rank-aware signal processing scenarios demonstrate the effectiveness of the LiGME models and the proposed proximal splitting algorithm.

\clearpage

\appendix
\def\thesection{Appendix \Alph{section}}

\section{Selesnick's algorithm for minimizing $J_{(\|\cdot\|_1)_B \circ \rmId}$}
\label{app:FB}
For problem~\eqref{intro:GMCL} with $\mathcal{X}=\mathcal{Z}=\mathbb{R}^n$, $\mathcal{Y}=\mathbb{R}^m$, and a special $B$, Selesnick presented an algorithm shown in Fact~\ref{fact:selesnick-forward-backward}.

\begin{fact}[{\cite[Proposition~15]{Selesnick2017-sparse}}]
  \label{fact:selesnick-forward-backward}
	Let $(A,B,\bmy,\mu,\theta) \in \bbR^{m \times n} \times \bbR^{q \times n} \times \bbR^{m} \times \bbR_{++} \times [0,1]$.
	Suppose that $B^{\sfT}B = (\theta / \mu) A^{\sfT}A$.
	Define $T_{\mathrm{Sel}} \colon \bbR^n \times \bbR^n \to \bbR^n \times \bbR^n \colon (\bmx, \bmv) \mapsto (\bmxi, \bmzeta)$ by
	\begin{align*}
		\bmxi &:= \Prox_{\tau \mu \| \cdot \|_1} \left[ \bmx - \tau A^{\sfT}(A(\bmx + \theta (\bmv - \bmx)) - \bmy) \right], \\
		\bmzeta &:= \Prox_{\tau \mu \| \cdot \|_1} \left[ \bmv - \tau \theta A^{\sfT}A(\bmv - \bmx) \right],
	\end{align*}
	where
	\begin{equation*}
		\tau \in \left( 0, \frac{2}{\max\left\{ 1, \frac{\theta}{1-\theta} \right\} \sqrt{\rho(A^{\sfT}A)}} \right).
	\end{equation*}
	Then, for any initial point $(\bmx_0, \bmv_0) \in \bbR^n \times \bbR^n$, the sequence $(\bmx_k)_{k \in \bbN} \subset \bbR^n$ generated by
	\begin{equation*}
		(\bmx_{k+1}, \bmv_{k+1}) = T_{\mathrm{Sel}} (\bmx_k, \bmv_k)
	\end{equation*}
	converges to a point in $\argmin\nolimits_{\bmx \in \bbR^n} J_{(\|\cdot\|_1)_B \circ \rmId}(\bmx)$.
\end{fact}

\section{Proof of Proposition~\ref{prop:exJ_cond}}
\label{app:prop1}
\noindent {\bf\underline{Proof of (a)}}: Fermat's rule \eqref{eq:Fermat} and the property \eqref{eq:subdifferentialinversion} of conjugate functions yield
$ B^{\sfT}B\mathfrak{L}x \in \argmin_{\hat{v}\in \mathcal{Z}}\Psi^*(\hat{v})
\IFF \partial \Psi^*(B^{\sfT}B\mathfrak{L}x) \ni 0_{\mathcal{Z}} \IFF B^{\sfT}B\mathfrak{L}x \in \partial \Psi(0_{\mathcal{Z}})
\IFF 0_{\mathcal{Z}} \in \partial \Psi(0_{\mathcal{Z}}) - B^{\sfT}B\mathfrak{L}x$. Since the sum rule \eqref{eq:sumrule} for $\partial \left(\Psi(\cdot)+\frac{1}{2}\|B(\cdot-\mathfrak{L}x)\|_{\widetilde{\mathcal{Z}}}^2\right)(0_{\mathcal{Z}})$ with \\ $\operatorname{dom}\left(\frac{1}{2}\|B(\cdot-\mathfrak{L}x)\|_{\widetilde{\mathcal{Z}}}^2\right) =\mathcal{Z}$ also yields
$ \partial \Psi(0_{\mathcal{Z}})+B^{\sfT}B(0_{\mathcal{Z}}-\mathfrak{L}x)=\partial \left(\Psi(\cdot)+\frac{1}{2}\|B(\cdot-\mathfrak{L}x)\|_{\widetilde{\mathcal{Z}}}^2\right)(0_{\mathcal{Z}})
$, we have $ B^{\sfT}B\mathfrak{L}x \in \argmin_{\hat{v}\in \mathcal{Z}}\Psi^*(\hat{v})
\IFF 0_{\mathcal{Z}} \in \partial \left(\Psi(\cdot)+\frac{1}{2}\|B(\cdot-\mathfrak{L}x)\|_{\widetilde{\mathcal{Z}}}^2\right)(0_{\mathcal{Z}}) \IFF 0_{\mathcal{Z}} \in \argmin_{\hat{v} \in \mathcal{Z}}\left(\Psi(\hat{v}) + \frac{1}{2}\|B(\hat{v} - \mathfrak{L}x)\|_{\widetilde{\mathcal{Z}}}^2 \right) \IFF \Psi_B\circ \mathfrak{L}(x) = \Psi(\mathfrak{L}x) - \left[ \Psi(0_{\mathcal{Z}}) +\frac{1}{2}\|B\mathfrak{L}x\|_{\widetilde{\mathcal{Z}}}^2\right]$, where the 2nd last equivalence is due to \eqref{eq:Fermat} and the last equivalence is by definition of $\Psi_B$.

% Thus, a simple calculation $0_{\mathcal{Z}} \in \argmin_{\hat{v} \in \mathcal{Z}}\left(\Psi(\hat{v}) + \frac{1}{2}\|B(\hat{v} - \mathfrak{L}x)\|_{\widetilde{\mathcal{Z}}}^2 \right) \IFF \Psi_B(\mathfrak{L}x) = \Psi(\mathfrak{L}x) - \left[ \Psi(0_{\mathcal{Z}}) +\frac{1}{2}\|B\mathfrak{L}x\|_{\widetilde{\mathcal{Z}}}^2\right]$ completes the proof.

\noindent {\bf\underline{Proof of (b)}}: We shall show $(C_1) \FI (C_2)$.
Fix $y \in \mathcal{Y}$ arbitrarily. Then we have, for every $x \in \mathcal{X}$,
\newcommand{\gv}{\psi_{v}}
\begin{align}
J_{\Psi_B \circ \mathfrak{L}}(x)&=\frac{1}{2} \| y - Ax \|_{\mathcal{Y}}^2 + \mu \Psi_B \circ \mathfrak{L} x  \label{eq:B1} \\
&=\frac{1}{2} \|y - Ax\|_{\mathcal{Y}}^2 + \mu \Psi(\mathfrak{L}x) - \mu \min_{v \in \mathcal{Z}} \left[ \Psi(v) + \frac{1}{2} \|B(\mathfrak{L}x - v)\|_{\widetilde{\mathcal{Z}}}^2 \right]  \nonumber \\
&= \left( \frac{1}{2} \|y\|_{\mathcal{Y}}^2 - \langle y, Ax \rangle_{\mathcal{Y}} + \frac{1}{2} \|Ax\|_{\mathcal{Y}}^2 \right) + \mu  \Psi(\mathfrak{L}x) \nonumber \\
&\qquad - \mu \min_{v \in \mathcal{Z}} \left[  \Psi(v) + \left( \frac{1}{2} \|Bv\|_{\widetilde{\mathcal{Z}}}^2 - \langle B v, B\mathfrak{L} x\rangle_{\widetilde{\mathcal{Z}}} +  \frac{1}{2} \|B\mathfrak{L}x\|_{\widetilde{\mathcal{Z}}}^2 \right) \right]  \nonumber \\
&= \frac{1}{2} \left( \|Ax\|_{\mathcal{Y}}^2 - \mu \|B\mathfrak{L}x\|_{\widetilde{\mathcal{Z}}}^2 \right) + \frac{1}{2} \|y\|_{\mathcal{Y}}^2 -  \langle y, Ax \rangle_{\mathcal{Y}} +\mu \Psi(\mathfrak{L}x) +  \mu\max_{v \in \mathcal{Z}} \gv(x)  \nonumber \\
                        &= \frac{1}{2} \langle x, (A^{\sfT}A - \mu \mathfrak{L}^{\sfT}B^{\sfT}B\mathfrak{L})x \rangle_{\mathcal{X}} + \frac{1}{2} \|y\|_{\mathcal{Y}}^2 - \langle y, Ax \rangle_{\mathcal{Y}} +\mu \Psi(\mathfrak{L}x)  +  \mu\max_{v \in \mathcal{Z}} \gv(x), \nonumber \\
  \label{eq:teagnoqeg} 
\end{align}
where
\begin{align}
  \label{eq:gvewhywe}  
  \gv \colon \mathcal{X} \to \bbR\colon x \mapsto -\left( \Psi(v) + \frac{1}{2} \|Bv\|_{\widetilde{\mathcal{Z}}}^2 -  \langle B v, B\mathfrak{L}x\rangle_{\widetilde{\mathcal{Z}}} \right).
\end{align}
Since $\gv$ is affine for every $v \in \mathcal{Z}$  and $\max_{v \in \mathcal{Z}} \gv(0_{\mathcal{X}}) \in \mathbb{R}$ due to $\dom \Psi =\mathcal{Z}$ and coercivity of $\Psi$, \cite[Proposition~9.3]{Bauschke2011-convex} yields $\max_{v \in \mathcal{Z}} \gv \in \Gamma_0(\mathcal{X})$. Moreover, %$\mathcal{X}\ni x \mapsto \frac{1}{2} \|y\|_{\mathcal{Y}}^2 - \langle y, Ax \rangle_{\mathcal{Y}} +\mu \Psi(\mathfrak{L}x)$ belongs to $\Gamma_0(\mathcal{X})$. By
the assumption $A^{\sfT}A - \mu \mathfrak{L}^{\sfT}B^{\sfT}B\mathfrak{L} \succeq O_{\mathcal{X}}$ ensures that the function $\mathcal{X} \ni x  \mapsto \frac{1}{2} \langle x, (A^{\sfT}A - \mu \mathfrak{L}^{\sfT}B^{\sfT}B\mathfrak{L})x \rangle_{\mathcal{X}}+\frac{1}{2} \|y\|_{\mathcal{Y}}^2 - \langle y, Ax \rangle_{\mathcal{Y}} +\mu \Psi(\mathfrak{L}x)$ also belongs to $\Gamma_0(\mathcal{X})$. Thus $J_{\Psi_B \circ \mathfrak{L}} \in \Gamma_0(\mathcal{X})$ holds. %Since $y \in \mathcal{Y}$ is chosen arbitrarily, the proof is completed.

Finally, since the affine function $\mathcal{X} \ni x  \mapsto \frac{1}{2} \|y\|_{\mathcal{Y}}^2 - \langle y, Ax \rangle_{\mathcal{Y}}$ in \eqref{eq:teagnoqeg} does not affect the convexity of $J_{\Psi_B \circ \mathfrak{L}}$,
we have
\begin{align*}
  (C_2) & \IFF   \frac{1}{2} \|A\cdot\|_{\cal{Y}}^2
          -\mu \|B\mathfrak{L}\cdot\|_{\mathcal{X}}^2 +\mu \Psi(\mathfrak{L} \cdot)  +  \mu\max_{v \in \mathcal{Z}} \gv(\cdot) \in \Gamma_0(\mathcal{X})
   \IFF  (C_3),
\end{align*}
    where the first equivalence holds by the expressions \eqref{eq:teagnoqeg} and \eqref{eq:B1}.

\noindent  {\bf\underline{Proof of (a')}}:
\begin{align}
  \label{eq:goala'}
  \opnorm{B^{\sfT}B\mathfrak{L}x}_* \leq 1 \Leftrightarrow  B^{\sfT}B\mathfrak{L}x \in \argmin_{\hat{v}\in \mathcal{Z}} \opnorm{\cdot}^*(\hat{v})
\end{align}
is verified by
$
   \opnorm{z}^* =
   \begin{cases}
     0 & \text{ if } \opnorm{z}_* \leq 1  \\
     +\infty & \text{ otherwise,}
\end{cases}
$
(see e.g. \cite[Ex. 3.26]{Boyd-Vandenberghe04}).

\noindent {\bf\underline{Proof of (b')}}: We shall show (C$_3$) $\FI$ (C$_1$) by contraposition.
Suppose $A^{\sfT}A - \mu \mathfrak{L}^{\sfT}B^{\sfT}B\mathfrak{L} \not \succeq \rmO_{\mathcal{X}}$, i.e., there exists $\hat{x} \in \mathcal{X} \setminus \{0_{\mathcal{X}}\}$ such that
\begin{equation}
\langle \hat{x}, (A^{\sfT}A - \mu \mathfrak{L}^{\sfT}B^{\sfT}B\mathfrak{L}) \hat{x} \rangle_{\mathcal{X}} < 0, \label{eq:exist_negative}
\end{equation}
and we shall prove $J_{\opnorm{\cdot}_B \circ \mathfrak{L}}^{(0)} \not \in \Gamma_0(\mathcal{X})$.
By $\mu \in \bbR_{++}$, we have
\begin{align}
  \text{\eqref{eq:exist_negative}}
  %& \ \IFF\ \|A\hat{x}\|_{\mathcal{Y}}^2 - \mu \| B\mathfrak{L}\hat{x}\|_{\widetilde{\mathcal{Z}}}^2 < 0 \
                                     \IFF \ \|B\mathfrak{L}\hat{x}\|_{\widetilde{\mathcal{Z}}}^2 > \frac{1}{\mu} \|A\hat{x}\|_{\mathcal{Y}}^2 \ \geq 0 \nonumber 
\end{align}
implying thus %$\mathfrak{L}\hat{x} \not \in \Null(B) = \Null(B^{\sfT}B)$, i.e.,
$B^{\sfT}B\mathfrak{L}\hat{x} \neq 0_{\mathcal{Z}}$
and $\opnorm{B^{\sfT}B\mathfrak{L}\bar{x}}_*=1$ for $\bar{x} := (\opnorm{B^{\sfT}B\mathfrak{L}\hat{x}}_*)^{-1}\hat{x} \in \mathcal{X}$.
%Let $\mathfrak{c}:=\sup_{w \in \mathcal{Z}\colon \Psi(w)\leq 1}|\langle w, B^{\sfT}BL\bar{x}\rangle|>0$ and
%Define $\bar{x} := \opnorm{B^{\sfT}B\mathfrak{L}\hat{x}}_*^{-1}\hat{x} \in \mathcal{X}$. Since $\opnorm{B^{\sfT}B\mathfrak{L}\bar{x}}_*=1$ holds,
The statement (a') %(i.e., $\Psi_B\circ \mathfrak{L}(x) = \Psi(\mathfrak{L}x) - \frac{1}{2}\|B\mathfrak{L}x\|_{\widetilde{\mathcal{Z}}}^2 \IFF \opnorm{B^{\sfT}B\mathfrak{L}x}_*\leq 1$)
yields  
\begin{align}
(\forall \lambda \in (0,1)) \  J_{\opnorm{\cdot}_B \circ \mathfrak{L}}^{(0)}(\lambda \bar{x}) &=\frac{1}{2} \|A(\lambda \bar{x})  \|_{\mathcal{Y}}^2 + \mu \opnorm{\mathfrak{L}(\lambda \bar{x})} - \frac{\mu}{2} \|B\mathfrak{L}(\lambda\bar{x}) \|_{\widetilde{\mathcal{Z}}}^2  \nonumber \\
\relax &= \frac{1}{2} \left( \|A\bar{x}\|_{\mathcal{Y}}^2 - \mu \|B\mathfrak{L}\bar{x}\|_{\widetilde{\mathcal{Z}}}^2 \right) \lambda^2 +  \mu \opnorm{\mathfrak{L}\bar{x}}  \lambda
\nonumber %\label{eq:GMC_theta}
\end{align}
and $J_{\opnorm{\cdot}_B \circ \mathfrak{L}}^{(0)}(0_{\calX})=  \frac{1}{2} \left( \|A0_{\calX}\|_{\mathcal{Y}}^2 - \mu \|B\mathfrak{L}0_{\calX}\|_{\widetilde{\mathcal{Z}}}^2 \right)  +  \mu \opnorm{\mathfrak{L}0_{\calX}} =0$,
from %\eqref{eq:def_barw} and 
which we have
\begin{align}
 \frac{1}{2} J_{\opnorm{\cdot}_B \circ \mathfrak{L}}^{(0)}(0_{\mathcal{X}}) + \frac{1}{2} J_{\opnorm{\cdot}_B \circ \mathfrak{L}}^{(0)}(\bar{x}) - J_{\opnorm{\cdot}_B \circ \mathfrak{L}}^{(0)} \left( \frac{0_{\mathcal{X}} +  \bar{x}}{2} \right) %  \nonumber \\
= \frac{1}{8} \left( \|A\bar{x}\|_{\mathcal{Y}}^2 - \mu \|B\mathfrak{L}\bar{x}\|_{\widetilde{\mathcal{Z}}}^2 \right) <0. \nonumber 
%&= \frac{\lambda (1-\lambda)}{2\opnorm{B^{\sfT}B\mathfrak{L}\hat{x}}_*^2} \langle \hat{x}, (A^{\sfT}A - \mu \mathfrak{L}^{\sfT}B^{\sfT}B\mathfrak{L}) \hat{x}\rangle_{\mathcal{X}} < 0, \nonumber 
\end{align}
%which shows the nonconvexity of $J_{\opnorm{\cdot}_B \circ \mathfrak{L}}$.
\hfill \qed

\section{Proof of Lemma~\ref{lem:qualification}}
\label{app:qualification}
We will show
\begin{align}
  \nonumber 
  {\rm span}\left(\dom \left(\left(\Psi+ \frac{1}{2} \|B \cdot \|_{\widetilde{\mathcal{Z}}}^2\right)^*\right) - \ran(B^*)\right) \subset {\rm cone}\left(\dom \left(\left(\Psi+ \frac{1}{2} \|B \cdot \|_{\widetilde{\mathcal{Z}}}^2\right)^*\right) - \ran(B^*)\right),%\\
%    \label{eq:goalPropqual}
\end{align}
which is equivalent to \eqref{eq:qualification}.
\noindent By the even symmetry of $\Psi$, we have for $v \in \calZ$
\begin{align*}
  \left(\Psi+ \frac{1}{2} \|B \cdot \|_{\widetilde{\mathcal{Z}}}^2\right)^*(-v)
  = \sup_{w \in \mathcal{Z}}\left(
  \langle w, -v\rangle_{\mathcal{Z}} - \Psi(w) - \frac{1}{2} \|B w \|_{\widetilde{\mathcal{Z}}}^2 
  \right) % \\
  % =& \sup_{w \in \mathcal{Z}}\left(
  % \langle v, w\rangle_{\mathcal{Z}} - \Psi(-w) - \frac{1}{2} \|B (-w) \|_{\widetilde{\mathcal{Z}}}^2
  % \right)
  = \left(\Psi+ \frac{1}{2} \|B \cdot \|_{\widetilde{\mathcal{Z}}}^2\right)^*(v)
\end{align*}
and 
\begin{align}
 \label{eq:evendom}
  \mathcal{D}:=\dom \left(\left(\Psi+ \frac{1}{2} \|B \cdot \|_{\widetilde{\mathcal{Z}}}^2\right)^*\right)=-\dom \left(\left(\Psi+ \frac{1}{2} \|B \cdot \|_{\widetilde{\mathcal{Z}}}^2\right)^*\right).
\end{align}
Let $v  = \sum_{i \in I} \alpha_i (v_i-w_i) \in {\rm span}\left(\mathcal{D} - \ran(B^*)\right)$ for some $(\alpha_i,v_i,w_i)_{i \in I} \subset (\mathbb{R}\setminus\{0\}) \times \mathcal{D} \times \ran(B^*)$ with finite $I \subset \mathbb{N}$. Then we have 
\begin{align}
  v = \! \sum_{i \in I} |\alpha_i| ({\rm sgn}(\alpha_i) v_i) \! -\! \sum_{l \in I} \alpha_l w_l
  = \!  \sum_{\iota\in I}|\alpha_\iota| \left( \! \sum_{i \in I} \frac{|\alpha_i|}{\sum_{\iota\in I}|\alpha_\iota|} ({\rm sgn}(\alpha_i) v_i)\!  -\! \frac{\sum_{l \in I} \alpha_l w_l}{\sum_{\iota\in I}|\alpha_\iota|}\! \right)\! \!,
  \label{eq:ri}
\end{align}
where $\frac{\sum_{l \in I} \alpha_l w_l}{\sum_{\iota\in I}|\alpha_\iota|} \in \ran(B^*)$. Moreover, by ${\rm sgn}(\alpha_i) v_i \in \mathcal{D}$ $(i \in I)$ due to \eqref{eq:evendom} and by $\sum_{i \in I} \frac{|\alpha_i|}{\sum_{\iota\in I}|\alpha_\iota|} ({\rm sgn}(\alpha_i) v_i) \in \mathcal{D}$ due to the convexity of $\mathcal{D}$, \eqref{eq:ri} implies $v \in  {\rm cone}\left(\mathcal{D} - \ran(B^*)\right)$.

\section{Proof of Theorem~\ref{def:Tprop}} \label{app:theo1}

{\bf \underline{Proof of (a)}:}
Recall that, 
under the assumption in Problem~1,
\eqref{eq:teagnoqeg} gives an expression of
%Note that, in advance, 
%The summary of the proof of (a) is as follows.
%The statement (a) is proven by utilizing properties of subdifferential in Section~\ref{sec:elem_convex}.
                                                                                                           %                                                                                                            as shown in \eqref{eq:teagnoqeg},
$J_{\Psi_B \circ \mathfrak{L}}$ as a sum of convex functions.
%         \begin{align}[0mm]
%           \label{eq:Jconvex}
%   J_{\Psi_B \circ \mathfrak{L}}(\bmx) &= \frac{1}{2} \langle \bmx, (A^{\sfT}A - \mu \mathfrak{L}^{\sfT}B^{\sfT}B\mathfrak{L})\bmx\rangle_{\mathcal{X}}  + \frac{1}{2} \|\bmy\|_{\mathcal{Y}}^2 - \langle \bmy, A\bmx\rangle_{\mathcal{Y}} +\mu \Psi(\mathfrak{L}\bmx)  +  \mu\max_{\bmv \in \mathcal{Z}} \gv(\bmx) \nonumber \\
% \end{align}
% ($ \gv \colon \mathcal{X} \to \bbR\colon x \mapsto -\left( \Psi(v) + \frac{1}{2} \|Bv\|_{\widetilde{\mathcal{Z}}}^2 -  \langle B v, B\mathfrak{L}x\rangle_{\widetilde{\mathcal{Z}}} \right)$ is defined in \eqref{eq:gvewhywe}).
%Fermat's rule \eqref{eq:Fermat} characterizes the solution set $\mathcal{S}$ of Problem~\ref{prob:optim} as $\mathcal{S}=\{ \bmx^{\star} \in \mathcal{X} \mid 0_{\mathcal{X}} \in \partial J_{\Psi_B \circ \mathfrak{L}}(\bmx^{\star}) \}$.
The proof of (a) is decomposed into two steps. \\
{\bf (Step 1)}  By applying properties of the subdifferential in Section~\ref{sec:elem_convex}, %and Claim~\ref{eq:newchain} below,
we will derive
an alternative characterization of
$\mathcal{S}=\{ \bmx^{\star} \in \mathcal{X} \mid 0_{\mathcal{X}} \in \partial J_{\Psi_B \circ \mathfrak{L}}(\bmx^{\star}) \}$
 in terms of zeros of the sum of an affine operator $F$ and a set-valued operator $G$ involving $\partial \Psi$ (see Claim~\ref{claim:translation}).
%\newcommand{\infpostcomp}{\rhd \hspace{-9.4pt}\cdot\hspace{2.5mm}}
% \begin{claim}
%   \label{eq:newchain}
%   Let $\mathcal{H}$ be finite dimensional real Hilbert spaces, let $g \in \Gamma_0(\mathcal{Z})$, and let $H \in  \mathcal{B}(\mathcal{H},\mathcal{Z})$ be such that $\ran H \cap \dom g  \not = \varnothing$.
%   Then $\partial ( g\circ H ) \circ H^* = ( H^*\circ  \partial g  \circ H) \circ H^*$.
% \end{claim}
\begin{claim}
	\label{claim:translation}
                                                                                                           %                                                                                                            $0_{\mathcal{Z}} \in \operatorname{ri}\left(\dom \left(\Psi+ \frac{1}{2} \|B \cdot \|_{\widetilde{\mathcal{Z}}}^2\right)^* - \ran(B^*)  \right)$.
        In Problem~1,
        for any $\bmx^{\star} \in \calX$, we have 
        $\bmx^{\star} \in \calS$ if and only if there exists $(\bmv^{\star}, \bmw^{\star}) \in \mathcal{Z} \times \mathcal{Z}$ s.t. 
	\begin{equation}
	(\bmzero_{\mathcal{X}}, \bmzero_{\mathcal{Z}}, \bmzero_{\mathcal{Z}}) \in F(\bmx^{\star}, \bmv^{\star}, \bmw^{\star}) + G(\bmx^{\star}, \bmv^{\star}, \bmw^{\star}), \label{eq:zero_of_FG}
	\end{equation}
	where $F \colon \calR \to \calR$ and $G \colon \calR \to 2^{\calR}$ are defined as
	\begin{align}
	&F(\bmx, \bmv, \bmw) := \left( (A^{\sfT}A - \mu \mathfrak{L}^{\sfT}B^{\sfT}B\mathfrak{L})\bmx - A^{\sfT}\bmy, \mu B^{\sfT}B \bmv, \bmzero_{\mathcal{Z}}  \right),  \nonumber \\
	&G(\bmx, \bmv, \bmw) := \!\{ \mu \mathfrak{L}^{\sfT}B^{\sfT}B \bmv \!+\! \mu \mathfrak{L}^{\sfT}\bmw \} \!\times\! (-\mu B^{\sfT}B\mathfrak{L}\bmx \!+\! \mu \partial \Psi(\bmv))\!\times\! (-\mu \mathfrak{L} \bmx \!+\! \mu \partial \Psi^*(\bmw) ). \nonumber 
	\end{align}
      \end{claim}
      \noindent {\bf (Step 2)}
      By using $\mathfrak{P}$ in \eqref{eq:S_def},
      we will confirm for any $x^{\star} \in \mathcal{X}$ that 
%      can be expressed as $\Tlcp=(\rmId + P_{\epsilon}^{-1} \circ G)^{-1}\circ (\rmId - P_{\epsilon}^{-1} \circ F)$, which shall be confirmed in the proof of Theorem 1(b).
\begin{align}
& (\bmx^{\star},\bmv^{\star},\bmw^{\star}) \in \Fix(\Tlcp) 
\IFF \Tlcp(\bmx^{\star},\bmv^{\star},\bmw^{\star}) = (\bmx^{\star},\bmv^{\star},\bmw^{\star})  \nonumber \\
  \IFF&  (\mathfrak{P} \!-\! F)(\bmx^{\star},\!\bmv^{\star},\!\bmw^{\star}) \!\in\! (\mathfrak{P} \!+\! G)(\bmx^{\star},\!\bmv^{\star},\!\bmw^{\star}) \quad (\IFF \eqref{eq:zero_of_FG}) % \nonumber \\  
%  \IFF& \text{$(\bmx^{\star}, \!\bmv^{\star}, \!\bmw^{\star}) \!\in\! \calR$ satisfies \eqref{eq:zero_of_FG}}
        \label{eq:claimC2outline}
\end{align}
implying thus, with Claim \ref{claim:translation}, $\bmx^{\star} \in \calS$ if and only if there exists $(\bmv^{\star}, \bmw^{\star}) \in \mathcal{Z} \times \mathcal{Z}$ such that $(\bmx^{\star}, \bmv^{\star}, \bmw^{\star}) \in \Fix(\Tlcp)$.
%from which we obtain (a), i.e., $\calS = \Xi (\Fix(\Tlcp))$.

For proof of Claim~\ref{claim:translation}, we will use, in \eqref{eq:teagnoqeg} and \eqref{eq:gvewhywe}, 
\begin{align}
  (x \in \calX)\quad & \max_{\bmv \in \mathcal{Z}} \gv(x)
                       =\max_{\bmv \in \mathcal{Z}} \left(\langle \bmv, B^{\sfT}B\mathfrak{L}\bmx\rangle_{\mathcal{Z}} - \Psi(\bmv) - \frac{1}{2} \|B\bmv\|_{\widetilde{\mathcal{Z}}}^2 \right) \nonumber \\
                     & \phantom{\max_{\bmv \in \mathcal{Z}} \gv(x)}= \left[\left(\Psi+ \frac{1}{2} \|B \cdot \|_{\widetilde{\mathcal{Z}}}^2\right)^*\circ B^{\sfT} \right] \circ B\mathfrak{L}(\bmx),
                         \label{eq:mgvconjugate}
                       %                                                                                                            \\
%                       = \left(\Psi+ \frac{1}{2} \|B \cdot \|_{\widetilde{\mathcal{Z}}}^2\right)^*\circ B^{\sfT}B\mathfrak{L}, \\
\end{align}
and
\begin{align}
    \label{eq:domPB}
  & \operatorname{dom}\left(
\left(\Psi+ \frac{1}{2} \|B \cdot \|_{\widetilde{\mathcal{Z}}}^2\right)^*\circ B^{\sfT}
    \right)= \widetilde{\mathcal{Z}},
\end{align}
where \eqref{eq:domPB} is verified, with the coercivity of $\Psi$, by
\begin{align}
  (\bmz \in \widetilde{\mathcal{Z}}) \quad &
                                             \left(\Psi+ \frac{1}{2} \|B \cdot \|_{\widetilde{\mathcal{Z}}}^2 \right)^*(B^{\sfT}{\bmz}) = \sup_{\bmv \in \mathcal{Z}} \left(\langle \bmv, B^{\sfT}\bmz\rangle_{\mathcal{Z}} - \Psi(\bmv) - \frac{1}{2} \|B\bmv\|_{\widetilde{\mathcal{Z}}}^2 \right) \nonumber \\
   & \leq \sup_{\bmv \in \mathcal{Z}} \left( - \Psi(\bmv)  \right)
 + \sup_{\bmv \in \mathcal{Z}} \left(\langle \bmv, B^{\sfT}\bmz\rangle_{\mathcal{Z}} - \frac{1}{2} \|B\bmv\|_{\widetilde{\mathcal{Z}}}^2 \right) \nonumber \\
  & =
  \max_{\bmv \in \mathcal{Z}} \left( - \Psi(\bmv)  \right)
  + \max_{\bmv \in \mathcal{Z}} \left(\langle B\bmv, \bmz\rangle_{\widetilde{\mathcal{Z}}} - \frac{1}{2} \|B\bmv\|_{\widetilde{\mathcal{Z}}}^2 \right) \nonumber
  < \infty.
\end{align}

Now, we shall prove Step 1 and Step 2.
\noindent {\bf Step 1: Proof of Claim~\ref{claim:translation}}. 
% Recall that \eqref{eq:teagnoqeg} shows
%   \begin{align}[0mm]
%     J_{\Psi_B \circ \mathfrak{L}}(\bmx) &= \frac{1}{2} \langle \bmx, (A^{\sfT}A - \mu \mathfrak{L}^{\sfT}B^{\sfT}B\mathfrak{L})\bmx\rangle_{\mathcal{X}}  + \frac{1}{2} \|\bmy\|_{\mathcal{Y}}^2 - \langle \bmy, A\bmx\rangle_{\mathcal{Y}} +\mu \Psi(\mathfrak{L}\bmx)  +  \mu\max_{\bmv \in \mathcal{Z}} \gv(\bmx) \nonumber 
% \end{align}
% ($\gv$ is defined in \eqref{eq:gvewhywe}).
  Since the first three terms of the RHS of \eqref{eq:teagnoqeg} are differentiable over $\mathcal{X}$, the sum rule \eqref{eq:sumrule} implies
%Since $\operatorname{dom}\left(\frac{1}{2} \bmx^{\sfT} (A^{\sfT}A - \mu L^{\sfT}B^{\sfT}BL)\bmx  + \frac{1}{2} \|\bmy\|^2 - \bmy^{\sfT}A\bmx \right) = \mathcal{X}$
\begin{align}
  \partial J_{\Psi_B \circ \mathfrak{L}} (\bmx) &= \nabla \left(\frac{1}{2} \langle \bmx, (A^{\sfT}A - \mu \mathfrak{L}^{\sfT}B^{\sfT}B\mathfrak{L})\bmx\rangle_{\mathcal{X}}  + \frac{1}{2} \|\bmy\|_{\mathcal{Y}}^2 - \langle \bmy, A\bmx \rangle_{\mathcal{Y}} \right) + \partial \left(\mu \Psi \circ \mathfrak{L}  +  \mu\max_{\bmv \in \mathcal{Z}} \gv\right)(\bmx) \nonumber \\
                                     &= (A^{\sfT}A - \mu \mathfrak{L}^{\sfT}B^{\sfT}B\mathfrak{L})\bmx - A^{\sfT}\bmy
                                       +\mu \partial \left( \Psi \circ \mathfrak{L}  + \max_{\bmv \in \mathcal{Z}} \gv\right)(\bmx).
 \label{eq:ibuqehrh}
\end{align}
Moreover, by $\operatorname{dom}(\max_{\bmv \in \mathcal{Z}} \gv)=\mathcal{X}$ due to \eqref{eq:mgvconjugate} and \eqref{eq:domPB}, the sum rule \eqref{eq:sumrule} decomposes 
%shows $\partial \left(\Psi \circ \mathfrak{L}  +  \max_{\bmv \in \mathcal{Z}} \gv\right) =
                                       %                                        \partial \left(\Psi \circ \mathfrak{L}\right)  +  \partial \left( \max_{\bmv \in \mathcal{Z}} \gv\right)$. Thus \eqref{eq:ibuqehrh} yields
\eqref{eq:ibuqehrh} as 
\begin{align}
  \partial J_{\Psi_B \circ \mathfrak{L}} (\bmx)
  &= (A^{\sfT}A - \mu \mathfrak{L}^{\sfT}B^{\sfT}B\mathfrak{L})\bmx - A^{\sfT}\bmy
                                       +\mu \partial \left( \Psi \circ \mathfrak{L}\right)(\bmx)  + \mu \partial \left(\max_{\bmv \in \mathcal{Z}} \gv\right)(\bmx).                
 \label{eq:ibuqehrh3}
\end{align}  
Apply the chain rule \eqref{eq:chainrule} to $\partial (\Psi \circ \mathfrak{L})$ with $\operatorname{dom}( \Psi)=\mathcal{Z}$ for simplification 
%$\partial \left( \Psi \circ \mathfrak{L}\right)= \mathfrak{L}^{\sfT} \circ \partial \Psi  \circ \mathfrak{L}$, we have 
\begin{align}
  \partial J_{\Psi_B \circ \mathfrak{L}} (\bmx)
  &= (A^{\sfT}A - \mu \mathfrak{L}^{\sfT}B^{\sfT}B\mathfrak{L})\bmx - A^{\sfT}\bmy
    +\mu \mathfrak{L}^{\sfT}\partial \Psi(\mathfrak{L}\bmx)  + \mu \partial \left(\max_{\bmv \in \mathcal{Z}} \gv\right)(\bmx).                
 \label{eq:ibuqehrh4}
\end{align}
%Recall that  $0 \in \operatorname{ri}\left(
%      \operatorname{dom}\left(\Psi+ \frac{1}{2} \|B \cdot \|^2\right)^*
%      - \operatorname{ran}\left(B^{\sfT}\right)
%    \right)$ is assumed.    
Apply again 
%Equation \eqref{eq:mgvconjugate} and
the chain rule \eqref{eq:chainrule} to \eqref{eq:mgvconjugate} 
%$\partial \left(\max_{\bmv \in \mathcal{Z}} \gv\right)
%  = \partial\left(\left[\left(\Psi+ \frac{1}{2} \|B \cdot \|_{\widetilde{\mathcal{Z}}}^2\right)^*\circ B^{\sfT}\right] \circ B\mathfrak{L} \right)
    %     $
with \eqref{eq:domPB} for
  \begin{align}
    \partial \left(\max_{\bmv \in \mathcal{Z}} \gv\right) = (B\mathfrak{L})^{\sfT}\partial \left[\left(\Psi+ \frac{1}{2} \|B \cdot \|_{\widetilde{\mathcal{Z}}}^2\right)^*\circ B^{\sfT}\right] \circ B\mathfrak{L},
    \label{eq:C7dash}
  \end{align}
  and to $\partial \left[\left(\Psi+ \frac{1}{2} \|B \cdot \|_{\widetilde{\mathcal{Z}}}^2\right)^*\circ B^{\sfT}\right]$ in \eqref{eq:C7dash} with \eqref{eq:qualification} to deduce further simplification 
  
%     Furthermore, thanks to Lemma~\ref{lem:qualification}, the chain rule \eqref{eq:chainrule} for $\partial \left[\left(\Psi+ \frac{1}{2} \|B \cdot \|_{\widetilde{\mathcal{Z}}}^2\right)^*\circ B^{\sfT}\right]$ guarantees 
%       \begin{align*}
%         \partial \left[\left(\Psi+ \frac{1}{2} \|B \cdot \|_{\widetilde{\mathcal{Z}}}^2\right)^*\circ B^{\sfT}\right]
%         =B  \partial \left(\Psi+ \frac{1}{2} \|B \cdot \|_{\widetilde{\mathcal{Z}}}^2\right)^* \circ B^{\sfT}.
%       \end{align*}
% These yield $\partial \left(\max_{\bmv \in \mathcal{Z}} \gv\right)
%       =(B^{\sfT}B\mathfrak{L})^{\sfT} \circ \partial \left(\Psi+ \frac{1}{2} \|B \cdot \|_{\widetilde{\mathcal{Z}}}^2 \right)^* \circ (B^{\sfT}B\mathfrak{L})$.
% Hence \eqref{eq:ibuqehrh4} results in 
\begin{align}
  \partial J_{\Psi_B \circ \mathfrak{L}} (\bmx)
  &= (A^{\sfT}A - \mu \mathfrak{L}^{\sfT}B^{\sfT}B\mathfrak{L})\bmx - A^{\sfT}\bmy
    +\mu \mathfrak{L}^{\sfT}\partial \Psi(\mathfrak{L}\bmx)  + \mu
    (B^{\sfT}B\mathfrak{L})^{\sfT} \partial \left(\Psi+ \frac{1}{2} \|B \cdot \|_{\widetilde{\mathcal{Z}}}^2 \right)^*(B^{\sfT}B\mathfrak{L} \bmx).
  \nonumber \\
  \label{eq:prob_grad_sum}
\end{align}
Furthermore, by 
% Recall that $x^{\star} \in \mathcal{S} \IFF  \bmzero_{\mathcal{X}} \in  \partial J_{\Psi_B \circ \mathfrak{L}} (\bmx^{\star})$.
% Since
%the property \eqref{eq:subdifferentialinversion} of conjugate functions yields
$\bmw^{\star} \in \partial \Psi(\mathfrak{L}\bmx^{\star}) \IFF \mathfrak{L}\bmx^{\star} \in \partial \Psi^{\ast} ( \bmw^{\star} )$ and
%as well as
$\bmv^{\star} \in \partial \left(\Psi+ \frac{1}{2} \|B \cdot \|_{\widetilde{\mathcal{Z}}}^2 \right)^*(B^{\sfT}B\mathfrak{L} \bmx^{\star})
\IFF B^{\sfT}B\mathfrak{L} \bmx^{\star} \in  \partial \left(\Psi+ \frac{1}{2} \|B \cdot \|_{\widetilde{\mathcal{Z}}}^2 \right)(\bmv^{\star})=\partial \Psi(\bmv^{\star}) + B^{\sfT}B\bmv^{\star} [\text{due to the property \eqref{eq:subdifferentialinversion} and}$ $\text{the sum rule \eqref{eq:sumrule} with } \operatorname{dom}(\|B \cdot \|_{\widetilde{\mathcal{Z}}}^2) =\mathcal{Z}]$,
 we deduce from \eqref{eq:prob_grad_sum} %and
%$[\bmx^{\star} \in \calS \IFF \bmzero_n \in  \partial J_{\Psi_B \circ L} (\bmx^{\star})]$
	\begin{align}
	\relax & \ \bmx^{\star} \in \calS[ \IFF  \bmzero_{\mathcal{X}} \in  \partial J_{\Psi_B \circ \mathfrak{L}} (\bmx^{\star})] \nonumber \\
	\IFF & \  
	\left\{
               \begin{array}{l}
                 \bmzero_{\mathcal{X}} = (A^{\sfT}A - \mu \mathfrak{L}^{\sfT}B^{\sfT}B\mathfrak{L})\bmx^{\star} - A^{\sfT}\bmy  +\mu \mathfrak{L}^{\sfT} \bmw^{\star} + \mu(B^{\sfT}B\mathfrak{L})^{\sfT}\bmv^{\star}, \\                 
                 B^{\sfT}B\mathfrak{L} \bmx^{\star} \in \partial \Psi(\bmv^{\star}) + B^{\sfT}B\bmv^{\star}, \\
                 \mathfrak{L}\bmx^{\star} \in \partial \Psi^{\ast}( \bmw^{\star} )
	\end{array}
	\right.
	\nonumber \\
	\IFF & \
	\left\{
	\begin{array}{l}
	\bmzero_{\mathcal{X}} = (A^{\sfT}A - \mu \mathfrak{L}^{\sfT}B^{\sfT}B\mathfrak{L})\bmx^{\star} - A^{\sfT}\bmy + \mu \mathfrak{L}^{\sfT}B^{\sfT}B\bmv^{\star} + \mu \mathfrak{L}^{\sfT}\bmw^{\star},  \\
	\bmzero_{\mathcal{Z}} \in - \mu B^{\sfT}B\mathfrak{L}\bmx^{\star} + \mu B^{\sfT}B\bmv^{\star} + \mu \partial \Psi(\bmv^{\star}),  \\
	\bmzero_{\mathcal{Z}} \in -\mu \mathfrak{L}\bmx^{\star} + \mu \partial \Psi^{\ast}( \bmw^{\star})
	\end{array}
	\right.
	\nonumber \\
	\IFF &  (\bmzero_{\mathcal{X}}, \bmzero_{\mathcal{Z}}, \bmzero_{\mathcal{Z}}) \in F(\bmx^{\star}, \bmv^{\star}, \bmw^{\star}) + G(\bmx^{\star}, \bmv^{\star}, \bmw^{\star}) \nonumber 
	\end{align}
        which completes the proof of Claim~\ref{claim:translation}.

\noindent {\bf Step 2:} \eqref{eq:claimC2outline} is verified by the definitions of $\Tlcp$ and $\mathfrak{P}$ in Theorem~\ref{def:Tprop} as
\begin{align}
\relax & \ \Tlcp(\bmx,\bmv,\bmw) = (\bmxi, \bmzeta, \bmeta) \nonumber\\
\IFF & \
\left\{
\begin{array}{l}
\left[ \sigma \rmId  - (A^{\sfT}A - \mu \mathfrak{L}^{\sfT}B^{\sfT}B\mathfrak{L}) \right] \bmx - \mu \mathfrak{L}^{\sfT}B^{\sfT}B \bmv - \mu \mathfrak{L}^{\sfT}\bmw + A^{\sfT}\bmy = \sigma \bmxi,   \\
2\mu B^{\sfT}B\mathfrak{L} \bmxi - \mu B^{\sfT}B\mathfrak{L} \bmx + (\tau \rmId - \mu B^{\sfT}B)\bmv  \in \left[ \tau \rmId + \mu \partial \Psi( \cdot ) \right] (\bmzeta),   \\
2\mu \mathfrak{L}\bmxi - \mu \mathfrak{L}\bmx + \mu \bmw \in (\mu \rmId + \mu \partial \Psi^{\ast})(\bmeta) 
\end{array}
\right.
  \nonumber \\
  % \IFF & \
  %        \left\{
  %        \begin{array}{l}
  %           [\sigma \rmI_n \bmx - \mu L^{\sfT}B^{\sfT}B \bmv- \mu L^{\sfT}\bmw]
  %          +[-(A^{\sfT}A - \mu L^{\sfT}B^{\sfT}BL) \bmx + A^{\sfT}\bmy] \\
  %          \quad = [\sigma \bmxi-\mu L^{\sfT}B^{\sfT}B \bmzeta \!-\! \mu L^{\sfT}\bmeta]+[ \mu L^{\sfT}B^{\sfT}B \bmzeta \!+\! \mu L^{\sfT}\bmeta],   \\
  %          \left[-\mu B^{\sfT}BL \bmx + \tau \rmI_{\mathcal{Z}}\bmv\right] +\left[- \mu B^{\sfT}B\bmv\right] \\
  %          \quad \in \left[-\mu B^{\sfT}BL \bmxi+\tau \rmId (\bmzeta)\right] + \left[-\mu B^{\sfT}BL \bmxi+\mu  \partial \Psi(\bmzeta)\right],   \\  
  %          \left[- \mu L\bmx + \mu \bmw\right] + \left[\bmzero_{\mathcal{Z}}\right] \\
  %          \quad \in  \left[- \mu L\bmxi + \mu \rmId \bmeta\right] +\left[- \mu L\bmxi+ \mu \partial \Psi^{\ast}(\bmeta) \right]
  %        \end{array}
  % \right.
  % \nonumber \\
  \IFF & \ (\mathfrak{P} - F)(\bmx,\bmv,\bmw) \in (\mathfrak{P} + G)(\bmxi,\bmzeta,\bmeta), \label{eq:Tlcp_PF_PG}
\end{align}
where we used the expression of the proximity operator as the resolvent of a subdifferential.
%the relations in \eqref{eq:claimC2outline} hold immediately.
% we have
% \begin{align}[5mm]
% & (\bmx^{\star},\bmv^{\star},\bmw^{\star}) \in \Fix(\Tlcp) 
% [\IFF \Tlcp(\bmx^{\star},\bmv^{\star},\bmw^{\star}) = (\bmx^{\star},\bmv^{\star},\bmw^{\star})]  \nonumber \\
% \IFF&  (\mathfrak{P} \!-\! F)(\bmx^{\star}, \bmv^{\star}, \bmw^{\star}) \!\in\! (\mathfrak{P} \!+\! G)(\bmx^{\star}, \bmv^{\star}, \bmw^{\star}) 
% \!\IFF\! \text{($(\bmx^{\star}, \bmv^{\star}, \bmw^{\star}) \!\in\! \calR$ satisfies \eqref{eq:zero_of_FG})}
% \label{eq:claimC2}
% \end{align}
% which verifies the relations in \eqref{eq:claimC2outline}.

\noindent {\bf \underline{Proof of (b)}:} We first prove $\mathfrak{P} \succ \rmO_{\calR}$ under the condition \eqref{eq:stepsize_condition}.
The Schur complement (see e.g. \cite[Theorem 7.7.6]{horn2012matrix}) yields
\begin{align}
 \mathfrak{P} \succ \rmO_{\calR}
&\IFF \sigma \rmId - \begin{bmatrix}
 -\mu \mathfrak{L}^{\sfT} B^{\sfT} B & -\mu \mathfrak{L}^{\sfT}
\end{bmatrix} {\begin{bmatrix}
	\tau \rmId & \rmO_{\mathcal{Z}} \\
	\rmO_{\mathcal{Z}} & \mu \rmId
	\end{bmatrix}}^{-1} \begin{bmatrix}
-\mu B^{\sfT}B\mathfrak{L} \\
-\mu \mathfrak{L}
\end{bmatrix}  \succ \rmO_{\calX}  \nonumber \\
&\IFF \sigma \rmId - \frac{\mu^2}{\tau} \mathfrak{L}^{\sfT} {(B^{\sfT}B)}^2 \mathfrak{L} - \mu \mathfrak{L}^{\sfT}\mathfrak{L} \succ \rmO_{\calX} \nonumber \\
&\IFF \left( \sigma \rmId - \frac{\kappa}{2} A^{\sfT}A - \mu \mathfrak{L}^{\sfT}\mathfrak{L} \right) + \left( \frac{\kappa}{2}A^{\sfT}A - \frac{\mu^2}{\tau} \mathfrak{L}^{\sfT} {(B^{\sfT}B)}^2 \mathfrak{L} \right)  \succ \rmO_{\calX}. \nonumber 
\end{align}
From the condition \eqref{eq:stepsize_condition}, it is sufficient to show that $\frac{\kappa}{2}A^{\sfT}A - \frac{\mu^2}{\tau} \mathfrak{L}^{\sfT} {(B^{\sfT}B)}^2 \mathfrak{L} \succeq \rmO_{\mathcal{X}}$.
Recalling $\|B^{\sfT}B\|_{\rm op}=\|B^{\sfT}\|_{\rm op}^2=\|B\|_{\rm op}^2$ for $B \in \mathcal{B}(\calZ, \widetilde{\calZ})$ and using the condition \eqref{eq:stepsize_condition}, we have %, for all $\bmx \in \mathcal{X}$,
\begin{align}
  (\forall \bmx \in \calX) & \ \left\langle \bmx, \left(\frac{\mu^2}{\tau} \mathfrak{L}^{\sfT} {(B^{\sfT}B)}^2 \mathfrak{L} \right) \bmx \right\rangle_{\mathcal{X}}=  \frac{\mu^2}{\tau} \|B^{\sfT} B\mathfrak{L}\bmx\|_{\mathcal{Z}}^2
%    \leq \frac{\mu^2}{\tau} \| B^{\sfT}\|_{\rm op}^2 \|B\mathfrak{L}\bmx\|_{\widetilde{\mathcal{Z}}}^2
    \leq \frac{\mu^2}{\tau} \| B\|_{\rm op}^2 \|B\mathfrak{L}\bmx\|_{\widetilde{\mathcal{Z}}}^2 \nonumber \\
\phantom{ (\forall \bmx \in \calX)} \leq & \  \mu^2 {\left[ \left( \frac{\kappa}{2} + \frac{2}{\kappa} \right) \mu \| B\|_{\rm op}^2\right]}^{-1} \| B\|_{\rm op}^2 \|B\mathfrak{L}\bmx\|_{\widetilde{\mathcal{Z}}}^2
\leq   \mu \frac{\kappa}{2} \|B\mathfrak{L}\bmx\|_{\widetilde{\mathcal{Z}}}^2,
\end{align}
which yields %, for any $\bmx \in \mathcal{X}$,
\begin{align}
(\forall \bmx \in \calX) \ \left\langle \bmx, \left( \frac{\kappa}{2}A^{\sfT}A - \frac{\mu^2}{\tau} \mathfrak{L}^{\sfT} {(B^{\sfT}B)}^2 \mathfrak{L} \right) \bmx \right\rangle_{\mathcal{X}}
%=& \frac{\kappa}{2} \bmx^{\sfT}A^{\sfT}A\bmx - \frac{\mu^2}{\tau} {(B\mathfrak{L}\bmx)}^{\sfT} BB^{\sfT} (B\mathfrak{L}\bmx)  \nonumber \\
%\geq& \frac{\kappa}{2} \bmx^{\sfT}A^{\sfT}A\bmx - \frac{\mu^2}{\tau} \rho(BB^{\sfT}) {(B\mathfrak{L}\bmx)}^{\sfT} (B\mathfrak{L}\bmx)  \nonumber \\
%\geq& \frac{\kappa}{2} \bmx^{\sfT}A^{\sfT}A\bmx - \mu^2 {\left[ \left( \frac{\kappa}{2} + \frac{2}{\kappa} \right) \mu \rho(B^{\sfT}B) \right]}^{-1} \rho(BB^{\sfT}) {(B\mathfrak{L}\bmx)}^{\sfT} (B\mathfrak{L}\bmx)  \nonumber \\
%\geq& \frac{\kappa}{2} \bmx^{\sfT}A^{\sfT}A\bmx -   \mu \frac{\kappa}{2} {(B\mathfrak{L}\bmx)}^{\sfT} (B\mathfrak{L}\bmx)  \nonumber \\
%=& \frac{\kappa}{2} \bmx^{\sfT} \left( A^{\sfT}A - \mu \mathfrak{L}^{\sfT}B^{\sfT}B\mathfrak{L} \right) \bmx + \mu \left[ \frac{\kappa}{2} - {\left( \frac{\kappa}{2} + \frac{2}{\kappa} \right)}^{-1} \right] {(B\mathfrak{L}\bmx)}^{\sfT} (B\mathfrak{L}\bmx)  \nonumber \\
\geq & \frac{\kappa}{2} \langle \bmx, \left( A^{\sfT}A - \mu \mathfrak{L}^{\sfT}B^{\sfT}B\mathfrak{L} \right) \bmx \rangle_{\mathcal{X}} \geq 0, \nonumber
\end{align}
where the last inequality is due to the assumption $A^{\sfT}A - \mu \mathfrak{L}^{\sfT}B^{\sfT}B\mathfrak{L} \succeq \rmO_{\mathcal{X}}$ in Problem~1.

Next, we prove that $\Tlcp$ is $\frac{\kappa}{2\kappa - 1}$-averaged nonexpansive over $(\calR, \langle \cdot, \cdot \rangle_{\mathfrak{P}}, \| \cdot\|_{\mathfrak{P}})$.
%The relation  $P \succ \rmO_{\calR}$ yields that $P$ is invertible.
By applying $\mathfrak{P}\succ \rmO_{\calR}$ to \eqref{eq:Tlcp_PF_PG}, we have for $(\bmx,\bmv,\bmw), (\bmxi,\bmzeta,\bmeta) \in \calR$,
\begin{align}
\Tlcp(\bmx,\bmv,\bmw) = (\bmxi,\bmzeta,\bmeta) 
%&\IFF \mathfrak{P}^{-1}(\mathfrak{P} - F)(\bmx,\bmv,\bmw) \in \mathfrak{P}^{-1}(\mathfrak{P} + G)(\bmxi,\bmzeta,\bmeta)  \nonumber \\
&\IFF (\rmId - \mathfrak{P}^{-1} \circ F)(\bmx,\bmv,\bmw) \in (\rmId + \mathfrak{P}^{-1} \circ G)(\bmxi,\bmzeta,\bmeta). \label{eq:obuqegeonrh} 
\end{align}
Moreover, as will be shown in the end of this proof, $\mathfrak{P}^{-1} \circ G$ is maximally monotone over $(\calR, \langle \cdot, \cdot \rangle_{\mathfrak{P}}, \| \cdot\|_{\mathfrak{P}})$, by which the resolvent $(\rmId + \mathfrak{P}^{-1} \circ G)^{-1}$ is guaranteed to be single-valued and therefore
\begin{equation}
  \label{eq:Tlcpex}
  \Tlcp= (\rmId + \mathfrak{P}^{-1} \circ G)^{-1}\circ (\rmId - \mathfrak{P}^{-1} \circ F),
\end{equation}
where $\frac{1}{2}$-averaged nonexpansiveness of $(\rmId + \mathfrak{P}^{-1} \circ G)^{-1}$ is guaranteed automatically.
% from $(\calR, \langle \cdot, \cdot \rangle_{\mathfrak{P}}, \| \cdot\|_{\mathfrak{P}})$ to $(\calR, \langle \cdot, \cdot \rangle_{\mathfrak{P}}, \| \cdot\|_{\mathfrak{P}})$.
% Thus \eqref{eq:obuqegeonrh} implies

To show that $\Tlcp$ is $\frac{\kappa}{2\kappa-1}$-averaged nonexpansive in $(\Hlb, \langle\cdot, \cdot \rangle_{\Hlb}, \| \cdot\|_{\Hlb})$, Fact~\ref{fact:averaged} tells us the it is sufficient to show the nonexpansiveness of 
% Now, we show that $\rmId - \mathfrak{P}^{-1} \circ F$ is $\frac{1}{\kappa}$-averaged nonexpansive.
% It suffices to prove that
$\rmId - \kappa \mathfrak{P}^{-1} \circ F$ %is nonexpansive
because of
$\rmId - \mathfrak{P}^{-1} \circ F = \left( 1 - \frac{1}{\kappa} \right) \rmId + \frac{1}{\kappa} \left( \rmId - \kappa \mathfrak{P}^{-1} \circ F \right)$.

Define first 
\begin{equation*}
M := \begin{bmatrix}
A^{\sfT}A - \mu \mathfrak{L}^{\sfT}B^{\sfT}B\mathfrak{L} & \rmO_{\mathcal{B}(\mathcal{Z}, \mathcal{X})} & \rmO_{\mathcal{B}(\mathcal{Z}, \mathcal{X})} \\
\rmO_{\mathcal{B}(\mathcal{X}, \mathcal{Z})} & \mu B^{\sfT}B & \rmO_{\mathcal{Z}} \\
\rmO_{\mathcal{B}(\mathcal{X}, \mathcal{Z})} & \rmO_{\mathcal{Z}} & \rmO_{\mathcal{Z}}
\end{bmatrix} \in \mathcal{B}(\calR,\calR),
\end{equation*}
which satisfies
$
F(\bmx,\bmv,\bmw) = M\begin{bmatrix}
\bmx \\
\bmv \\
\bmw
\end{bmatrix} + \begin{bmatrix}
-A^{\sfT}\bmy \\
\bmzero_{\mathcal{Z}} \\
\bmzero_{\mathcal{Z}}
\end{bmatrix}$
for every $(\bmx,\bmv,\bmw) \in \calR$, $M^{\sfT}=M$, and $M\succeq \rmO_{\calR}$ (due to the assumption $A^{\sfT}A - \mu \mathfrak{L}^{\sfT}B^{\sfT}B\mathfrak{L} \succeq \rmO_{\mathcal{X}}$ in Problem 1). Then we have for all ${\mathbf u}_1, {\mathbf u}_2 \in \calR$,
\begin{align}
& \| (\rmId - \kappa \mathfrak{P}^{-1} \circ F)({\mathbf u}_1) - (\rmId - \kappa \mathfrak{P}^{-1} \circ F) ({\mathbf u}_2) \|_{\mathfrak{P}}^2  \nonumber \\
=& \| ({\mathbf u}_1 - {\mathbf u}_2) - \kappa [(\mathfrak{P}^{-1} \circ F) ({\mathbf u}_1) - (\mathfrak{P}^{-1} \circ F)({\mathbf u}_2)] \|_{\mathfrak{P}}^2  \nonumber \\
=& \| {\mathbf u}_1 - {\mathbf u}_2 \|_{\mathfrak{P}}^2 - 2\kappa \langle {\mathbf u}_1 - {\mathbf u}_2, F({\mathbf u}_1) - F({\mathbf u}_2)\rangle_{\mathcal{H}} + \kappa^2 \| (\mathfrak{P}^{-1} \circ F)({\mathbf u}_1) - (\mathfrak{P}^{-1} \circ F)({\mathbf u}_2) \|_{\mathfrak{P}}^2  \nonumber \\
=& \| {\mathbf u}_1 - {\mathbf u}_2 \|_{\mathfrak{P}}^2 - 2\kappa \langle {\mathbf u}_1 - {\mathbf u}_2, M{\mathbf u}_1 - M{\mathbf u}_2\rangle_{\mathcal{H}}  + \kappa^2 \langle \mathfrak{P}^{-1}M{\mathbf u}_1 - \mathfrak{P}^{-1}M{\mathbf u}_2, M{\mathbf u}_1 - M{\mathbf u}_2 \rangle_{\mathcal{H}} \nonumber \\
= & \| {\mathbf u}_1 - {\mathbf u}_2 \|_{\mathfrak{P}}^2 - 2\kappa \left\langle {\mathbf u}_1 - {\mathbf u}_2, \left(M - \frac{\kappa}{2} M \mathfrak{P}^{-1}M\right) ({\mathbf u}_1 - {\mathbf u}_2) \right\rangle_{\Hlb}, \nonumber %\label{eq:PF_nonexp_M}
\end{align}
which implies
\begin{align*}
\text{($\rmId - \kappa \mathfrak{P}^{-1} \circ F$ is nonexpansive)}
\IFF M - \frac{\kappa}{2} M \mathfrak{P}^{-1}M \succeq \rmO_{\calR} 
\IFF \begin{bmatrix}
M & M \\
M & 2\kappa^{-1} \mathfrak{P}
\end{bmatrix} \succeq \rmO_{\calR\times\calR},
\end{align*}
where the last equivalence is due to the Schur complement. Moreover, since for every ${\mathbf u}_1, {\mathbf u}_2 \in \calR$
\begin{align}
  & \left\langle
  \begin{bmatrix}
        {\mathbf u}_1 \\
        {\mathbf u}_2 
      \end{bmatrix},
                        \begin{bmatrix}
M & M \\
M & 2\kappa^{-1} \mathfrak{P}
\end{bmatrix}
      \begin{bmatrix}
        {\mathbf u}_1 \\
        {\mathbf u}_2 
      \end{bmatrix}
  \right\rangle_{\calR\times \calR} \nonumber \\
  & = \langle {\mathbf u}_1, M{\mathbf u}_1\rangle_{\calR}+\langle {\mathbf u}_1, M{\mathbf u}_2\rangle_{\calR}+\langle {\mathbf u}_2, M{\mathbf u}_1\rangle_{\calR} +2\kappa^{-1} \langle {\mathbf u}_2, \mathfrak{P} {\mathbf u}_2\rangle_{\calR} \nonumber \\
 & =  \langle {\mathbf u}_1+{\mathbf u}_2, M({\mathbf u}_1+{\mathbf u}_2)\rangle_{\calR} +2\kappa^{-1} \left\langle {\mathbf u}_2, \left(\mathfrak{P}-\frac{\kappa}{2} M\right) {\mathbf u}_2\right\rangle_{\calR}, \nonumber 
\end{align}
to show the nonexpansiveness of $\rmId - \kappa \mathfrak{P}^{-1} \circ F$, it is sufficient to prove \\
% holds, we have
% \begin{align}[3mm]
%   \relax
%   \rmO_{\calR} \preceq  \mathfrak{P} - \frac{\kappa}{2}M \ \Rightarrow \ 
% \rmO_{\calR\times\calR} \preceq 
% \begin{bmatrix}{cc}
% M & M \\
% M & 2\kappa^{-1} \mathfrak{P}
% \end{bmatrix} [\IFF \text{($\rmId - \kappa \mathfrak{P}^{-1} \circ F$ is nonexpansive)}].
% \nonumber 
% \end{align}
% Since we have for every $\bmtheta_1, \bmtheta_2 \in \calR$
% \begin{align}
% 	\begin{bmatrix}
% 		\bmtheta_1^{\sfT} & \bmtheta_2^{\sfT}
% 	\end{bmatrix}
% 	\begin{bmatrix}
%           2M & M^{\sfT} \\
% 		M & \kappa^{-1} P
% 	\end{bmatrix}
% 	\begin{bmatrix}
% 		\bmtheta_1 \\
% 		\bmtheta_2
% 	\end{bmatrix}
% 	&=
% 	2 \bmtheta_1^{\sfT} M \bmtheta_1 + 2 \bmtheta_2^{\sfT} M\bmtheta_1 + \frac{1}{\kappa} \bmtheta_2^{\sfT} P\bmtheta_2  \nonumber \\
% 	&= 2 \left\| \sqrt{M} \left( \bmtheta_1 + \frac{1}{2}\bmtheta_2 \right) \right\|^2 + \frac{1}{\kappa}\bmtheta_2^{\sfT} \left( P - \frac{\kappa}{2}M \right) \bmtheta_2, \nonumber 
% \end{align}
%Thus, it suffices to prove
$\mathfrak{P} - \frac{\kappa}{2}M \succeq~\rmO_{\calR}$, where
\begin{align}
\relax & \mathfrak{P} - \frac{\kappa}{2}M 
= \begin{bmatrix}
\sigma \rmId & - \mu \mathfrak{L}^{\sfT}B^{\sfT}B & -\mu \mathfrak{L}^{\sfT} \\
-\mu B^{\sfT}B\mathfrak{L} & \tau \rmId & \rmO_{\mathcal{Z}} \\
-\mu \mathfrak{L} & \rmO_{\mathcal{Z}} & \mu \rmId
\end{bmatrix}
- \frac{\kappa}{2} \begin{bmatrix}
A^{\sfT}A - \mu \mathfrak{L}^{\sfT}B^{\sfT}B\mathfrak{L} & \rmO_{\mathcal{B}(\mathcal{Z}, \mathcal{X})} & \rmO_{\mathcal{B}(\mathcal{Z}, \mathcal{X})} \\
\rmO_{\mathcal{B}(\mathcal{X}, \mathcal{Z})} & \mu B^{\sfT}B & \rmO_{\mathcal{Z}} \\
\rmO_{\mathcal{B}(\mathcal{X}, \mathcal{Z})} & \rmO_{\mathcal{Z}} & \rmO_{\mathcal{Z}}
\end{bmatrix}  \nonumber \\
&=\!\! \begin{bmatrix}
\sigma \rmId - (\kappa/2) A^{\sfT}A & \rmO_{\mathcal{B}(\mathcal{Z}, \mathcal{X})} & -\mu \mathfrak{L}^{\sfT} \\
\rmO_{\mathcal{B}(\mathcal{X}, \mathcal{Z})} & \rmO_{\mathcal{Z}} & \rmO_{\mathcal{Z}} \\
-\mu \mathfrak{L} & \rmO_{\mathcal{Z}} & \mu \rmId
\end{bmatrix}\!\! +\!\! \begin{bmatrix}
(\kappa \mu / 2)\mathfrak{L}^{\sfT}B^{\sfT}B\mathfrak{L} & - \mu \mathfrak{L}^{\sfT}B^{\sfT}B & \rmO_{\mathcal{B}(\mathcal{Z}, \mathcal{X})} \\
-\mu B^{\sfT}B\mathfrak{L} & \tau \rmId - (\kappa \mu / 2) B^{\sfT}B & \rmO_{\mathcal{Z}} \\
\rmO_{\mathcal{B}(\mathcal{X}, \mathcal{Z})} & \rmO_{\mathcal{Z}} & \rmO_{\mathcal{Z}}
\end{bmatrix}.
\nonumber
\end{align}
% Let us confirm that the condition \eqref{eq:stepsize_condition} implies $\mathfrak{P} - \frac{\kappa}{2}M \succeq \rmO_{\calR}$. Since
Indeed, by the Schur complement, we have
\begin{align*}
\relax &\begin{bmatrix}
\sigma \rmId - (\kappa/2) A^{\sfT}A & \rmO_{\mathcal{B}(\mathcal{Z}, \mathcal{X})} & -\mu \mathfrak{L}^{\sfT} \\
\rmO_{\mathcal{B}(\mathcal{X}, \mathcal{Z})} & \rmO_{\mathcal{Z}} & \rmO_{\mathcal{Z}} \\
-\mu \mathfrak{L} & \rmO_{\mathcal{Z}} & \mu \rmId
\end{bmatrix} \succeq \rmO_{\calR} 
                              \Leftrightarrow \sigma \rmId - \frac{\kappa}{2} A^{\sfT}A - \mu \mathfrak{L}^{\sfT}\mathfrak{L} \succeq \rmO_{\mathcal{X}} (\Leftarrow \eqref{eq:stepsize_condition})
%  \label{eq:FP_schur1} 
\end{align*}
%where RHS of \eqref{eq:FP_schur1} holds from the condition \eqref{eq:stepsize_condition},
and 
\begin{align}
 \relax  \rmO_{\calR} \preceq &\begin{bmatrix}
(\kappa \mu / 2)\mathfrak{L}^{\sfT}B^{\sfT}B\mathfrak{L} & - \mu \mathfrak{L}^{\sfT}B^{\sfT}B & \rmO_{\mathcal{B}(\mathcal{Z}, \mathcal{X})} \\
-\mu B^{\sfT}B\mathfrak{L} & \tau \rmId - (\kappa \mu / 2) B^{\sfT}B & \rmO_{\mathcal{Z}} \\
\rmO_{\mathcal{B}(\mathcal{X}, \mathcal{Z})} & \rmO_{\mathcal{Z}} & \rmO_{\mathcal{Z}}
\end{bmatrix}
            \nonumber \\
         \IFF
         \ \rmO_{\mathcal{X} \times \mathcal{Z}} \preceq &
\begin{bmatrix}
\mathfrak{L}^{\sfT} & \rmO_{\mathcal{Z}} \\
 \rmO_{\mathcal{Z}}  & \rmId \\
\end{bmatrix}                                                            
\begin{bmatrix}
(\kappa \mu / 2)B^{\sfT}B & - \mu B^{\sfT}B \\
-\mu B^{\sfT}B & \tau \rmId - (\kappa \mu / 2) B^{\sfT}B \\
\end{bmatrix}
\begin{bmatrix}
\mathfrak{L} & \rmO_{\mathcal{Z}} \\
 \rmO_{\mathcal{Z}}  & \rmId \\
\end{bmatrix}   
            \nonumber \\
%  \Leftarrow 
%         & \rmO_{\mathcal{Z}} \preceq & \frac{\kappa \mu}{2} B^{\sfT}B - \mu^{2}B^{\sfT}B \left( (\tau+\epsilon) \rmId - \frac{\kappa \mu}{2} B^{\sfT}B \right)^{-1}B^{\sfT}B \nonumber \\
  \Leftarrow 
         \ \rmO_{\mathcal{Z} \times \mathcal{Z}} \preceq & \begin{bmatrix}
           \frac{\kappa \mu}{2}B^{\sfT}B &  -\mu B^{\sfT}B  \\
           -\mu B^{\sfT}B & \tau \rmId - \frac{\kappa \mu}{2} B^{\sfT}B
\end{bmatrix},
                                          \label{eq:FP_schur2}
\end{align}
implying thus $\text{(RHS of \eqref{eq:FP_schur2})} \Rightarrow \mathfrak{P} - \frac{\kappa}{2}M \succeq \rmO_{\calR}$.
%and 
%it suffices to confirm RHS of \eqref{eq:FP_schur2}.
The RHS of \eqref{eq:FP_schur2} is ensured by \eqref{eq:stepsize_condition}
because for every $v_1, v_2 \in \calZ$
\begin{align}
  & \left\langle
  \begin{bmatrix}
        v_1 \\
        v_2 
      \end{bmatrix},
 \begin{bmatrix}
           \frac{\kappa \mu}{2}B^{\sfT}B &  -\mu B^{\sfT}B  \\
           -\mu B^{\sfT}B & \tau \rmId - \frac{\kappa \mu}{2} B^{\sfT}B
\end{bmatrix}  
      \begin{bmatrix}
        v_1 \\
        v_2 
      \end{bmatrix}
  \right\rangle_{\mathcal{Z}\times \mathcal{Z}} \nonumber \\
  & =  \langle v_1, \frac{\kappa \mu}{2}B^{\sfT}B v_1\rangle_{\mathcal{Z}}-\langle v_1,  \mu B^{\sfT}B v_2\rangle_{\mathcal{Z}}-\langle v_2,  \mu B^{\sfT}Bv_1\rangle_{\mathcal{Z}} + \left\langle v_2, \left[\tau \rmId - \frac{\kappa \mu}{2} B^{\sfT}B\right]
         v_2\right\rangle_{\mathcal{Z}} \nonumber \\
  & =  \frac{2\mu}{\kappa}\left\| \frac{\kappa}{2}Bv_1-Bv_2 \right\|_{\widetilde{\mathcal{Z}}}^2 -\frac{2\mu}{\kappa}\|Bv_2\|_{\widetilde{\mathcal{Z}}}^2
        + \left\langle v_2, \left[\tau \rmId - \frac{\kappa \mu}{2} B^{\sfT}B\right] v_2\right\rangle_{\mathcal{Z}} \nonumber \\
  & =  \frac{2\mu}{\kappa}\left\| \frac{\kappa}{2}Bv_1-Bv_2 \right\|_{\widetilde{\mathcal{Z}}}^2 + \tau\| v_2\|_{\mathcal{Z}}^2
        -\mu\left(\frac{\kappa }{2}+\frac{2}{\kappa}\right)\|Bv_2\|_{\widetilde{\mathcal{Z}}}^2 \nonumber \\
  & \geq   \tau\| v_2\|_{\mathcal{Z}}^2
           -\mu\left(\frac{\kappa }{2}+\frac{2}{\kappa}\right)\|Bv_2\|_{\widetilde{\mathcal{Z}}}^2
           \geq  \left(\tau
           -\mu\left(\frac{\kappa }{2}+\frac{2}{\kappa}\right)\|B\|_{\rm op}^2\right)\| v_2\|_{\mathcal{Z}}^2\geq 0.
\nonumber
\end{align}
%holds, we have \eqref{eq:stepsize_condition} $\FI \ \mu\left(\frac{\kappa }{2}+\frac{2}{\kappa}\right) \|B\|_{\rm op}^2 \leq \tau \ \FI \ \text{(RHS of \eqref{eq:FP_schur2})} [ \Rightarrow \mathfrak{P} - \frac{\kappa}{2}M \succeq \rmO_{\calR}]$.
Therefore we have proved that $\rmId - \mathfrak{P}^{-1} \circ F$ is $\frac{1}{\kappa}$-averaged nonexpansive and that
%$\frac{1}{2}$-averaged nonexpansiveness of $(\rmId + \mathfrak{P}^{-1} \circ G)^{-1}$ that
$\Tlcp= (\rmId + \mathfrak{P}^{-1} \circ G)^{-1}\circ (\rmId - \mathfrak{P}^{-1} \circ F)$ %(see \eqref{eq:Tlcpex})
is $\frac{\kappa}{2\kappa - 1}$-averaged nonexpansive over $(\calR, \langle \cdot, \cdot \rangle_{\mathfrak{P}}, \| \cdot\|_{\mathfrak{P}})$.
%(Note: The choice $(\alpha_1, \alpha_2)=(\frac{1}{2}, \frac{1}{\kappa})$ in Fact~\ref{fact:averaged} yields $\frac{\alpha_1 + \alpha_2 - 2\alpha_1\alpha_2}{1-\alpha_1\alpha_2} = \frac{\frac{1}{2} + \frac{1}{\kappa} - 2\frac{1}{2\kappa}}{1-\frac{1}{2\kappa}}= \frac{\kappa + 2 - 2}{2\kappa-1} =\frac{\kappa}{2\kappa-1}$).

Finally, the maximal monotonicity of $\mathfrak{P}^{-1} \circ G$ over $(\calR, \langle \cdot, \cdot \rangle_{\mathfrak{P}}, \| \cdot\|_{\mathfrak{P}})$ is shown as follows.
Let $G_1 \colon \calR(=\mathcal{X}\times \mathcal{Z} \times \mathcal{Z}) \to 2^{\calR} \colon (x, v, w) \mapsto \{0_{\mathcal{X}}\} \times (\mu \partial \Psi(v)) \times (\mu \partial \Psi^{\ast}(w))$ and $G_2 \colon \calR \to \calR \colon (x,v,w) \mapsto (\mu \mathfrak{L}^{\sfT}B^{\sfT}Bv + \mu \mathfrak{L}^{\sfT}w, -\mu B^{\sfT}B\mathfrak{L}x, -\mu \mathfrak{L}x)$.
Then $G_1$ is maximally monotone over $(\calR, \langle \cdot, \cdot \rangle_{\Hlb}, \| \cdot\|_{\Hlb})$ by \cite[Theorem 20.48, Proposition 16.9 and 20.23]{Bauschke2011-convex}.
Also, $G_2$ is a bounded linear skew-symmetric operator, i.e., $G_2^{\sfT} = -G_2$, and is thus maximally monotone over $(\calR, \langle \cdot, \cdot \rangle_{\Hlb}, \| \cdot\|_{\Hlb})$ by \cite[Example 20.35]{Bauschke2011-convex}.
Then, by $\dom(G_2) = \calR$ and \cite[Corollary 25.5(i)]{Bauschke2011-convex}, $G = G_1+G_2$ is maximally monotone over $(\calR, \langle \cdot, \cdot \rangle_{\Hlb}, \| \cdot\|_{\Hlb})$, which implies the monotonicity of $\mathfrak{P}^{-1} \circ G$ over $(\calR, \langle \cdot, \cdot \rangle_{\mathfrak{P}}, \| \cdot\|_{\mathfrak{P}})$.
Finally, we confirm the maximal monotonicity of $\mathfrak{P}^{-1} \circ G$ over $(\calR, \langle \cdot, \cdot \rangle_{\mathfrak{P}}, \| \cdot\|_{\mathfrak{P}})$ by contradiction.
Assume that there exists $({\mathbf u},{\mathbf z}) \not\in \gra(\mathfrak{P}^{-1} \circ G)$, which means $({\mathbf u}, \mathfrak{P}{\mathbf z}) \not\in \gra(G)$, such that for all $({\mathbf u}', {\mathbf z}') \in \gra(\mathfrak{P}^{-1} \circ G)$, $\langle {\mathbf u} - {\mathbf u}', {\mathbf z} - {\mathbf z}' \rangle_{\mathfrak{P}} = \langle {\mathbf u} - {\mathbf u}', \mathfrak{P} ({\mathbf z} - {\mathbf z}') \rangle_{\Hlb} \geq 0$.
However, since $({\mathbf u}', \mathfrak{P} {\mathbf z}') \in \gra(G)$, it contradicts the maximal monotonicity of $G$ over $(\calR, \langle \cdot, \cdot \rangle_{\Hlb}, \| \cdot\|_{\Hlb})$.

\noindent {\bf \underline{Proof of (c)}}: Thanks to (b), the direct application of Krasnosel'ski{\u{\i}}-Mann iteration in Fact~\ref{fact:KM} to $\Tlcp$ yields (c).
\hfill \qed

\renewcommand{\sfT}{\top} %\mathsf{T}}
\section{Proof of Proposition~\ref{rem:selectionB}} \label{app:prop2}
If $\theta=0$, we have $B_{\theta}=\rmO_l$ and $J_{\Psi_{\rmO_l} \circ \mathfrak{L}}\in \Gamma_0(\mathbb{R}^n)$ for all $\bmy \in \bbR^m$ by Proposition~\ref{prop:exJ_cond}(b).

Let $\theta \in (0, 1]$.
Proposition~\ref{prop:exJ_cond}(b) shows
\begin{align}
(J_{\Psi_{B_{\theta}} \circ \mathfrak{L}} \in \Gamma_0(\mathbb{R}^n) \text{ for all $\bmy \in \bbR^m$)} \Leftarrow \rmO_n  \preceq A^{\sfT}A - \mu \mathfrak{L}^{\sfT}B_{\theta}^{\sfT}B_{\theta}\mathfrak{L}.
\label{eq:eqD1}
\end{align}
The relation $[\rmO_{l \times (n-l)}\ \ \rmI_{l}] \tilde{\mathfrak{L}} = \mathfrak{L}$ and the definition \eqref{eq:tildeAdef} of $\tilde{A}_1$ and $\tilde{A}_2$ show
\begin{align}
\relax &  \rmO_n  \preceq A^{\sfT}A - \mu \mathfrak{L}^{\sfT}B_{\theta}^{\sfT}B_{\theta}\mathfrak{L}= A^{\sfT}A - \mu ([\rmO_{l \times (n-l)}\ \ \rmI_{l}] \tilde{\mathfrak{L}})^{\sfT}B_{\theta}^{\sfT}B_{\theta}([\rmO_{l \times (n-l)}\ \ \rmI_{l}] \tilde{\mathfrak{L}}) \nonumber \\
\IFF \ & \rmO_n  \preceq {(A\tilde{\mathfrak{L}}^{-1})}^{\sfT} {(A\tilde{\mathfrak{L}}^{-1})} - \mu {[\rmO_{l \times (n-l)}\ \ \rmI_{l}]}^{\sfT}B_{\theta}^{\sfT} B_{\theta} [\rmO_{l \times (n-l)}\ \ \rmI_{l}] \nonumber \\
\relax & \qquad = [\tilde{A}_1\ \ \tilde{A}_2]^{\sfT} [\tilde{A}_1\ \ \tilde{A}_2] - \mu {[\rmO_{l \times (n-l)}\ \ \rmI_{l}]}^{\sfT}B_{\theta}^{\sfT} B_{\theta} [\rmO_{l \times (n-l)}\ \ \rmI_{l}] \nonumber \\
\relax & \qquad = \begin{bmatrix}
\tilde{A}_1^{\sfT}\tilde{A}_1 & \tilde{A}_1^{\sfT}\tilde{A}_2 \\
\tilde{A}_2^{\sfT}\tilde{A}_1 & \tilde{A}_2^{\sfT}\tilde{A}_2 - \mu B_{\theta}^{\sfT}B_{\theta}
\end{bmatrix}. \label{eq:iff_lcp_block}
\end{align}
Note that $\tilde{A}_1^{\sfT}\tilde{A}_1\succeq \rmO_{l}$ holds obviously and $\tilde{A}_1^{\sfT}\tilde{A}_1{(\tilde{A}_1^{\sfT}\tilde{A}_1)}^{\dagger}\tilde{A}_1^{\sfT}\tilde{A}_2 = \tilde{A}_1^{\sfT}\tilde{A}_2$ holds due to $
\ran(\tilde{A}_1^{\sfT}\tilde{A}_2) \subset \ran(\tilde{A}_1^{\sfT}) = \ran(\tilde{A}_1^{\sfT} \tilde{A}_1) \subset \Null\left(\tilde{A}_1^{\sfT}\tilde{A}_1{(\tilde{A}_1^{\sfT}\tilde{A}_1)}^{\dagger} - \rmI_{n-l} \right)$.
% Thus \cite[Theorem~1]{albert1969-conditions}\footnotemark[\ref{footnote:albert1969}] implies
Thus \cite[Theorem~1]{albert1969-conditions}\renewcommand{\thefootnote}{\arabic{footnote}}\setcounter{footnote}{9}\footnote{
%   %\label{footnote:albert1969}
 	\cite[Theorem~1]{albert1969-conditions} shows that for a block matrix
 	\begin{equation*}
 	S := \begin{bmatrix}
 	S_{11} & S_{12} \\
 	S_{12}^{\sfT} & S_{22}
 	\end{bmatrix}
 	\end{equation*}
 	with $S_{11}$ and $S_{22}$ symmetric, $S \succeq \rmO$ if and only if $S_{11} \succeq \rmO$, $S_{11}S_{11}^{\dagger}S_{12} = S_{12}$, and $S_{22} - S_{12}^{\sfT}S_{11}^{\dagger}S_{12} \succeq \rmO$.
       }
      implies
\begin{align}\relax
[\text{(LHS of \eqref{eq:eqD1})$\IFF$}]  
\text{(RHS of \eqref{eq:iff_lcp_block})} \IFF \tilde{A}_2^{\sfT}\tilde{A}_2 - \mu B_{\theta}^{\sfT}B_{\theta} - \tilde{A}_2^{\sfT}\tilde{A}_1 {(\tilde{A}_1^{\sfT}\tilde{A}_1)}^{\dagger}\tilde{A}_1^{\sfT}\tilde{A}_2 \succeq \rmO_{l}
\label{eq:B4}
% \nonumber \\
%  &\IFF \tilde{A}_2^{\sfT}\tilde{A}_2 - \mu B_{\theta}^{\sfT}B_{\theta} - \tilde{A}_2^{\sfT}\tilde{A}_1 {(\tilde{A}_1^{\sfT}\tilde{A}_1)}^{\dagger}\tilde{A}_1^{\sfT}\tilde{A}_2 \succeq \rmO_{l}  \nonumber \\  
%&\IFF \left[ \tilde{A}_2^{\sfT}\tilde{A}_2 - \tilde{A}_2^{\sfT}\tilde{A}_1 {(\tilde{A}_1^{\sfT}\tilde{A}_1)}^{\dagger}\tilde{A}_1^{\sfT}\tilde{A}_2 \right]  - \mu B_{\theta}^{\sfT}B_{\theta} \succeq \rmO_{l}. \nonumber 
\end{align}
Since $B_{\theta}$ in \eqref{eq:Btheta} satisfies $\theta^{-1}\mu B_{\theta}^{\sfT}B_{\theta} = \tilde{A}_2^{\sfT}\tilde{A}_2 - \tilde{A}_2^{\sfT}\tilde{A}_1 {(\tilde{A}_1^{\sfT}\tilde{A}_1)}^{\dagger}\tilde{A}_1^{\sfT}\tilde{A}_2$, we have
\begin{align}
\text{(RHS of \eqref{eq:B4})} 
  \IFF \theta^{-1}\mu B_{\theta}^{\sfT}B_{\theta} - \mu B_{\theta}^{\sfT}B_{\theta} \succeq \rmO_{l}.
  %\IFF B_{\theta}^{\sfT}B_{\theta} \succeq \rmO_{l}.%  \nonumber \\
%	&\IFF \left[ \tilde{A}_2^{\sfT}\tilde{A}_2 - \tilde{A}_2^{\sfT}\tilde{A}_1 {(\tilde{A}_1^{\sfT}\tilde{A}_1)}^{\dagger}\tilde{A}_1^{\sfT}\tilde{A}_2 \right]  - \mu B_{\theta}^{\sfT}B_{\theta} \succeq \rmO_{l}. \nonumber
\label{eq:ADlast}
\end{align}
Since RHS of \eqref{eq:ADlast} holds due to
$\theta^{-1}\geq 1$ and
$B_{\theta}^{\sfT}B_{\theta} \succeq \rmO_{l}$, $J_{\Psi_{B_{\theta}} \circ \mathfrak{L}} \in \Gamma_0(\bbR^n)$ has been proven.
\hfill \qed  

%\section*{Bibliography}
\bibliographystyle{plain}
\bibliography{AYY2020_GME0903}

\begin{thebibliography}{10}

\bibitem{abe19:_convex_edge_preser_signal_recov}
J.~Abe, M.~Yamagishi, and I~Yamada.
\newblock Convexity-edge-preserving signal recovery with linearly involved
  generalized minimax concave penalty function.
\newblock In {\em {{Proc. IEEE International Conference}} on {{Acoustics}},
  {{Speech}} and {{Signal Processing}} ({{ICASSP}})}, pages 4918--4922, 2019.

\bibitem{albert1969-conditions}
A.~Arthur.
\newblock Conditions for positive and nonnegative definiteness in terms of
  pseudoinverses.
\newblock {\em SIAM J. Appl. Math.}, 17:434--440, 1969.

\bibitem{Bauschke2011-convex}
H.~H. Bauschke and P.~L. Combettes.
\newblock {\em Convex Analysis and Monotone Operator Theory in Hilbert Spaces}.
\newblock 2nd ed., Springer, 2017.

\bibitem{bayram2015-penalty}
{\.{I}.}~Bayram.
\newblock Penalty functions derived from monotone mappings.
\newblock {\em IEEE Signal Process. Lett.}, 22:265--269, 2015.

\bibitem{bayram2016-convergence}
{\.{I}.}~Bayram.
\newblock On the convergence of the iterative shrinkage/thresholding algorithm
  with a weakly convex penalty.
\newblock {\em IEEE Trans. Signal Process.}, 64:1597--1608, 2016.

\bibitem{Beck-Teboulle10}
A.~Beck and M.~Teboulle.
\newblock Gradient-based algorithms with applications to signal-recovery
  problems.
\newblock In D.~P. Palomar and Y.~C. Eldar, editors, {\em Convex Optimization
  in Signal Processing and Communications}, pages 42--88. Cambidge Univ. Press,
  2010.

\bibitem{Ben-Israel-Greville03}
A.~Ben-Israel and T.~N.~E. Greville.
\newblock {\em Generalized Inverses : Theory and Applications}.
\newblock Springer-Verlag, 2nd ed. edition, 2003.

\bibitem{Bertero-Boccacci98}
M.~Bertero and P.~Boccacci.
\newblock {\em Introduction to inverse problems in imaging}.
\newblock Institute of Physics, Bristol, 1989.

\bibitem{Blake1987-visual}
A.~Blake and A.~Zisserman.
\newblock {\em Visual Reconstruction}.
\newblock MIT Press, 1987.

\bibitem{Boyd-Vandenberghe04}
S.~Boyd and L.~Vandenberghe.
\newblock {\em Convex Optimization}.
\newblock Cambridge Univ. Press, 2004.

\bibitem{byrne2007applied}
C.~L. Byrne.
\newblock {\em Applied Iterative Methods}.
\newblock Ak Peters/CRC Press, 2007.

\bibitem{Candes-Romberg-Tao06}
E.~J. Cand{\`{e}}s, J.~K. Romberg, and T.~Tao.
\newblock Stable signal recovery from incomplete and inaccurate measurements.
\newblock {\em Commun. Pure Appl. Math.}, 59(8):1207--1223, 2006.

\bibitem{carlsson2016-Convexification}
M.~Carlsson.
\newblock On convex envelopes and regularization of non-convex functionals
  without moving global minima.
\newblock {\em Journal of Optimization Theory and Applications}, 183:66--84,
  2016.

\bibitem{chambolle05}
A.~Chambolle.
\newblock An algorithm for total variation minimization and applications.
\newblock {\em J. Math. Imaging Vision}, 20:89--97, 2004.

\bibitem{Claerbout-Muir73}
J.~Claerbout and F.~Muir.
\newblock Robust modelling of erastic data.
\newblock {\em Geophysics}, 38:826--844, 1973.

\bibitem{combettes2008proximal}
P.~L. Combettes and J.-C. Pesquet.
\newblock A proximal decomposition method for solving convex variational
  inverse problems.
\newblock {\em Inverse Problems}, 24(6):065014, 2008.

\bibitem{Combettes-Pesquet2011}
P.~L. Combettes and J.-C. Pesquet.
\newblock Proximal splitting methods in signal processing.
\newblock In H.~H. Bauschke, R.~S. Burachik, P.~L. Combettes, V.~Elser, D.~R.
  Luke, and H.~Wolkowicz, editors, {\em Fixed-Point Algorithms for Inverse
  Problems in Science and Engineering}, pages 185--212. Springer, 2011.

\bibitem{combettes2012primal}
P.~L. Combettes and J.-C. Pesquet.
\newblock Primal-dual splitting algorithm for solving inclusions with mixtures
  of composite, lipschitzian, and parallel-sum type monotone operators.
\newblock {\em Set-Valued and variational analysis}, 20(2):307--330, 2012.

\bibitem{Combettes-Wajs2005}
P.~L. Combettes and V.~R. Wajs.
\newblock Signal recovery by proximal forward-backward splitting.
\newblock {\em Multiscale Model. Simul.}, 4:1168--1200, 2005.

\bibitem{Combettes2015-Compositions}
P.~L. Combettes and I.~Yamada.
\newblock Compositions and convex combinations of averaged nonexpansive
  operators.
\newblock {\em J. Math. Anal. Appl.}, 425:55--70, 2015.

\bibitem{Condat2013-primal_dual}
L.~Condat.
\newblock A primal-dual splitting method for convex optimization involving
  {L}ipschitzian, proximable and linear composite terms.
\newblock {\em J. Optim. Theory Appl.}, 158:460--479, 2013.

\bibitem{Daubechies-Defrise-Mol06}
I.~Daubechies, M.~Defrise, and C.~De Mol.
\newblock An iterative theresholding algorithm for linear inverse problems with
  a sparsity constraint.
\newblock {\em Commun. Pure Appl. Math.}, 57(11):1413--1457, 2004.

\bibitem{ding2015-ArtifactFree}
Y.~Ding and I.~Selesnick.
\newblock Artifact-free wavelet denoising: {N}on-convex sparse regularization,
  convex optimization.
\newblock {\em IEEE Signal Process. Lett.}, 22(9):1364--1368, 2015.

\bibitem{Donoho2006-compressed}
D.~L. Donoho.
\newblock Compressed sensing.
\newblock {\em IEEE Trans. Inform. Theory}, 52:1289--1306, 2006.

\bibitem{du2018minmax}
H.~Du and Y.~Liu.
\newblock Minmax-concave total variation denoising.
\newblock {\em Signal, Image and Video Processing}, 12:1027--1034, 2018.

\bibitem{Ekeland-Temam99}
I.~Ekeland and R.~Temam.
\newblock {\em Convex Analysis and Variational Problems}.
\newblock SIAM, 1999.

\bibitem{Elad2010-sparse}
M.~Elad.
\newblock {\em Sparse and Redundant Representations}.
\newblock Springer, 2010.

\bibitem{fornasier2008-iterative}
M.~Fornasier and H.~Rauhut.
\newblock Iterative thresholding algorithms.
\newblock {\em J. Appl. Comput. Harmon. Anal.}, 25:187--208, 2008.

\bibitem{gandy2011tensor}
S.~Gandy, B.~Recht, and I.~Yamada.
\newblock Tensor completion and low-n-rank tensor recovery via convex
  optimization.
\newblock {\em Inverse Problems}, 27(2):025010, 2011.

\bibitem{Golub-Hansen-OLeary99}
G.~H. Golub, P.~C. Hansen, and D.~P. {O'}Leary.
\newblock Tikhonov regularization and total least squares.
\newblock {\em SIAM J. Matrix Anal. Appl.}, 21(2):185--194, 1999.

\bibitem{Groetsch1972-mann}
C.~W. Groetsch.
\newblock A note on segmenting {M}ann iterates.
\newblock {\em J. Math. Anal. Appl.}, 40:369--372, 1972.

\bibitem{groetsch1993inverse}
C.~W. Groetsch.
\newblock {\em Inverse Problems in the Mathematical Sciences}.
\newblock Springer, 1993.

\bibitem{Hanke-Hansen93}
M.~Hanke and P.~C. Hansen.
\newblock Regularization methods for large-scale problems.
\newblock {\em Survey of Mathematics for Industry}, 3:253--315, 1993.

\bibitem{Hansen93}
P.~C. Hansen.
\newblock The use of {L}-curve in the regularization of discrete ill-posed
  problems.
\newblock {\em SIAM J. Sci. Comput.}, 14:1487--1503, 1993.

\bibitem{Horel62}
A.~E. Horel.
\newblock Application of ridge analysis to regression problems.
\newblock {\em Chem. Eng. Progress}, 58:54--59, 1962.

\bibitem{Horel-Kennard70}
A.~E. Horel and R.~W. Kennard.
\newblock Application of ridge analysis to regression problems.
\newblock {\em Technometrics}, 12:55--67, 1970.

\bibitem{horn2012matrix}
R.~A. Horn and C.~R. Johnson.
\newblock {\em Matrix Analysis}.
\newblock Cambridge Univ. Press, 2012.

\bibitem{kreyszig1978introductory}
E.~Kreyszig.
\newblock {\em Introductory Functional Analysis with Applications}.
\newblock ({W}iley {C}lassics {L}ibrary ed.) {J}ohn {W}iley \& {S}ons, 1989.

\bibitem{lanza2017-Nonconvex}
A.~Lanza, S.~Morigi, I.~Selesnick, and F.~Sgallari.
\newblock Nonconvex nonsmooth optimization via convex-nonconvex
  majorization-minimization.
\newblock {\em Numer. Math.}, 136(2):343--381, 2017.

\bibitem{Lanza2016-convex}
A.~Lanza, S.~Morigi, and F.~Sgallari.
\newblock Convex image denoising via non-convex regularization with parameter
  selection.
\newblock {\em J. Math. Imaging and Vision}, 56:1--26, 2016.

\bibitem{Mohammadi2016-overall_convexity}
M.~Malek-Mohammadi, C.~R. Rojas, and B.~Wahlberg.
\newblock A class of nonconvex penalties preserving overall convexity in
  optimization-based mean filtering.
\newblock {\em IEEE Trans. Signal Process.}, 64:6650--6664, 2016.

\bibitem{Mollenhoff2015-primal}
T.~M\"ollenhoff, E.~Strekalovskiy, M.~Moeller, and D.~Cremers.
\newblock The primal-dual hybrid gradient method for semiconvex splittings.
\newblock {\em SIAM J. Imaging Sci.}, 8:827--857, 2015.

\bibitem{nashed76:_gener_inver_applic}
M.~Z. Nashed, editor.
\newblock {\em Generalized Inverses and Applications}.
\newblock Academic Press, 1976.

\bibitem{nashed02:_inver_probl_image_analy_medic_imagin}
M.~Z. Nashed and O.~Scherzer, editors.
\newblock {\em Inverse Problems, Image Analysis, and Medical Imaging}, volume
  313.
\newblock American Methematical Society, 2002.

\bibitem{nikolova1998-estimation}
M.~Nikolova.
\newblock Estimation of binary images by minimizing convex criteria.
\newblock In {\em Proc. IEEE Int. Conf. Image Process.}, pages 108--112, 1998.

\bibitem{nikolova1999-Markovian}
M.~Nikolova.
\newblock Markovian reconstruction using a {GNC} approach.
\newblock {\em IEEE Trans. Image Process.}, 8(9):1204--1220, 1999.

\bibitem{Nikolova2015-energy}
M.~Nikolova.
\newblock Energy minimization methods.
\newblock In O.~Scherzer, editor, {\em Handbook of Mathematical Methods in
  Imaging}, pages 138--186. Springer, 2011.

\bibitem{ogura02:_non}
N.~Ogura and I.~Yamada.
\newblock Non-strictly convex minimization over the fixed point set of an
  asymptotically shrinking nonexpansive mapping.
\newblock {\em Numer. Funct. Anal. Optim.}, 23(1\&2):113--137, 2002.

\bibitem{Ono2015-hierarchical}
S.~Ono and I.~Yamada.
\newblock Hierarchical convex optimization with primal-dual splitting.
\newblock {\em IEEE Trans. Signal Process.}, 63:373--388, 2015.

\bibitem{pierra76:_method}
G.~Pierra.
\newblock {\em M\'{e}thodes de d\'{e}composition et croisement d'algorithmes
  pour des probl\`{e}mes d'optimisation}.
\newblock PhD thesis, University of Grenoble, 1976.

\bibitem{pierra1984decomposition}
G.~Pierra.
\newblock Decomposition through formalization in a product space.
\newblock {\em Mathematical Programming}, 28(1):96--115, 1984.

\bibitem{Rockafellar70}
R.~T. Rockafellar.
\newblock {\em Convex Analysis}.
\newblock Princeton Univ. Press, 1970.

\bibitem{Rockafellar2009-variational}
R.~T. Rockafellar.
\newblock {\em Variational Analysis}.
\newblock Springer, 2009.

\bibitem{rudin1992nonlinear}
L.~I. Rudin, S.~Osher, and E.~Fatemi.
\newblock Nonlinear total variation based noise removal algorithms.
\newblock {\em Physica D: Nonlinear Phenomena}, 60:259--268, 1992.

\bibitem{Santosa-Symes86}
F.~Santosa and W.~W. Symes.
\newblock Linear inversion of band limited reflection seismograms.
\newblock {\em SIAM J. Sci. Comput.}, 7(4):1307--1330, 1986.

\bibitem{Selesnick2017-sparse}
I.~Selesnick.
\newblock Sparse regularization via convex analysis.
\newblock {\em IEEE Trans. Signal Process.}, 65:4481--4494, 2017.

\bibitem{selesnick2017-Sparsea}
I.~Selesnick and M.~Farshchian.
\newblock Sparse signal approximation via nonseparable regularization.
\newblock {\em IEEE Trans. Signal Process.}, 65(10):2561--2575, 2017.

\bibitem{soubies2017-Continuous}
E.~Soubies, L.~{Blanc-F\'eraud}, and G.~Aubert.
\newblock A continuous exact {L0} penalty {(CEL0)} for least squares
  regularized problem.
\newblock {\em SIAM J. Imaging Sci.}, 8(3):1607--1639, 2015.

\bibitem{Starck-Murtagh-Fadili15}
J.~L. Starck, F.~Murtagh, and J.~Fadili.
\newblock {\em Sparse Image and Signal Processing: Waveletes and Related
  Geometric Multiscale Analysis}.
\newblock Cambidge Univ. Press, 2015.

\bibitem{Taylor-Bank-McCoy79}
H.~Taylor, S.~Bank, and J.~McCoy.
\newblock Deconvolution with the $l_1$-norm.
\newblock {\em Geophysics}, 44:39--52, 1979.

\bibitem{Theodoridis15}
S.~Theodoridis.
\newblock {\em Machine Learning - A Bayesian and Optimization Perspective}.
\newblock Academic Press, 2015.

\bibitem{Tibshirani1996-lasso}
R.~Tibshirani.
\newblock Regression shrinkage and selection via the lasso.
\newblock {\em J. R. Statist. Soc. B}, 58:267--288, 1996.

\bibitem{Tikhonov63}
A.~N. Tikhonov.
\newblock Solution of incorrectly formulated problems and the regularization
  method.
\newblock {\em Soviet. Math. Doctl.}, 4:1035--1038, 1963.

\bibitem{Tikhonov77}
A.~N. Tikhonov and V.~Y. Arsenin.
\newblock {\em Solutions of Ill-Posed Problems}.
\newblock Winston, 1977.

\bibitem{Vu2013-splitting}
B.~C. {V\~{u}}.
\newblock A splitting algorithm for dual monotone inclusions involving
  cocoercive operators.
\newblock {\em Adv. Comput. Math.}, 38:667--681, 2013.

\bibitem{yamada19:_global}
I.~Yamada and M.~Yamagishi.
\newblock Global optimization of sum of convex and nonconvex functions with
  proximal splitting techniques.
\newblock In {\em Proc. the International Conference on Continuous Optimization
  (ICCOPT)}, 2019.

\bibitem{Yamada2017-SVM}
I.~Yamada and M.~Yamagishi.
\newblock Hierarchical convex optimization by the hybrid steepest descent
  method with proximal splitting operators - {E}nhancements of {SVM} and
  {L}asso.
\newblock In {H. H. Bauschke, D. R. Luke, and R. Burachik}, editor, {\em
  Splitting Algorithm, Modern Operator Theory, and Applications}, pages
  413--489. Springer, 2019.

\bibitem{Yamada-Yukawa-Yamagishi2011}
I.~Yamada, M.~Yukawa, and M.~Yamagishi.
\newblock Minimizing the {M}oreau envelope of nonsmooth convex functions over
  the fixed point set of certain quasi-nonexpansive mappings.
\newblock In H.~H. Bauschke, R.~S. Burachik, P.~L. Combettes, V.~Elser, D.~R.
  Luke, and H.~Wolkowicz, editors, {\em Fixed-Point Algorithms for Inverse
  Problems in Science and Engineering}, pages 345--390. Springer, 2011.

\bibitem{yamagishi17:_nonex_lagran}
M.~Yamagishi and I.~Yamada.
\newblock Nonexpansiveness of a linearized augmented {L}agrangian operator for
  hierarchical convex optimization.
\newblock {\em Inverse Problems}, 33(4):35pp., 2017.

\bibitem{yin2019stable}
L.~Yin, A.~Parekh, and I.~Selesnick.
\newblock Stable principal component pursuit via convex analysis.
\newblock {\em IEEE Trans. Signal Process.}, 67(10):2595--2607, 2019.

\bibitem{Zhang2010-mcp}
C.-H. Zhang.
\newblock Nearly unbiased variable selection under minimax concave penalty.
\newblock {\em Ann. Statist.}, 38:894--942, 2010.

\bibitem{Zhong2018-novel}
D.~Zhong, C.~Yi, H.~Xiao, H.~Zhang, and A.~Wu.
\newblock A novel fault diagnosis method for rolling bearing based on improved
  sparse regularization via convex optimization.
\newblock {\em Complexity}, 2018:10, 2018.

\end{thebibliography}
\end{document}